\documentclass[12pt,reqno]{amsart}

\usepackage[margin=1.0in]{geometry}  
\usepackage[colorlinks]{hyperref}
\usepackage{graphicx}              
\usepackage{amsmath}               
\usepackage{amsfonts}              
\usepackage{amsthm}                
\usepackage{amssymb}
\usepackage{appendix}
\usepackage{longtable}
\usepackage{todonotes}  

\newtheorem{thm}{Theorem}[section]

\newtheorem{lem}[thm]{Lemma}
\newtheorem{lemma}[thm]{Lemma}
\newtheorem{prop}[thm]{Proposition}
\newtheorem{cor}[thm]{Corollary}
\newtheorem{conj}[thm]{Conjecture}
\newtheorem{defn}[thm]{Definition}

\numberwithin{equation}{section}
\theoremstyle{remark}

\newcommand{\PSL}{\operatorname{PSL}}
\newcommand{\PGL}{\operatorname{PGL}}
\newcommand{\OK}{\mathcal{O}_K}
\newcommand{\bd}[1]{\mathbf{#1}}  
\newcommand{\RR}{\mathbb{R}}      
\newcommand{\ZZ}{\mathbb{Z}}      
\newcommand{\CC}{\mathbb{C}}

\newcommand{\Aa}{\mathcal{A}}
\newcommand{\mat}[1]{\left(\begin{matrix} #1 \end{matrix} \right)}  
\newcommand{\al}[1]{\begin{align}#1\end{align}}
\newcommand{\aln}[1]{\begin{align*}#1\end{align*}}
\newcommand{\ff}{\mathfrak{f}}

\newcommand{\FF}{\mathfrak{F}}
\newcommand{\rN}{\mathcal{R}_N}
\newcommand{\rn}{\mathcal{R}_N(n)}
\newcommand{\hrn}{\widehat{\mathcal{R}}_N}
\newcommand{\rnu}{\mathcal{R}_N^U}
\newcommand{\hrnu}{\widehat{\mathcal{R}}_N^U}
\newcommand{\tf}{\mathfrak{t}}
\newcommand{\mn}{\mathcal{M}_N}
\newcommand{\en}{\mathcal{E}_N}
\newcommand{\mnn}{\mathcal{M}_N(n)}
\newcommand{\enn}{\mathcal{E}_N(n)}
\newcommand{\enu}{\mathcal{E}_N^U}
\newcommand{\mnu}{\mathcal{M}_N^{U}}
\newcommand{\sk}[1]{\substack{#1}}
\newcommand{\Tf}{\mathfrak{T}}
\newcommand{\mci}{\mathcal{I}}

\newcommand{\II}{\mathcal{I}}

\newcommand{\BB}{L}
\newcommand{\fs}{\mathfrak{S}}
\newcommand{\fd}{\mathfrak{d}}

\newcommand{\QQ}{\mathbb{Q}}
\newcommand{\PP}{\mathbb{P}}

\title[\TITLERUNNING]{\vspace*{-1.3cm} \TITLE}

\date{\today}

\author[Elena Fuchs]{Elena Fuchs}
\address{%
Department of Mathematics,
UC Davis, One Shields Avenue, Davis, CA 95616}
\email{efuchs@math.ucdavis.edu}

\author[Katherine E. Stange]{Katherine E. Stange}
\address{%
Department of Mathematics, University of Colorado,
Campus Box 395, Boulder, Colorado 80309-0395}
\email{kstange@math.colorado.edu}

\author[Xin Zhang]{Xin Zhang}
\address{%
Department of Mathematics, University of Illinois, 
1409 West Green Street,
 Urbana, IL 61801}
\email{xz87@illinois.edu}


\thanks{%
Fuchs has been supported by NSF DMS-1501970, the Sloan Foundation, and the BSF. 
Stange has been supported by NSF EAGER DMS-1643552 and NSF CAREER CNS-1652238.
}

\begin{document}

\title{Local-global principles in circle packings}


\begin{abstract}
  We generalize work of Bourgain-Kontorovich \cite{BK14} and Zhang \cite{Zh14}, proving an almost local-to-global property for the curvatures of certain circle packings, to a large class of Kleinian groups.  Specifically, we associate in a natural way an infinite family of integral packings of circles to any Kleinian group $\mathcal A\leq\textrm{PSL}_2(K)$ satisfying certain conditions, where $K$ is an imaginary quadratic field, and show that the curvatures of the circles in any such packing satisfy an almost local-to-global principle.  A key ingredient in the proof of this is that $\mathcal A$ possesses a spectral gap property, which we prove for any infinite-covolume, geometrically finite, Zariski dense Kleinian group in $\PSL_2(\mathcal{O}_K)$ containing a Zariski dense subgroup of $\PSL_2(\ZZ)$.
\end{abstract}

\maketitle

\section{Introduction}

Local-to-global questions have been studied throughout the history of number theory. 
Here, we consider the set of curvatures appearing in circle packings which are orbits of thin Kleinian groups:  when is the set of curvatures essentially characterised by congruence conditions alone?  In this context, a \emph{thin Kleinian group} is one commensurable to an infinite index subgroup of a Bianchi group $\PSL_2(\mathcal{O}_K)$, but simultaneously Zariski dense in $\PGL_2$.


This question was first considered in 2003 in a groundwork paper by Graham, Lagarias, Mallows, Wilks and Yan \cite{GLMWY03}.  They observed that for several primitive integral Apollonian packings there appears to be a set of congruence classed modulo $24$ or $48$ such that any large enough integer having such a residue is indeed a curvature in that packing.  They conjectured that this is the case for all packings.  In 2011, the first-named author of the present paper made a detailed study of congruence conditions for Apollonian packings \cite{Fu11}.  Together with Sanden, this author performed extensive numerical experiments and conjectured that in fact all primitive integral Apollonian packings can be described in terms of conditions modulo $24$ \cite{FS11}.

The first step towards trying to prove this conjecture is in \cite{GLMWY03}, where it is shown that at least $cx^{1/2}$ integers less than $x$ appear as curvatures in a given integral Apollonian packing, where $c$ is a constant depending on the packing.  Sarnak then made an observation in \cite{Sa07} which became the basis for all future developments on this question.  In that letter, Sarnak showed that in any primitive Apollonian packing there are, up to a constant, at least $\frac{x}{\sqrt{\log x}}$ integers less than $x$ which appear as curvatures in the packing.  His approach was to observe that if one fixes a circle in the packing and considers only those circles tangent to it, their curvatures, without multiplicity, are exactly the set of numbers that are primitively represented by a shifted binary quadratic form $f(x,y)-a$ whose coefficients depend on the circle that is fixed.  Sarnak's idea was then expanded by Bourgain and Fuchs to prove that in fact a positive fraction of all integers appear in any primitive integral Apollonian packing \cite{BF12}.  The methods of \cite{BF12} were then taken several steps further by Bourgain and Kontorovich in \cite{BK14} to prove an asymptotic local-to-global principal for Apollonian packings: they showed that, if $A$ is the set of positive integers that are admissible as curvatures in a given primitive integral Apollonian packing according to their residue modulo $24$, the subset of $A$ of integers which do \emph{not} appear as curvatures in the packing make up a zero density subset of all integers.

How far can one take the method in \cite{BK14} to prove asymptotic local-to-global principles in the thin setting?  For example, the third-named author of this paper successfully used the tools of \cite{BK14} to prove an asymptotic local-to-global principle in so-called \emph{integral Apollonian 3-packings} \cite{Zh14}.  In this paper, we identify the key necessary conditions for these methods to work, which, when satisfied, guarantee an asymptotic local-to-global principle for an integral circle packing or, viewed differently, an orbit of a thin subgroup of $\textrm{PSL}_2(\mathbb C)$.  As a consequence, we immediately have that an asymptotic local-to-global principle holds for the \emph{$K$-Apollonian packings} described by the second-named author \cite{StangeKApp} and for \emph{superintegral polyhedral packings} described by Kontorovich-Nakamura \cite{KN}.  We provide a concrete example of such a packing and give more details on the packings of Stange and Kontorovich-Nakamura in Section~\ref{sec:examples}.  See Figures \ref{fig:cuboct} and \ref{fig:kapp}.

\begin{figure}
  \includegraphics[width=0.8\textwidth]{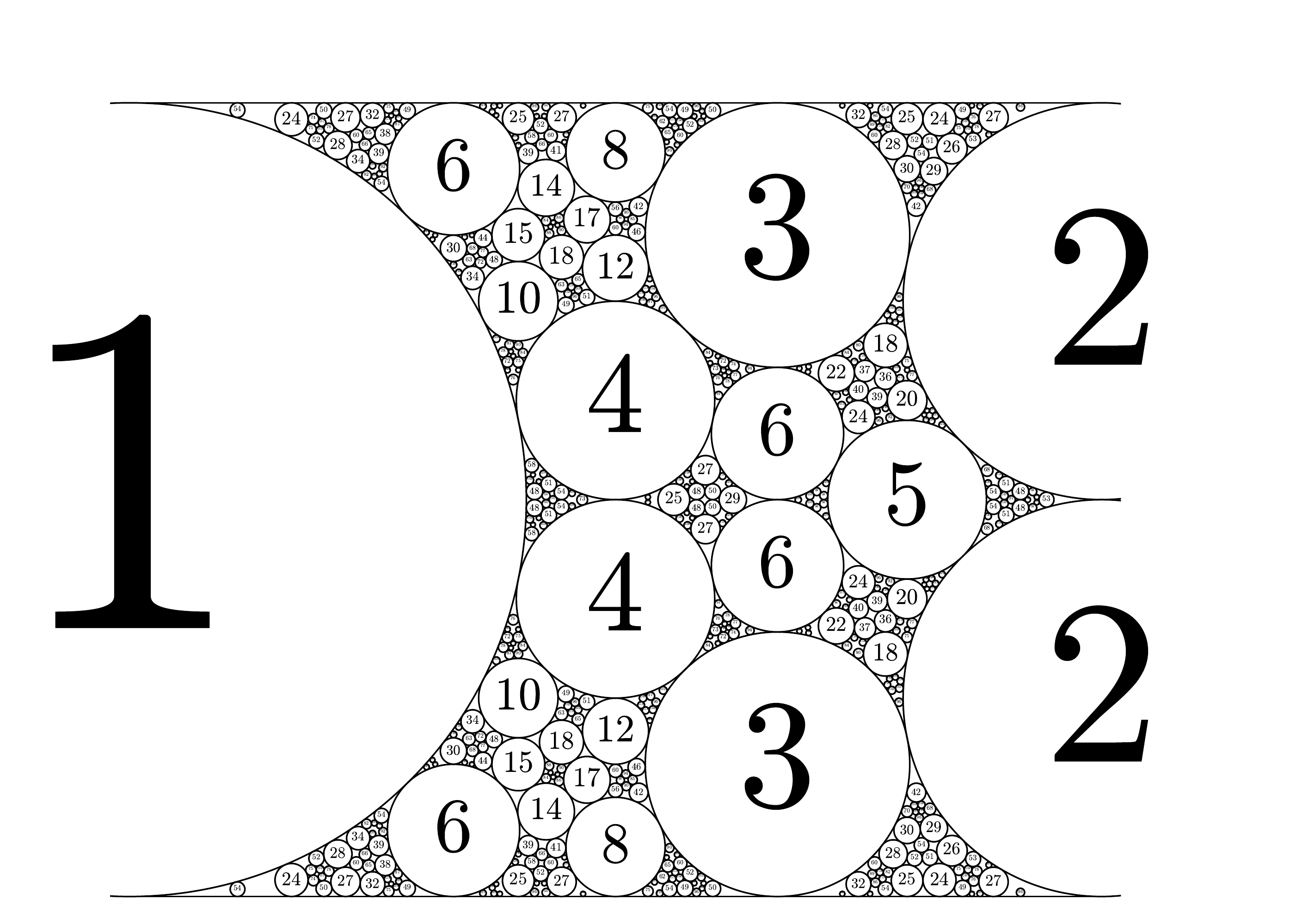}
  \caption{The limit set of an example packing to which Theorem \ref{mainthm} applies (approximation to portion with $0 \le x \le 3$), with curvatures shown (scaled by $3/\sqrt{6}$ to give a primitive integral packing).  See Section \ref{sec:cuboct}.}
  \label{fig:cuboct}
\end{figure}

\begin{figure}
  \includegraphics[width=0.8\textwidth]{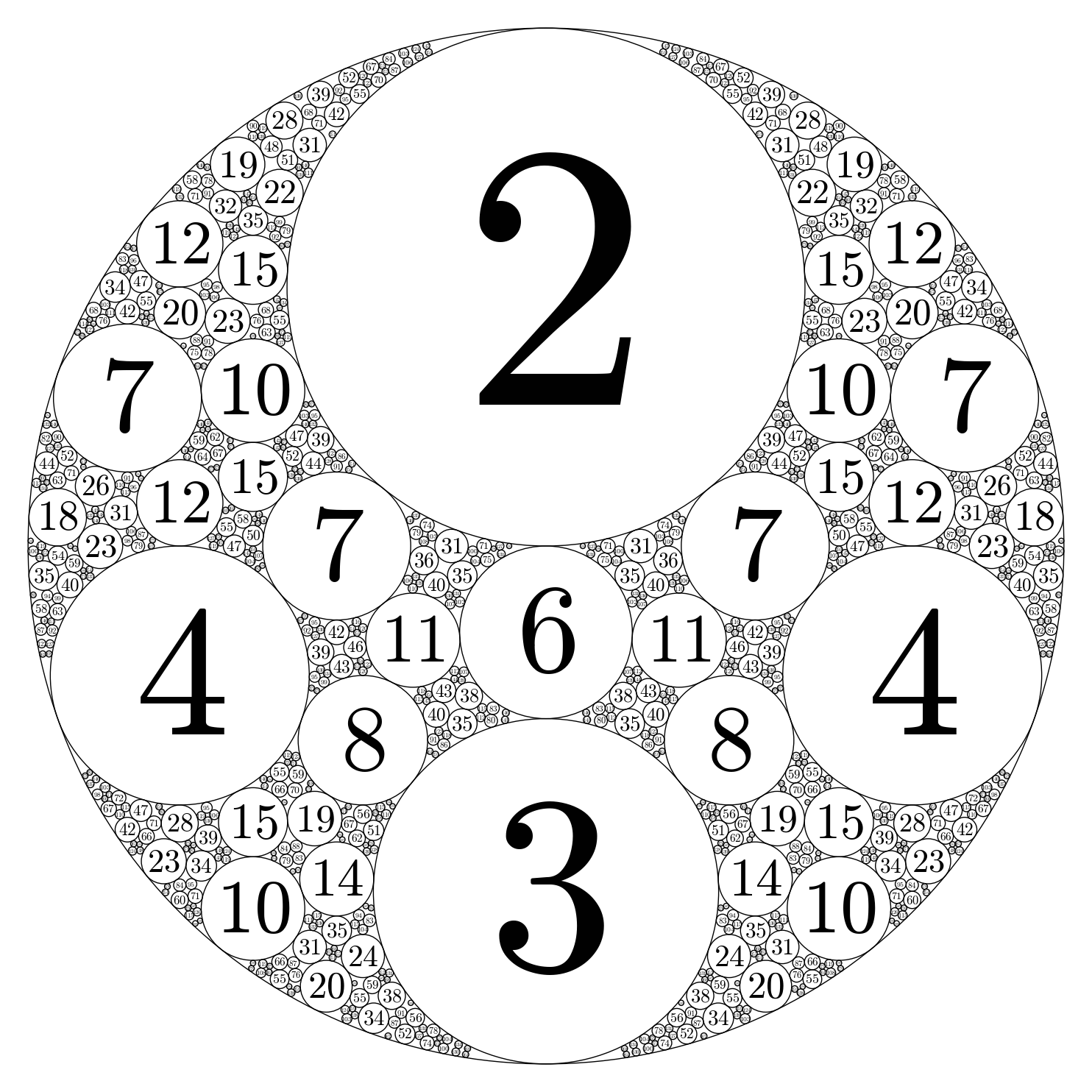}
  \caption{The limit set of an example packing to which Theorem \ref{mainthm} applies, with curvatures shown (scaled by $1/\sqrt{2}$ to give a primitive integral packing).  This is an example of a $K$-Apollonian packing for $K = \mathbb{Q}(\sqrt{-2})$.  See Section \ref{sec:kapp}.}
  \label{fig:kapp}
\end{figure}

In the work on Apollonian packings by Bourgain, Fuchs, Kontorovich, and Zhang, the curvatures in the packings were represented as coordinates of points in an orbit of a thin subgroup of $\textrm{O}_Q(\mathbb Z)$, where $Q$ is a signature $(3,1)$ quadratic form which is simply the Descartes form in the Apollonian case, and an analogue thereof in the 3-packing case.  In both of these cases, one can view the curvatures as curvatures of circles obtained by considering the orbit via M\"obius transformations of a fixed circle (or line) in the complex plane under the action of a thin (Kleinian) subgroup $\mathcal A$ of $\textrm{PSL}_2(\mathcal O_K)$ where $K$ is an imaginary quadratic field.  In the original Apollonian case, $K=\mathbb Q(i)$, and in the 3-packing case $K=\mathbb Q(\sqrt{-2})$.  One can pass between these two interpretations of the set of curvatures via the spin homomorphism $\rho:\textrm{PSL}_2(\mathbb C)\rightarrow \textrm{O}_{\mathbb R}(3,1)$, but the $\textrm{PSL}_2$ setup is more convenient for several reasons: for example, there are numerous choices for the analogue of the Descartes form if one chooses to work in $\textrm{O}_{\mathbb R}(3,1)$; also, $\textrm{SL}_2$ is simply connected, while the orthogonal group is not. 

\begin{defn}\label{KCPdef}
  Let $\mathcal{A}$ be a Kleinian group, and let $C_1, \ldots, C_n$ be circles in the extended complex plane.  Write $\mathcal{A} C_i$ for the orbit of $C_i$ under $\mathcal{A}$, as a subset of the plane (a union of circles).  Then 
  \[
    \bigcup_{i=1}^n \mathcal{A} C_i
  \]
  is called a \emph{Kleinian circle packing}.  Such a packing is called \emph{integral} if, after a universal scaling factor is applied, the set of curvatures can be taken to be a subset of $\ZZ$.
\end{defn}

We define the curvature of a circle $N(\widehat{\RR})$, where $N = \begin{pmatrix} A & B \\ C & D \end{pmatrix}$, to be $2 \Im( \overline{C}D )$; then the radius is $1/|2\Im(\overline{C}D)|$, but the curvature contains some further information in the form of the sign, which can be interpreted as orientation.
In general, the circles in a Kleinian circle packing may overlap, although they do not in the most famous cases, such as the Apollonian circle packing.

Although one might conjecture a local-global principle for a larger class of integral Kleinian circle packings, our methods require that the packing contain `congruence families' of circles, which give rise to integral binary quadratic forms as in the Apollonian case.  Therefore we define a restricted class of groups.

\begin{defn}\label{def:familial}
A Kleinian group $\mathcal{A}$ is called \emph{familial} if:
  \begin{enumerate}
    \item $\PSL_2(\ZZ) \cap \mathcal{A}$ contains a principal congruence subgroup, and
    \item the entries of $\mathcal{A}$ are contained in some fractional ideal $\mathfrak{a}$ of an imaginary quadratic field $K = \QQ(\sqrt{-d})$, $d>0$.
  \end{enumerate}
\end{defn}

Furthermore, the methods require that the group $\mathcal A$ has a spectral gap property: i.e. that the family of graphs $\{\textrm{Cay}(\mathcal A/\mathcal{A}(q), \overline{S})\}_q$ is an expander family.  Here $\mathcal A/\mathcal{A}(q)$ denotes $\mathcal A$ reduced modulo $q$, the set $S$ is a finite generating set of $\mathcal A$, $\overline{S}$ denotes its image under reduction, and $q$ ranges over all positive integers.  In Section~\ref{spectralgapsec}, we show that this is the case for a class of groups including those we intend to consider, i.e., we show the following.

\begin{thm}\label{mainspectralthm}
  Any infinite-covolume, geometrically finite, Zariski dense Kleinian group contained in $\textrm{PSL}_2(\mathcal{O}_K)$, containing a Zariski dense subgroup of $\textrm{PSL}_2(\ZZ)$ has a combinatorial spectral gap.  If, furthermore, the limit set of this Kleinian group has Hausdorff dimension strictly greater than $1$, then it has a geometric spectral gap.
\end{thm}

Salehi-Golsefidy and Zhang, the third named author, generalize this result even further using ideas similar to that of the proof of Theorem~\ref{mainspectralthm} in an upcoming preprint \cite{SGZ17}.  The existence of a  geometric spectral gap is a crucial ingredient both in \cite{BK14} and in \cite{Zh14}, and indeed in almost all works that have investigated arithmetic aspects of thin groups.  We indicate exactly how this spectral gap is relevant in Section \ref{lemmata}.  Note that, as explained below, every familial group $\mathcal A$ does satisfy the Hausdorff dimension hypothesis in Theorem~\ref{mainspectralthm}.

We now state the general setup of the paper.  Let $K$ be an imaginary quadratic field.  Henceforth, we will assume that $\mathcal A$ denotes an infinite-covolume, geometrically finite, Zariski-dense, familial Kleinian group in $\PSL_2(K)$.  We will consider an associated packing $\mathcal{P} := M\mathcal{A}C$, where $C$ is any circle tangent to the real line and having the form $C = N(\widehat{\RR})$, where $N, M \in \PSL_2(K)$.  

This last condition, on the tangency of $C$ to the real line, is crucial to the methods of the paper, as it guarantees, together with the congruence subgroup condition of Definition \ref{def:familial}, that a collection of integral binary quadratic forms govern the curvatures of the packing.  

Under these conditions, the packing $\mathcal{P}$ is necessarily integral as in Definition \ref{def:familial} (see Section \ref{familialsec}).  We let $\mathcal{K} \subset \ZZ$ be the set of curvatures, after some a universal scaling factor is applied as in the definition of integrality.  

Let $\mathcal{K}_a$ be the set of integers passing all the local obstructions by $\mathcal{K}$.  In other words, 
\al{\mathcal{K}_a=\{n\in\ZZ\;\;\vert\;\; \forall q\in\ZZ, \exists k\in\mathcal{K}, \text{such that }n\equiv k(\textrm{mod }q)\}}
We call the integers in $\mathcal{K}_a$ \emph{admissible}.

An immediate corollary of the spectral gap statement in Theorem~\ref{mainspectralthm} is the following.
\begin{cor}\label{local}{There exists a positive integer $L_0$ such that $\mathcal{K}_a$ is the union of some congruence classes mod $L_0$.}
\end{cor}
Of course, this also follows by strong approximation (see \cite{Rapinchuk}) for $\textrm{SL}_2$.  However, the proof of our Theorem~\ref{mainspectralthm} not only gives the existence of $L_0$ but also gives an algorithm to quickly determine its exact value:  in particular, the prime factors of $L_0$ will come from the level of the congruence subgroup contained in $\mathcal A$, any failure of primitivity of the packing $\mathcal{P}$, and the primes $2$ and $3$, as well as the matrix $M$ if $M$ is fractional.  See Theorem \ref{thm:explicit-strong} and \eqref{eqn:L0}, for details.

Now let $\mathcal{K}_a(N)=\mathcal{K}_a\cap [0,N]$ be the set of admissible integers up to $N$, and similarly denote $\mathcal{K}(N)=\mathcal{K}\cap [0,N]$.  Then Corollary \ref{local} directly implies that 
\al{\#\mathcal{K}_a(N)=c_{M,\mathcal{A},C}N+O(1),}
where $c_{M,\mathcal{A},C}$ is the proportion of admissible congruence classes.
We predict that all sufficiently large admissible integers are actually curvatures, or in other words,
\begin{conj}\label{0359}
\al{\#\mathcal{K}(N)=c_{M,\mathcal{A},C}N+O(1).}
\end{conj}
In place of the full conjecture, we prove the following theorem:
\begin{thm}\label{mainthm}
  Let $\mathcal A$ and $\mathcal{P} = M \mathcal{A} C$ be as above.  There exists a positive number $\eta$, depending only on $M$, $\mathcal{A}$ and $C$, such that 
\al{\#\mathcal{K}(N)=c_{M,\mathcal{A},C}N+O(N^{1-\eta})}
 \end{thm}

We feel that it is unlikely that our method can prove Conjecture \ref{0359} without significant new ideas.  

We mention a remark of Chris Leininger: in fact our geometric finiteness assumption can be relaxed to be \emph{€œfinitely generated}.  It is a corollary of the Tameness Theorem \cite{Agol} that any finitely-generated Zariski dense subgroup of the Bianchi group containing a congruence subgroup of $\textrm{PSL}_2(\mathbb Z)$ must be geometrically finite.

Kontorovich and Nakamura define a family of dense circle packings of the plane defined by hyperbolic reflection groups built from uniform polyhedra and their growths \cite{KN}.  For infinitely many of their examples, Kontorovich and Nakamura verify in their paper that such packings satisfy the hypotheses of Theorem \ref{mainthm}, and hence have a local-to-global principle.

Of course, it is possible to construct examples of integral Kleinian packings which \emph{fail} to satisfy the hypotheses of Theorem \ref{mainthm}.  For example, one may take a non-congruence subgroup of $\PSL_2(\ZZ)$, and adjoin another element to obtain a non-Fuchsian group; then consider the orbit of a $K$-rational circle tangent to $\widehat{\RR}$.  In such a case, one cannot guarantee the existence of a suitable family of binary quadratic forms: one only obtains quadratic forms in four related variables.  It is therefore an interesting open question to develop methods which will prove an analogue to Theorem \ref{mainthm} for such packings.

In Section \ref{sec:examples}, in order the demonstrate the variety of examples to which our work applies, we verify that the hypotheses of Theorem \ref{mainthm} hold for the $K$-Apollonian packings of the second-named author \cite{StangeKApp}, and also for an explicit example of a cuboctahedral packing (which also arises in the work of Kontorovich and Nakamura; Figure \ref{fig:cuboct}).

The main method in the proof of this theorem is the Hardy-Littlewood circle method.  In the major arc analysis of the circle method, the main ingredient is an effective counting of group elements for $\mathcal A$ and its congruence subgroups originally achieved by Vinogradov \cite{Vi13}.  In doing this, we require a geometric spectral gap for $\mathcal A$ in order to have a uniform control over the error terms.  In Section~\ref{spectralgapsec} we establish a combinatorial spectral gap for $\mathcal A$, which in turn implies a geometric spectral gap for $\mathcal A$ by the methods in \cite{BGS11}, proving Theorem \ref{mainspectralthm}.  Moreover, we require that $\delta=\delta(\mathcal A)$, the critical exponent of $\mathcal{A}$, which is also the Hausdorff dimension of the limit set of $\mathcal{A}$, is strictly greater than 1, which is guaranteed by our assumption that $\mathcal A$ is familial, and a limit set classification theorem of Bishop-Jones \cite[Corollary 1.8]{Bishop}. \par

Besides the existence of the spectral gap, which is crucial for minor arcs as well as major arcs, the main ingredient in the minor arc analysis is the quadratic form structure, which allows us to do abelian harmonic analysis of two free variables.  Certain Kloosterman-type sums naturally appear here, where we apply standard methods to gain power savings.  In fact the power saving here, as well as in \cite{BK14} and \cite{Zh14}, is so significant that one does not need further restriction on the critical exponent $\delta$ (besides $\delta>1$), in contrast to the works \cite{BK10}, \cite{BK142}, and \cite{Zh16}, which require the critical exponent to be very big in order to get enough cancellation in the minor arc analysis.

Note that our methods, while similar to that in \cite{BK14} and \cite{Zh14}, require several new ingredients and careful generalizations to work.  One crucial such ingredient is the spectral gap of Theorem~\ref{mainspectralthm}.  This theorem applies to a much wider class of groups than our local-to-global analysis, and generalizes the case of the Apollonian group, proven by P. Varj\'u in the appendix of \cite{BK14}.  In proving this theorem, we do not, for instance, have any concrete information about the generators of the group we work with, or exactly at which primes and to what level there are local obstructions for the group.  Indeed, in the proof of Theorem~\ref{mainspectralthm}, we are able to derive, in the case of the groups considered within this paper, exactly what the local obstructions should be: something that was done explicitly for the Apollonian group in \cite{Fu11}.  

Secondly, the fact that we work with an arbitrary imaginary quadratic field $K$ (as opposed to $\mathbb{Q}(i)$ and $\mathbb{Q}(\sqrt{-2})$ as in \cite{BK14} and \cite{Zh14}, respectively), and an abstract subgroup of $\textrm{PSL}_2(\mathcal O_K)$ makes the local analysis in the major arcs section (Section~\ref{sectionmajor}) much less straightforward: where the authors of \cite{BK14} and \cite{Zh14} could depend on concrete local information about the groups they work with, we derive this without relying on explicit information about the local obstructions.  

Thirdly, in both \cite{BK14} and \cite{Zh14}, the level of the congruence subgroup contained in the Apollonian group in question is $2$, which means that the curvatures of the circles are exactly the set of integers represented by a corresponding class of shifted binary quadratic forms.  In our paper this is no longer the case and it is possible that the curvatures we consider (after appropriate scaling to make them integral) comprise a subset of values of the corresponding class of shifted forms.  In fact, while the methods here deal with this nicely, this would make executing the positive density proof in \cite{BF12} significantly more cumbersome in our setting than in the original setup of the classical Apollonian group.

We have made a special effort to make our exposition of these methods particularly accessible, in the hope that it may benefit students and experts alike.

\vspace{0.1in}

\noindent{\bf Notation:} Sections~\ref{integralitysec} through \ref{sectionminor} are notation-heavy.  For ease of reading, we include a table of the major notation used in those sections in Table~\ref{notationtable} of Section~\ref{notations}.  We also note that whenever the constant $\eta$ appears,  it is assumed to satisfy not only the current claim,  but also all claims in  previous contexts.

\vspace{0.1in}

\noindent{\bf Acknowledgements:} We would like to thank Hee Oh for raising the question of how general the methods in \cite{BF12} and \cite{BK14} are, which is what motivated this paper.  We also thank Nathan Dunfield, Alireza Salehi-Golsefidy, Alex Kontorovich, Chris Leininger, Kei Nakamura and Hee Oh for helpful conversations.

\vspace{0.1in}

\noindent{\bf Figures:} Figures were produced with Sage Mathematics Software \cite{sagemath}.

\section{Integrality of $\mathcal{A}$}\label{integralitysec}

For the purpose of our methods, we intend to replace $\mathcal{A}$ with $\tilde{\mathcal A} :={\mathcal{A}}\cap \textrm{PSL}_2(\ZZ[\sqrt{-d}])$, where $d$ is as in Definition \ref{def:familial}, since we would like to work with an integral group.  The next lemma asserts that $\tilde{\mathcal A}$ is finite index in ${\mathcal{A}}$.

Without loss of generality, we can replace $\mathcal{A}$ with any finite-index subgroup for the purposes of Theorem \ref{mainthm}.  This is because 
a finite number of orbits of the subgroup comprise the full orbit of ${\mathcal A}$, and the congruence obstructions from these orbits can be combined to give the obstruction for the union.

For this reason, we are free to assume throughout the paper that $\mathcal{A}$ is torsion-free, by Selberg's theorem, saying that any matrix group contains a finite-index torsion-free subgroup \cite{Se60}, and, by the following lemma, that it is a subgroup of $\textrm{PSL}_2(\ZZ[\sqrt{-d}])$.

\begin{lemma}\label{finiteindexintegral}
  Let ${\mathcal{A}}$ be as defined in the introduction, and let $\tilde{\mathcal A}={\mathcal{A}}\cap \textrm{PSL}_2(\ZZ[\sqrt{-d}])$.  Then $[{\mathcal{A}}:\tilde{\mathcal A}]$ is finite.
\end{lemma}

\begin{proof}Recall that $K=\QQ(\sqrt{-d})$. If ${\mathcal A}\subset\textrm{PSL}_2(\ZZ[\sqrt{-d}])$, then the statement is trivial.  Hence, suppose ${\mathcal A}\not\subset\textrm{PSL}_2(\ZZ[\sqrt{-d}])$, such that the denominators featured in its elements are bounded above, as assumed in the previous section.  Let $q=p_1^{e_1}\cdots p_k^{e_k}$, where $p_1,\dots,p_k$ are distinct primes, be the least common multiple of all denominators featured among entries of elements of ${\mathcal A}$.   

  Let $H_1=\textrm{PSL}_2(\frac{1}{q} \mathbb Z[\sqrt{-d}])$, let $H_2={\mathcal A}$, and let $H_3= \textrm{PSL}_2(\mathbb Z[\sqrt{-d}])$.  Note that $H_1$ is not a group, but contains both $H_2$ and $H_3$.  Furthermore, it is covered by some union of cosets of $H_3$ in $\textrm{PSL}_2(K)$.
If $H_1$ is covered by a finite union of cosets of $H_3$, then $H_2 = H_2 \cap H_1$ is covered by a finite union of cosets of $H_2 \cap H_3$, i.e. $[\mathcal{A}:\tilde{\mathcal{A}}]$ is finite.  

Therefore, we will cover $H_1$ by a finite union of cosets of $H_3$.  To show this, note that if the $p_i$-adic expansions of $\gamma_1, \gamma_2 \in \textrm{PSL}_2(\frac{1}{q} \mathbb Z[\sqrt{-d}])$ agree in the $p_i^{-e_i}, p_i^{-e_i+1},\dots,p_i^{e_i}$ terms for all $1\leq i\leq k$, then the ``coefficients" of the entries of $\gamma_1\gamma_2^{-1}$ are $p_i$-adic integers for all $i$.  Here, what we mean by $p_i$-adic expansions of $\gamma$ is what one gets when one considers for each entry of $\gamma$ of the form $a+b\sqrt{-d}$ the $p_i$-adic expansion of $a$ and $b$.  By ``coefficients" of an entry $a+b\sqrt{-d}$ of $\gamma$ we mean precisely $a$ and $b$.  Since $\gamma_1\gamma_2^{-1}\in  \textrm{PSL}_2(\frac{1}{q^2} \mathbb Z[\sqrt{-d}])$, this in fact implies that $\gamma_1\gamma_2^{-1}\in\textrm{PSL}_2(\mathbb Z[\sqrt{-d}])$.  Since there are only finitely many possibilities for the $p_i^{-e_i}, p_i^{-e_i+1},\dots,p_i^{e_i}$ terms in the $p_i$-adic expansion of any number, where $i$ ranges over finitely many indices, we have that there are in fact finitely many cosets of $\textrm{PSL}_2(\mathbb Z[\sqrt{-d}])$ in $\textrm{PSL}_2(\frac{1}{q} \mathbb Z[\sqrt{-d}])$, as desired.
\end{proof}

We remark that a converse also holds:  if $\mathcal{A}$ has its intersection with the Bianchi group as a subgroup of finite index, then $\mathcal{A}$ has bounded denominators.

Therefore, from this point on we assume $\mathcal{A}$ is a torsion-free subgroup of $\PSL_2(\mathbb Z[\sqrt{-d}])$.  

\section{Families of quadratic forms}\label{familialsec}

We now describe the set of curvatures $\mathcal{K}$ as a union of values of a family of quadratic forms.
Write $\Delta$ for the discriminant of $\mathcal O_K$. If $d\equiv 1,2$(mod 4), then $\Delta=-4d$, and if $d\equiv 3$(mod 4), then $\Delta=-d$.  
Letting $\gamma=\mat{A_{\gamma}&B_{\gamma}\\C_{\gamma}&D_{\gamma}}\in \PSL_2(\CC)$, direct computation shows that $\gamma$ sends the horizontal line $\widehat{\RR}$ to a circle of curvature 
\al{\kappa(\gamma(\widehat{\RR}))=2 \Im (\overline{C_{\gamma}}D_{\gamma}) \in \RR.}
If $\gamma \in \PSL_2(\mathcal{O}_K)$, then $\kappa(\gamma(\widehat{\RR})) \in \sqrt{-\Delta}\ZZ$.

We may assume without loss of generality that $N(\widehat{\RR}) = \widehat{\RR}+\sqrt{\Delta}/2$.  For, $\PSL_2(\QQ)$ is transitive on circles of $\PSL_2(K) \widehat{\RR}$ tangent to $\widehat{\RR}$.  Therefore we may choose $N_0$ satisfying $N_0(\widehat{\RR}) = \widehat{\RR} + \sqrt{\Delta}/2$, and $NN_0^{-1} \in \PSL_2(\QQ)$.  Then we have 
\[
  M \mathcal{A} N = (M N N_0^{-1}) (N_0 N^{-1} \mathcal{A} N N_0^{-1}) N_0.
\]
But $M' = MNN_0^{-1} \in PSL_2(K)$, $N_0 \in PSL_2(\mathcal{O}_K)$, and $\mathcal{A}' = N_0N^{-1} \mathcal{A} N N_0^{-1}$ is again Zariski dense, infinite covolume, geometrically finite and familial.  Therefore let us assume $N(\widehat{\RR}) = \widehat{\RR} + \sqrt{\Delta}/2$.  By Lemma~\ref{finiteindexintegral}, we may again pass to a finite index subgroup $\mathcal A$ of $\mathcal A'$ and work with this group in order to prove Theorem~\ref{mainthm}.

With this choice of $N$, for any $\gamma \in \mathcal{A}$, and $M$ as above, the curvatures of the orbit $M \gamma \PSL_2(\ZZ) ( \widehat{\RR} + \frac{\sqrt{\Delta}}{2} )$ are given by the shifted quadratic form
\begin{equation}
  \label{0746} 
  \mathfrak{\widehat{f}}_{M\gamma}(a,c) = \sqrt{-\Delta}\left\vert C_{M \cdot\gamma} a+D_{M\gamma} c\right\vert^2+2\Im (\overline{C_{M\gamma} }D_{M\gamma})
\end{equation}
in terms of the entries $a$ and $c$ of $\begin{pmatrix} a & b \\ c & d \end{pmatrix} \in \PSL_2(\ZZ)$.  Therefore the packing $M \mathcal{A}( \widehat{\RR} + \frac{\sqrt{\Delta}}{2} )$ contains the curvatures of  
\[
  \left\{ \mathfrak{\widehat{f}}_{M\gamma}(Lx+1, Ly) : \gcd(x,y)=1 \right\},
\]
where $L$ is the level of the congruence subgroup contained in ${\mathcal{A}}$.  Write
\[
  \mathfrak{f}_{M\gamma}(a,c) = \frac{1}{\sqrt{-\Delta}} \mathfrak{\widehat{f}}_{M\gamma}(a,c).
\]
Then $\mathfrak{f}_{M\gamma}(a,c)$ is a shifted binary rational quadratic form, i.e.
\begin{equation}\label{1112}
  \mathfrak{f}_{M\gamma}(a,c) 
  = 
  \widetilde{\mathfrak{f}}_{M\gamma}(a,c) + \mathfrak{d}_\gamma
\end{equation}
where
\[
  \widetilde{\mathfrak{f}}_{M\gamma}(a,c) = \left| C_{M\gamma} a + D_{M\gamma} c \right|^2, \quad \mbox{and} \quad
  \mathfrak{d}_\gamma = 2\frac{\Im (\overline{C_{M\gamma} }D_{M\gamma})}{\sqrt{-\Delta}}.
\]
In particular, $\widetilde{f}_{M\gamma}$ has discriminant $\Delta\mathfrak{d}_\gamma^2 = -4\left(\Im(\overline{C_{M\gamma}} D_{M\gamma} )\right)^2 < 0$.\par

Unlike in the Apollonian case, it is possible that not all of these forms are primitive integral binary quadratic forms.  However, their deviation from such forms, which is a function of the denominators introduced by $M$, is uniformly bounded.  

\begin{lemma}\label{lem:primitive}
  Let $M \in \PSL_2(K)$ and let $d_1$ be such that $d_1 M \in \PGL_2(\mathcal{O}_K)$.  Up to multiplying and/or dividing by integers dividing $d_1^4$, the form $\widetilde{f}_{M\gamma}$ becomes a primitive integral binary quadratic form.
\end{lemma}

\begin{proof}
  We have that $M\gamma \in \PSL_2(K)$.  In particular, we have
  \begin{equation}
    \label{eqn:det}
    A_{M\gamma} D_{M\gamma}  - B_{M\gamma} C_{M\gamma} = 1.
  \end{equation}
  By assumption, $C_{M\gamma}, D_{M\gamma} \in \frac{1}{d_1}\mathcal{O}_K$.  Write
  \[
    C_{M\gamma} = \frac{C'_{M\gamma}}{d_1}, \quad
    D_{M\gamma} = \frac{D'_{M\gamma}}{d_1}.
  \]
  Where $C'_{M\gamma}, D'_{M\gamma} \in \mathcal{O}_K$.  In particular, the ideal generated by $C'_{M\gamma}$ and $D'_{M\gamma}$ has norm at most $d_1^4$ by \eqref{eqn:det}.  

  For any $C, D \in \mathcal{O}_K$, if the integral form
  \[
    |Cx + Dy|^2 = C\overline{C} x^2 + (C\overline{D} + \overline{C}D)xy + D\overline{D}y^2
  \]
  is imprimitive by a factor of, say, $e$ dividing all its coefficients, then $e \mid N(C, D)$ (the norm of the ideal).  To see this, suppose $p$ is prime and $p^k \mid |Cx + Dy|^2$ for all $(x,y)$.  If $p$ is inert, then this implies $(C,D) \subset p^{\lceil k/2 \rceil}\mathcal{O}_K$, so $p^k \mid N(C,D)$.  If $p = \mathfrak{p}\overline{\mathfrak{p}}$ is split, then $C$, $D$ and $C+D$ are each contained in some ideal $\mathfrak{p}^s\overline{\mathfrak{p}}^t$, where $s+t = k$.  Call these three pairs $(s,t) = (s_1,t_1), (s_2,t_2), (s_3,t_3)$, ordered so that $s_1 < s_2 < s_3$.  Then 
  \[
    (C,D) \subset \mathfrak{p}^{s_2}, \overline{\mathfrak{p}}^{t_2},
  \]
  since any two of $C,D,C+D$ generate $(C,D)$.
  Hence, $p^k \mid N(C,D)$.
  
  Therefore $|C_{M\gamma}'x + D_{M\gamma}'y |^2$ is imprimitive by a factor dividing $d_1^4$.

  This shows that the integrality and/or primitivity of $\widetilde{f}_{M\gamma}$ is achieved by multiplication and/or division by a factor of at most $d_1^4$, where $d_1$ is independent of $\gamma$.
 \end{proof}

Finally, the integral curvatures we seek to study are given by the union of the integers represented by these shifted forms, i.e.
\begin{equation}\label{quadformrep}
  \mathcal{K}_{M {\mathcal{A}} (\widehat{\RR} + \frac{\sqrt{\Delta}}{2})} = \bigcup_{\gamma \in {\mathcal{A}}} \left\{ \mathfrak{f}_{M\gamma}(Lx+1,Ly) : \gcd(x,y) = 1 \right\}.
\end{equation}
\section{setup of the circle method}
\label{sec:setup}

Throughout the circle method, there are the following growing parameters:
\[
 T,X,N,T_1,T_2,J,Q_0,K_0, U, H.
 \]
 Their precise relationships, used to tune the result, are boxed throughout the paper, and these are: \eqref{eqn:growing-par}, \eqref{eqn:growing-par2}, \eqref{eqn:setJ}, \eqref{0737} and \eqref{1145}.  We collect these equations here for reference:
 \begin{equation}
   \label{eqn:allgrowing}
   \boxed{
   \begin{aligned}[c]
   N &= T^2X^2, \notag \\
   T &= N^{1/200}, \notag \\
   T &=T_1T_2, \notag \\
   T &_2 =T_1^{\nu}, \notag\\
    J &= T^2X, \notag \\
  \end{aligned}
  \hspace{2em}
  \begin{aligned}[c]
  Q_0&=T^{\frac{2\delta-2\Theta}{80}} \notag \\
  K_0&=Q_0^3, \notag \\
H&=Q_0^{\frac{\eta_0}{4}},  \\
U&=H^{\frac{\eta_0}{20}}. 
\end{aligned}
}
 \end{equation}

Each element $\gamma\in\mathcal{A}$ corresponds to a shifted quadratic form of two variables that represents curvatures of circles tangent to $M\gamma(\mathbb R+\frac{\sqrt{\Delta}}{2})$, given in \eqref{1112}.
Note that $M$ is fixed throughout the paper.  

Our goal is to show that almost all admissible integers are represented by some such shifted form.  To do this, we consider this problem applied to growing subsets of $\mathcal A$, and the shifted binary forms corresponding to the elements in these subsets.  

We now define these growing subsets.  We choose three growing parameters $N$, $T$, and $X$ such that
\begin{equation}
  \label{eqn:growing-par}
  \boxed{  N=T^2X^2,\quad T=N^{\frac{1}{200}}.  }
\end{equation}
Since $T$ is a small power of $N$, we have that $X$ is almost of the scale of $N^{\frac{1}{2}}$.  We further write
\begin{equation}
  \label{eqn:growing-par2}
  \boxed{T=T_1T_2, T_2=T_1^{\nu}}
\end{equation}
where $\nu>0$ is a large number depending only on the spectral gap of $\mathcal{A}$, and we define the following set (counting with multiplicity):

\al{\label{FF}
\FF=\FF_T=\left\{\gamma=\gamma_1\gamma_2:\begin{array}{ccc}\gamma_1,\gamma_2\in\Aa\\ T_1/2\leq \Vert M\gamma_1\Vert  \leq T_1 \\ T_2/2\leq \Vert \gamma_2\Vert  \leq T_2 \\ \Im(\overline{C_{M\gamma}} D_{M\gamma})\geq T/100 \end{array}\right\}
}  
Here $\Vert \cdot\Vert $ stands for the Frobenius norm.

The reason that we define $\FF_T$ using two parameters $T_1$ and $T_2$ is that this is the necessary setup for Lemma~\ref{bk1} which is a result of Bourgain-Kontorovich from \cite{BK14}, and one that we will be using in the minor arcs analysis in this paper.  Lemma~\ref{bk2} and Lemma~\ref{bk3} are also stated within this setup, although for these two results one can set up the problem with just a growing ball of radius $T$.

Finally, we wish to let two integers $x$ and $y$ range over two sets of integers that are $\asymp X$.  
For this reason we introduce a smooth function $\psi$ supported on $[1,2]$, such that $\psi\geq 0$ and $\int_\RR \psi (x)dx=1$.
If $(Lx+1,Ly)=1$, then $\ff_{M\cdot \gamma} (Lx+1,Ly)$ is a curvature.   

We then define 
\al{ \label{RNdef}
\mathcal{R}_N(n)= \sum_{\gamma\in \FF_T}\sum_{\substack{x,y\in\ZZ \\(Lx+1,Ly)=1} }\psi\left(\frac{Lx+1}{X}\right)\psi\left(\frac{Ly}{X}\right) \bd{1}\{\ff_{M\gamma}(Lx+1,Ly)=n \} }
If $\rn>0$ then $n$ is a curvature.  Our goal is to show that $\rn>0$ for almost all $n$, in the sense described in Theorem~\ref{mainthm}.

From Theorem~2.2 in \cite{Vi13}, the size of $\FF_T$ is $\asymp T^{2\delta}$ (recall that $\delta$ is the critical exponent of $\mathcal{A}$).  Given the definition of $\FF_T$, the function $\mathcal{R}_N$ is supported on $n \asymp N$.  We obtain
$$\Vert \rN\Vert _{l_1}= \sum_{n\asymp N}\rn  \asymp T^{2\delta}X^2.$$
It is expected that this is roughly equidistributed on the set of admissible integers, so that
\al{\label{0848}\rn \gg \frac{T^{2\delta}X^2}{N}=T^{2\delta-2}}
for every admissible $n$.

It would be ideal to show \eqref{0848}.  However, current technology does not enable us to prove this.  Instead, we will show that $\rn \gg T^{2\delta-2}$ for every admissible integer in $[N/2,N]$ outside of an exceptional set of size $O(N^{1-\eta})$ for some $\eta>0$.   

Notice that in the definition of $\mathcal R_N$ in \eqref{RNdef}, the second sum is over pairs of integers $(Lx+1,Ly)$ which satisfy a coprimality condition that is hard to track directly in computations.  We hence rewrite this sum as one over all pairs $(Lx+1,Ly)$ using M\"obius orthogonality:
\al{\sum_{d|n}\mu(d) =
\begin{cases}
1 & \text{if } n=1,\\
0 & \text{if } n>1.
\end{cases}
}
Then
\al{\nonumber \mathcal{R}_N(n)&= \sum_{\gamma\in \FF_T}\sum_{x,y\in\ZZ}\sum_{u|(Lx+1,Ly)}\mu(u)\psi\left(\frac{Lx+1}{X}\right)\psi\left(\frac{Ly}{X}\right) \bd{1}\{\ff_{M\gamma}(Lx+1,Ly)=n \}.}
Notice that if $u|(Lx+1,Ly)$, then $(u,\BB)=1$, so $y\equiv 0 \pmod{u}$ and $\BB x\equiv -1 \pmod{u}$.  Let $u^*$ be the integer from $[1,\BB-1]$ such that $uu^{*}\equiv 1 \pmod{\BB}$.  Then we can write 
\al{\mathcal{R}_N(n)=\sum_{(u,\BB)=1}\mu(u)\sum_{\gamma\in\FF_T}\sum_{x,y\in\ZZ} \psi\left(\frac{\BB ux+uu^{*}}{X}\right)\psi\left(\frac{\BB uy}{X}\right)\bd{1}\{\ff_{M\gamma}(\BB ux+uu^{*},{\BB uy})=n \}.\ \label{0850}}
With this manipulation, the innermost sum becomes one over free variables $x,y$, allowing us to use abelian harmonic analysis to analyze it.  \\

To facilitate our analysis we will study a relative of $\rN$ which we denote by $\rnu$, where $U$ is a small power of $N$, and determined at \eqref{1145}.   We restrict the $u$-sum in \eqref{0850} to $u<U$ and define 
\al{\label{0922}\rnu(n)=\sum_{\substack{u<U\\(u,\BB)=1}}\mu(u)\sum_{\gamma\in\FF_T}\sum_{x,y\in\ZZ} \psi\left(\frac{\BB ux +uu^{*}}{X}\right)\psi\left(\frac{\BB uy}{X}\right)\bd{1}\{\ff_{M\gamma}(\BB ux+uu^{*},\BB uy)=n \}}

The following lemma shows that the difference between $\rN$ and $\rnu$ is small in $l_1$: 
\begin{lem}\label{1130}
  $$\vert\vert \rN-\rnu\vert\vert_{l_1} \ll \frac{T^{2\delta}X^2}{U}.$$
\end{lem}                
\begin{proof}
From \eqref{0850} and \eqref{0922},
\al{\nonumber&\sum_{n\in\ZZ}\vert \rN(n)-\rnu(n)\vert \\
\nonumber&\leq \sum_{n\in\ZZ}\left\vert \sum_{u\geq U}\mu(u)\sum_{\gamma\in\FF_T}\sum_{x,y\in\ZZ} \psi\left(\frac{\BB ux+uu^{*}}{X}\right)\psi\left(\frac{\BB uy}{X}\right)\bd{1}\{\ff_{M\gamma}(\BB ux+uu^{*},{\BB uy})=n \} \right\vert\\
\nonumber&\leq \sum_{n\in\ZZ}\sum_{u\geq U}\sum_{\gamma\in\FF_T}\sum_{x,y\in\ZZ} \psi\left(\frac{\BB ux +uu^{*}}{X}\right)\psi\left(\frac{\BB uy}{X}\right)\bd{1}\{\ff_{M\gamma}(\BB ux+uu^{*},\BB uy)=n \}\\
\nonumber&\leq\sum_{u\geq U}\sum_{\gamma\in\FF_T}\sum_{x,y\in\ZZ} \psi\left(\frac{\BB ux +uu^{*}}{X}\right)\psi\left(\frac{\BB uy}{X}\right)\\
\nonumber&\ll \sum_{u\geq U}T^{2\delta}\frac{X^2}{u^2} \ll \frac{T^{2\delta}X^2}{U}.
}
\end{proof}

We will study $\rN$ (and $\rnu$) via its Fourier transform: 
\al{\label{1248A} \hrn(\theta)= \sum_{\gamma\in \FF_T}\sum_{\substack{x,y\in\ZZ \\(Lx+1,Ly)=1} }\psi\left(\frac{Lx+1}{X}\right)\psi\left(\frac{Ly}{X}\right) e\left(\ff_{M\gamma}(Lx+1,Ly)\theta\right) 
}
using the fact that we can recover $\rN$ from $\hrn$ via the Fourier inversion formula:
\al{ \label{0537}\rN(n)=\int_0^1\hrn(\theta)e(-n\theta)d\theta.
}
It is in evaluating this integral in  \eqref{0537} that the circle method will be applied.  

By Dirichlet's approximation theorem, given any positive integer $J$, for every real number $\theta\in[0,1)$, there exist integers $r,q$ such that $1\leq q \leq J$ and $\left\vert\theta-\frac{r}{q}\right\vert <\frac{1}{q\cdot J}$.  The integer $J$ is called the depth of approximation, and we will take
  \begin{equation}
    \label{eqn:setJ}
    \boxed{J = T^2X.}
  \end{equation}  
  The general philosophy of the circle method is that most of the contribution to the integral \eqref{0537} should come from neighborhoods of rationals with small denominator.  Such neighborhoods are called \emph{major arcs}.  One shows that \eqref{0537} is bounded below, by bounding the major arcs below, and then bounding the minor arcs, considered an error term, above. 
 
  In our case the major arcs are comprised of $\theta\in[0,1]$ such that $|\theta-\frac{r}{q}|\leq\frac{K_0}{N}$, where $q\leq Q_0$.  Here  $Q_0$ and $K_0$ are small powers of $N$ which depend on the spectral gap and are given in \eqref{0737}.  Write $\beta = \theta - \frac{r}{q}$.

To define what we call the major arc contribution, we first introduce the hat function
\aln{\tf(x):= \max\{ 0, 1-|x| \},}  
whose Fourier transform is
\aln{\hat{\tf}(y)=\left(\frac{\text{sin}(\pi y)}{\pi y}\right)^2.}
In particular, $\hat{\tf}$ is nonnegative (we take $\hat{\tf}(0)=1$).

From $\tf$, we construct a spike function $\mathfrak{T}$, with period $1$ on $\RR$, to capture the major arcs: 
\al{\label{spike}\mathfrak{T}(\theta):=\sum_{q< Q_0}\sideset{}'\sum_{r(q)}\sum_{m\in\ZZ}\tf\left(\frac{N}{K_0}\left(\theta+m-\frac{r}{q}\right)\right).}
Our main term is then
\al{\label{mainterm}\mathcal{M}_N(n):=\int_0^1\mathfrak{T}(\theta)\hrn(\theta)e(-n\theta)d\theta}
and the error term is
\al{\label{errorterm}\enn:=\int_0^1(1-\Tf(\theta))\hrn(\theta)e(-n\theta)d\theta.}
Similarly, we define
\al{\label{maintermu}\mathcal{M}_N^U(n):=\int_0^1\mathfrak{T}(\theta)\hrnu(\theta)e(-n\theta)d\theta}
and
\al{\label{errortermu}\mathcal{E}_N^U(n):=\int_0^1(1-\mathfrak{T}(\theta))\hrnu(\theta)e(-n\theta)d\theta.}

In Lemma \ref{1130} we have shown that $\Vert \rN-\rnu\Vert _{l_1}$ is small.  Running the same argument as in the proof of Lemma \ref{1130}, one can bound the difference between ${\hrn}$ and ${\hrnu}$ in $l_1$ norm.  Then, one obtains 
\begin{lem}\label{1133}
$$\Vert \mn-\mnu\Vert _{l_1}\leq \frac{T^{2\delta}X^2}{U}.$$
\end{lem}
\noindent Together, Lemma  \ref{1130} and Lemma \ref{1133} then imply that 
\begin{lem}\label{lem1135}
$$\Vert \en-\enu\Vert _{l_1}\leq \frac{T^{2\delta}X^2}{U}.$$
\end{lem}
Appropriate lower bounds on $\mn(n)$ and average upper bounds on $ \enu$ are then combined to prove the main theorem.  Specifically, in Section \ref{sectionmajor} we show that 
\begin{thm}\label{0429}For any $n\in [N/2, N]\cap \mathcal{K}_a$, we have
$$\mn(n)\gg T^{2\delta-2}.$$
\end{thm}
In Section \ref{sectionminor1}, Section \ref{sectionminor2}, and Section \ref{sectionminor3} we work towards giving an $l_2$ bound for $\enu$:
\begin{thm}\label{0430}
$$\Vert \enu \Vert _{l_2}^2 \ll T^{4\delta-4}N^{(1-\eta)^2}.$$
\end{thm}
The value of $\eta$ will be described in the course of the proof.

Then using the H\"older inequality together with Lemma \ref{lem1135}, we have 
\al{\label{0443}\Vert \en\Vert _{l_1}\ll T^{2\delta-2}N^{1-\eta}. }
We are now able to prove Theorem \ref{mainthm} assuming Theorem \ref{0429} and \eqref{0443}.
\begin{proof}[Proof of Theorem \ref{mainthm}:]
  Let $\mathfrak{E}(N)$ be the set of exceptional numbers $[N/2,N]$ (those admissible but not occurring as curvatures).  Then, by \eqref{0443}, 
\aln{\sum_{n\in\mathfrak{E}(N)}|\en(n) |\leq  \Vert \en\Vert _{l_1}\ll T^{2\delta-2}N^{1-\eta}.}
For $n\in \mathfrak{E}(N)$, $\rn=0$ and, by Theorem \ref{0429}, $\mn(n)\gg T^{2\delta-2}$.  Therefore 
\aln{\vert\mathcal{E}_N(n)\vert= \vert\rn-\mathcal{M}_N(n) \vert  \gg T^{2\delta-2}. }
Thus 
\aln{\#\mathfrak{E}(N) \cdot T^{2\delta-2}\ll \sum_{n\in\mathfrak{E}(N)}|\en(n) |\ll T^{2\delta-2} N^{1-\eta},}
so that 
\al{\label{0157}\#\mathfrak{E}(N)\ll N^{1-\eta}. }
\par 
This is the desired result for the interval $[N/2,N]$, and we extend it to the full interval $[0,N]$ as follows. Divide $[0,N]$ into a union of subintervals dyadically: $[0,N]=[N/2,N]\cup[N/4,N/2]\cup[N/8,N/4]\cup\cdots$. Applying \eqref{0157} to each subinterval (replacing $N$ by $N/2^m$ for $0\leq m<\log_2(N)$) and collecting the error terms,  we obtain Theorem \ref{mainthm} as desired.
\end{proof}

\section{\label{lemmata}preliminary lemmata}
In this section we introduce several lemmata due to Bourgain-Kontorovich which will be used in later sections.  
Note that they are not stated exactly as the lemmata which we cite from \cite{BK14}, which are stated in the framework of counting in orbits of the Apollonian group in $\textrm{O}_{\mathbb R}(3,1)$ acting on Descartes quadruples in $\mathbb Z^4$, while we use the lemmata in the context of subgroups of $\textrm{PSL}_2(\mathcal O_K)$ acting on a circle.  However, the proofs of these lemmata in \cite{BK14} are very general, and apply almost verbatim to the context in which we phrase them below, with their group $\Gamma$ replaced by $\mathcal A$ in our case, and the set of first coordinates of points in the orbit of $\Gamma$ acting on a vector replaced by the curvatures of the circles one gets as in the orbit of $\mathcal A$ that we consider.  We note also that in Lemma~\ref{bk2} we sum over cosets of $\mathcal A(q)$ while Bourgain-Kontorovich sum over cosets of a larger subgroup.  However, this is not necessary to execute the circle method as we do here.  Finally, as stated below, Bourgain-Kontorovich's bounds involving $T^{\delta}$ and $T^{\Theta}$ are adjusted to involve $T^{2\delta}$ and $T^{2\Theta}$, respectively.  This is because we work in $\textrm{PSL}_2$ and not in $\textrm{SO}_{\mathbb R}(3,1)$ as is the case in \cite{BK14}, and the spin homomorphism from $\textrm{PSL}_2$ to $\textrm{SO}_{\mathbb R}(3,1)$ is quadratic in the entries of the matrices of $\textrm{PSL}_2$.

These results are the point at which the spectral gap for $\mathcal{A}$ feeds into our analysis.  The first two of these are statements about equidistribution modulo $q$.  The first says that the curvatures cannot have too strong a preference for a given congruence class modulo $q$, as $\gamma$ varies.  It is used in the minor arc analysis.

\begin{lem}[Bourgain-Kontorovich \cite{BK14}, Lemma 5.2]\label{bk1}
There exists a positive constant $\nu$ and some $\eta_0 >0$ which only depend on the spectral gap of $\Aa$, such that for any $1\leq q<N$ and any $r(${mod} $q$),
\aln{\sum_{\gamma\in\FF_T}\bd{1}\left\{\frac{2\Im(C_{M\cdot\gamma}\overline{D_{M\cdot \gamma}})}{\sqrt{-\Delta}}\equiv r (\textrm{mod }q)\right\}\ll\frac{T^{2\delta}}{q^{\eta_0}},}
where $T = T_1T_2$ (for notations see the definition of $\FF$ in \eqref{FF}).
The implied constant is independent of $r$.
\end{lem}

The second lemma states that the behaviour of the form $\ff_{M\gamma}$ on $\gamma$ from any given congruence class is independent of the congruence class, in the sense that each class contributes equally to an exponential sum.  It is used for the setup of the major arc analysis, to separate the non-archimedean and archimedean contributions.  Write $\mathcal{A}(q)$ for the kernel of reduction modulo $q$.
\begin{lem}[Bourgain-Kontorovich \cite{BK14}, Lemma 5.3]\label{bk2}Let $1<K<T_2^{\frac{1}{10}}$, fix $|\beta|<\frac{K}{N}$, and fix $x,y\asymp X$. Then for any $\gamma_0\in\mathcal{A}$, any $q\geq1$, we have
  $$\sum_{\gamma\in\FF_T\cap\gamma_0\mathcal{A}(q)}e(\beta\ff_{M\gamma}(Lx+1,Ly))=\frac{1}{[\mathcal{A}:\mathcal{A}(q)]}\sum_{\gamma\in\FF_T}e(\beta\ff_{M\gamma}(Lx+1,Ly))+O(T^{2\Theta_1}K),$$
where $\Theta_1<\delta$ depends only on the spectral gap for $\Aa$, and the implied constant does not depend on $q,\gamma_0,x$ or $y$.
\end{lem}

The last lemma is used to bound the archimedean piece of the major arc analysis.  It uses the spectral gap to control the error in counting $\gamma \in \FF_T$ where $\ff_{M\gamma}$ takes certain values.
\begin{lem}[Bourgain-Kontorovich \cite{BK14}, Lemma 5.4]\label{bk3} Fix $N/2\leq n\leq N,1<K\leq T_2^{\frac{1}{10}}$, and $x,y\asymp X$.  Then 
$$\sum_{\gamma\in\FF_T}\bd{1}{\left\{|\ff_{M\cdot \gamma}(Lx+1,Ly)-n|<\frac{N}{K}\right\}}\gg\frac{T^{2\delta}}{K}+T^{2\Theta_2},$$
where $\Theta_2<\delta$ depends only on the spectral gap for $\Aa$.  The implied constant is independent of $x,y$ and $n$.
\end{lem}

Let $\Theta$ be the larger of the $\Theta$'s that satisfy Lemma \ref{bk2} and Lemma \ref{bk3} respectively, then we set the two parameters $Q_0,K_0$ as
\al{\label{0737}
\boxed{
  Q_0=T^{\frac{2\delta-2\Theta}{80}}, K_0=Q_0^3.}
}

\section{\label{sectionmajor}Major arc analysis}
In this section we prove Theorem \ref{0429} bounding $\mathcal{M}_N(n)$ below.  We give a brief overview of the argument, before treating all the details.  First, we will write 
    \[
      \mathcal{M}_N(n) = \sum_{x,y \mbox{ in a region }} \fs_{Q_0}(n) \mathfrak{M}(n) + \mbox{error},
    \]
    where $\mathfrak{M}(n)$ is the \emph{Archimedean part} and $\fs_{Q_0}(n)$ is the \emph{non-Archimedean part} (depending on a parameter $Q_0$ controlling the size of the major arcs); both depend on $x,y$.  We need Lemma \ref{bk2} (dependent on the spectral gap) in order to accomplish this separation of Archimedean from non-Archimedean.  
    
    The Archimedean part is bounded below by Lemma \ref{bk1}, and most of the attention of this section is given to bounding $\fs_{Q_0}(n)$ below.  This requires a careful local analysis that is one of the novelties of our treatment as compared with previous works \cite{BK14, Zh14}.
  
    The limit
    $\fs(n) = \lim_{Q_0 \rightarrow \infty} \fs_{Q_0}(n)$ is the \emph{singular series}, whose purpose is to be supported only on the admissible values of $n$, and bounded below where it is supported.  We break it down as
    \[
      \fs(n) = \sum_{q=1}^\infty B_q(n) = \prod_p (1 + B_p(n) + B_{p^2}(n) + \cdots ).
    \]
    In turn,
    \[
      B_q(n) = \sum_{r (q) } \tau_q(r)c_q(r-n),
    \]
    where $\tau_q(r)$ is the probability of the quadratic form $\ff_{M\gamma}$ taking on the value $r$ modulo $q$, as $\gamma$ ranges among cosets of $\mathcal{A}$ modulo $q$, and $c_q$ is a Ramanujan sum, which is multiplicative with respect to $q$.  For a prime $p$, one should think of $B_p(n)$ as measuring some deviation from the equidistribution of the probabilities $\tau_p(n)$ modulo $p$; $B_{p^k}(n)$ for larger $k$ gives finer information about the behaviour of these probabilities as we lift to powers of $p$.  This is captured by the relationship
    \[
      1 + B_p(n) + B_{p^2}(n) + \cdots + B_{p^k}(n) = p^k \tau_{p^k}(n).
    \]
    This factor is non-zero if and only if $n$ is represented as a curvature modulo $p^k$.
    
    The goal, then, is to understand $B_{p^k}(n)$.  First, we use strong approximation for $\mathcal{A}$ to show that the $B_{p^k}(n)$ eventually vanish as $k$ increases.  In particular, $B_{p^k}(n) = 0$ once we have uniform lifting in the sense of strong approximation (Lemmas \ref{0939} and \ref{0948}).  We find that for all but finitely many primes, $B_{p^k}(n) = 0$ for $k \ge 2$.  Therefore $\fs(n)$ is controlled by the product over good primes $\prod'_{p}(1 + B_p(n))$.

    The final step is to control $B_p(n)$:  we show that for $p \nmid n$, $B_p(n) = O(1/p)$, while for $p \mid n$, $B_p(n) = O(1/p^2)$.  This requires a direct counting argument, finding all solutions modulo $p$ to the requirement that the curvature be equal to $n$; at its core is an argument using Gauss sums.  In other words, we show that equidistribution of curvatures modulo $p$ does not fail too badly.

    Now we begin.  From \eqref{mainterm}, \eqref{spike} and \eqref{1248A}, we have
\al{\label{0128}\nonumber
\mathcal{M}_N(n)=&\int_0^1\mathfrak{T}(\theta)\hrn(\theta)e(-n\theta)d\theta\\
\nonumber=&\int_{-\infty}^{\infty}\sum_{q<Q_0}\sum_{r(q)}{}'\frak{t}\left(\frac{N}{K_0}\beta\right)\hrn\left(\beta+\frac{r}{q}\right)e\left(-n\left(\beta+\frac{r}{q}\right)\right)d\beta\\
\nonumber=&\sum_{\sk{x,y\in\ZZ\\(Lx+1,Ly)=1}}\psi\left(\frac{Lx+1}{X}\right)\psi\left(\frac{Ly}{X}\right)\sum_{q<Q_0}\sum_{r(q)}{}'\sum_{\gamma\in\FF_T}e\left(\frac{r}{q}(\ff_{M\gamma}(Lx+1,Ly)-n)\right)\\&\cdot\int_{-\infty}^{\infty}\tf\left(\frac{N}{K_0}\beta\right)e(\beta(\ff_{M\gamma}(Lx+1,Ly)-n)d\beta
}
Now we decompose the set $\FF$ as left cosets of $\Aa(q)=\left\{\mat{a&b\\c&d}\in\Aa\Big\vert \mat{a&b\\c&d}\equiv I \pmod q\right\}$ and apply Lemma \ref{bk2} with $K=K_0$ to obtain  
\al{\label{0127}\nonumber&\sum_{\gamma\in\FF}e\left(\frac{r}{q}(\ff_{M\gamma}(Lx+1,Ly)-n)\right)\cdot\int_{-\infty}^{\infty}\tf\left(\frac{N}{K_0}\beta\right)e(\beta(\ff_{M\gamma}(Lx+1,Ly)-n))d\beta\\
\nonumber=&\sum_{\gamma_0\in \Aa/\Aa(q)}e\left(\frac{r}{q}(\ff_{M\gamma_0}(Lx+1,Ly)-n)\right)\int_{-\infty}^{\infty}\mathfrak{t}\left(\frac{N}{K_0}\beta\right) \sum_{\gamma\equiv\gamma_0(q)}e(\beta(\ff_{M\gamma}(Lx+1,Ly)-n))d\beta\\
\nonumber=&\frac{1}{[\Aa:\Aa(q)]}\sum_{\gamma_0\in \Aa/\Aa(q)}e\left(\frac{r}{q}(\ff_{M\gamma_0}(Lx+1,Ly)-n)\right)\int_{-\infty}^{\infty}\mathfrak{t}\left(\frac{N}{K_0}\beta\right)\sum_{\gamma\in\FF_T}e(\beta(\ff_{M\gamma}(Lx+1,Ly)-n))d\beta\\
\nonumber&+O\left(\frac{T^{2\Theta}K_0^2Q_0^6}{N}\right)\\
\nonumber=&\frac{K_0}{N}\cdot\frac{1}{[\Aa:\Aa(q)]}\sum_{\gamma_0\in \Aa/\Aa(q)}e\left(\frac{r}{q}(\ff_{M\gamma_0}(Lx+1,Ly)-n)\right)\cdot \sum_{\gamma\in\FF_T}\hat{\mathfrak{t}}\left(\frac{K_0}{N}(\ff_{M\gamma}(Lx+1,Ly)-n)\right)\\
&+O\left(\frac{T^{2\Theta}K_0^2Q_0^6}{N}\right),
}
where Lemma~\ref{bk2} is applied to obtain the third line above.  Inserting \eqref{0127} into \eqref{0128}, we get
\al{\label{0152}
\mnn=\sum_{\sk{x,y\in\ZZ\\(Lx+1,Ly)=1}}\psi\left(\frac{Lx+1}{X}\right)\psi\left(\frac{Ly}{X}\right)\fs_{Q_0}(n)\mathfrak{M}(n)+O\left(\frac{T^{2\Theta}X^2K_0^2Q_0^8}{N}\right)
}
where 

\al{\fs_{Q_0}(n)=\fs_{Q_0;x,y}(n):= \sum_{q<Q_0} \frac{1}{[\Aa:\Aa(q)]}\sum_{\gamma_0\in \Aa/\Aa(q)}c_q\left(\ff_{M\cdot \gamma_0}(Lx+1,Ly)-n\right)  }
and
\al{\mathfrak{M}(n)=\mathfrak{M}_{x,y}(n):=\frac{K_0}{N}\sum_{\gamma\in\FF}\hat{\mathfrak{t}}\left(\frac{K_0}{N}(\ff_{M\gamma}(Lx+1,Ly)-n)\right).}
Here $c_q$ is the Ramanujan sum defined by
\al{\label{1101}c_q(n)=\sideset{}'\sum_{r(q)}e\left(\frac{rn}{q}\right).}
Fixing $n$, we have that $c_q(n)$ is multiplicative with respect to $q$, and locally,
\al{\label{093444}c_{p^k}(n)=\begin{cases}0&\text{if } p^m\Vert n,m\leq k-2,\\
-p^{k-1}& \text{if } p^{k-1}\Vert n,\\
p^{k-1}(p-1)& \text{if } p^k|n.
\end{cases}}

The error term in \eqref{0152} is $O(T^{2\delta-2-\epsilon})$ by our choice of $K_0$ (see \eqref{eqn:growing-par} and \eqref{0737}), where $\epsilon$ is any small positive number at most $\frac{33}{20}(\delta-\Theta) > 0$.  
Applying Lemma \ref{bk3} with $K=K_0$, we can give a lower bound for the Archimedean piece $\mathfrak{M}(n)$ for any $N/2\leq n\leq N$:
\al{
\mathfrak{M}(n)\gg \frac{T^{2\delta}}{N}.
}
Therefore, Theorem \ref{0429} is proved once we show that $\fs_{Q_0}(n)\gg 1$ for every $n$ admissible (or, what actually suffices, once we have proven it up to log factors, since our aim is to get a power saving, which absorbs all log powers). 
The rest of this section is devoted to proving the following.
\begin{prop}\label{1203}
We have 
$\mathfrak{S}_{Q_0}(n)\gg \frac{1}{\log n}$
if $n$ is admissible,
and 
$\mathfrak{S}_{Q_0}(n)\ll \frac{\log n}{Q_0}$
if $n$ is not admissible.
\end{prop}

To understand $\fs_{Q_0}(n)$, we first push $Q_0$ to $\infty$.  We define a formal singular series
\al{\label{0303}
\fs(n)=\sum_{q=1}^{\infty}B_q(n),
} 
where 
\al{B_q(n)= \frac{1}{[\Aa:\Aa(q)]}\sum_{\gamma_0\in \Aa/\Aa(q)}c_q\left(\ff_{M\gamma_0}(Lx+1,Ly)-n\right). }
So to understand $\fs_{Q_0}(n)$ or $\fs(n)$ one must understand $B_q(n)$ for each $q$.

We rewrite $B_q(n)$ as  
\al{\label{1230}B_q(n)=\sum_{r(q)}\tau_q(r)c_q\left(r-n\right),}
where 
\al{\label{0412}\tau_q(r)=&\frac{1}{[\Aa:\Aa(q)]}\sum_{\gamma_0\in\Aa/\Aa(q)}\bd{1}\{\ff_{M\gamma_0}(Lx+1,Ly)\equiv r \pmod{q}\}\\
\label{0413}=&\frac{1}{[\Aa:\Aa(q)]}\sum_{\gamma_0\in\Aa/\Aa(q)}\bd{1}\left\{\kappa\left(M\cdot\gamma_0\left(\RR+\frac{\sqrt{\Delta}}{2}\right)\right)  \equiv r \pmod{q} \right\}.
}
The term $\tau_q(r)$ can be viewed as the probability that a curvature is congruent to $r$ mod $q$, as $\gamma$ ranges over $\mathcal{A}$.  To get from \eqref{0412} to \eqref{0413} we used the fact that $\ff_{M\gamma}(Lx+1,Ly)=\kappa\left(M\cdot\gamma w_{x, y}\left(\RR+\frac{\sqrt{\Delta}}{2}\right)\right)$
for some $w_{x, y}\in\Gamma(\BB)$ with left column $(Lx+1,Ly)^{T}$.\\

First we need the multiplicativity of $\Aa$ which will lead to the multiplicativity of $B_q(n)$:
\begin{lem} \label{0258}Write $q=\prod_i{p_i^{n_i}}$, then 
\aln{\Aa(q)\cong \prod_i \Aa(p_i^{n_i}).}
\end{lem}
Lemma \ref{0258} will lead immediately to the multiplicativity of $B_q(n)$ with respect to $q$.  Apriori Lemma \ref{0258} is not true for a general group $\Aa$.  If this is the case, we replace $\mathcal{A}$ by some congruence subgroup of $\Aa$ which satisfies the multiplicative property (such a subgroup exists by strong approximation in $\textrm{SL}_2$).  As noted in Section \ref{integralitysec}, we may move to a finite index subgroup without loss of generality.

Given this multiplicativity, we split \eqref{0303} into an Euler product
\al{\fs(n)=\prod_p(1+B_p(n)+B_{p^2}(n)+\cdots  ).   }
The arithmetic meaning of each factor of the Euler product is illustrated by the following formula: 
\begin{equation}
  \label{1015}
1+B_p(n)+B_{p^2}(n)+\cdots+ B_{p^k}(n)
=p^{k} \tau_{p^k}(n).
\end{equation}
To see this, let $s_\gamma$ be such that $p^{s_\gamma} \mid\mid \mathfrak{f}_{M\gamma}(Lx+1,Ly) - n$.  Then,
\begin{align*}
  &1 + \sum_{m=1}^k B_{p^m}(n)  \\
  &= \frac{1}{[\mathcal{A}:\mathcal{A}(p^k)]} \sum_{\gamma \in \mathcal{A}/\mathcal{A}(p^k)} \left( 1 + \sum_{m=1}^k c_{p^m}( \mathfrak{f}_{M\gamma}(Lx+1,Ly) - n ) \right) \\
  &= \frac{1}{[\mathcal{A}:\mathcal{A}(p^k)]} \sum_{\gamma \in \mathcal{A}/\mathcal{A}(p^k)} 
  \left\{ \begin{array}{ll} 0 & s_\gamma < k \\ p^k & {s_\gamma} \ge k \\ \end{array} \right. .
\end{align*}
Therefore, $1+B_p(n)+B_{p^2}(n)+\cdots+ B_{p^k}(n)$ is non-zero if and only if $n$ is represented (mod $p^k$).

Our goal for the rest of the section is to access $\fs(n)$ (and prove Proposition \ref{1203}) by analysing the values of $B_{p^k}(n)$.  First, we will show that
\begin{lem}\label{0939} There is an integer $P_{\textrm{bad}}\geq 1$ such that 
\begin{enumerate}
\item For any $p\nmid P_{\textrm{bad}}$ and $k\geq 2$, $B_{p^k}(n)=0$.
\item For each of the finitely many primes $p\mid P_{\textrm{bad}}$,  $\exists k_p'$ such that $B_{p^k}(n)=0$ for any $k\geq k_p'$.   
\end{enumerate}
\end{lem}
Indeed, Lemma \ref{0939} follows from the following fact for $\Aa(q)$ given by strong approximation in $\textrm{SL}_2$:  
\begin{lem}\label{0948}
  There is an integer $L_1 \ge 1$ such that
\begin{enumerate}
\item For any $p\nmid L_1$ and $k\geq 1$, \al{\label{2305}\Aa(p^{k-1})/\Aa(p^k)=\textrm{SL}_2(\mathcal{O}_K)(p^{k-1})/\textrm{SL}_2(\mathcal{O}_K)(p^{k}) }
\item For each of the finitely many primes $p\mid L_1$,  $\exists k_p$ such that  
\eqref{2305} holds for any $k\geq k_p$.
\end{enumerate}
\end{lem}
We refer to primes that divide $P_{\textrm{bad}}$ as \emph{bad primes}, and those that do not divide $P_{\textrm{bad}}$ as \emph{good primes}.  We given an explicit form of Lemma \ref{0948} in Theorem \ref{thm:explicit-strong}, which allows the computation of a valid $L_1$.  
We use a Hensel lifting argument to deduce Lemma \ref{0939} from Lemma \ref{0948}. 

\begin{proof}[Proof of Lemma \ref{0939}]
First we rewrite $B_{p^k}(n):$
\al{\label{1102}
\nonumber B_{p^k}(n)=&\frac{1}{[\Aa:\Aa(p^k)]}\sum_{\gamma\in \Aa/\Aa(p^k)}c_{p^k}(F_1(\gamma)-n)\\
=&\frac{1}{[\Aa:\Aa(p^k)]}\sum_{\gamma_0\in \Aa/\Aa(p^{k-1})}\sum_{\substack{\gamma\in\Aa/\Aa(p^k)\\\gamma\equiv\gamma_0(p^{k-1})}}c_{p^k}(F_1(\gamma)-n)}
where $F_1(\gamma)=\kappa(M\gamma(\widehat{\RR}+\frac{\sqrt{\Delta}}{2}))$ and we view $F_1$ as an algebraic function over the real and imaginary parts of the entries of $\gamma$.
As we have assumed $\mathcal{A} \subset \PSL_2(\ZZ[\sqrt{-d}])$ in Section \ref{integralitysec}, we may write $$\gamma=\mat{a_1+a_2\sqrt{d}\bd{i}&b_1+b_2\sqrt{d}\bd{i}\\c_1+c_2\sqrt{d}\bd{i}&d_1+d_2\sqrt{d}\bd{i}}.$$
We assume for the moment that $M$ is also in $\PSL_2(\ZZ[\sqrt{-d}])$, and write
$$M=\mat{M_{11}+M_{12}\sqrt{d}\bd{i}&M_{21}+M_{22}\sqrt{d}\bd{i}\\M_{31}+M_{32}\sqrt{d}\bd{i}&M_{41}+M_{42}\sqrt{d}\bd{i}}.$$
Then
\begin{align}\label{F1}
F_1(\gamma)=&F_1(a_1,a_2, b_1,b_2, c_1,c_2,d_1,d_2)\\
=&(M_{31}a_1+M_{41}c_1)^2+d(M_{32}a_2+M_{42}c_2)^2
+ (M_{31}a_1+M_{41}c_1)(M_{32}b_2+M_{42}d_2)\\&-(M_{31}b_1+M_{41}d_1)(M_{32}a_2+M_{42}c_2)  .
\end{align}
As $\gamma\in \textrm{PSL}_2(\mathcal{O}_K)$, these variables are also subject to the following conditions: 
\begin{eqnarray}\label{F2F3}
&&F_2(a_1,a_2, b_1,b_2, c_1,c_2,d_1,d_2):=a_1d_1-a_2d_2d-b_1c_1+b_2c_2d-1=0\nonumber\\
&&F_3(a_1,a_2, b_1,b_2, c_1,c_2,d_1,d_2):=a_1d_2+a_2d_1-b_1c_2-b_2c_1=0
\end{eqnarray}

If $M$ is integral, one can check that the Jacobian matrix $$J=\frac{\partial (F_1,F_2,F_3)}{\partial (a_1,a_2, b_1,b_2, c_1,c_2,d_1,d_2)}$$  maps $\ZZ_p^8$ onto $\ZZ_p^3$ as a linear transformation, at each point of the  affine variety $\mathbb{V}[\QQ_p]$ defined by the following equations: $$\begin{cases}F_1=n\\F_2=0\\F_3=0\end{cases}.$$
For any $k\geq k_p$ ($k_p$ can be taken to be 2 if $p$ is good), the lifting $\mathbb{V}[\ZZ/p^{k-1}\ZZ]\rightarrow \mathbb{V}[\ZZ/p^{k}\ZZ]$ becomes regular by Lemma \ref{0948}, and for $\gamma_0$ such that $F_1(\gamma_0) \equiv n \pmod{p^{k-1}}$, this gives 
\begin{equation}
  \label{eqn:prob}
  \text{Prob}\left(
  F_1(\gamma)\equiv n\pmod{p^k}
  \;\Big\vert\; 
  \gamma\in\Aa/\Aa(p^k),\gamma\equiv\gamma_0 \pmod{p^{k-1}}
  \right)=\frac{1}{p}.
\end{equation}

Returning to \eqref{1102}, from \eqref{093444} the innermost sum of \eqref{1102} is zero unless $F_1(\gamma_0)\equiv n \pmod{p^{k-1}}$.  
However, in this case, if $k \geq k_p$, by \eqref{eqn:prob} (at $k$ and $k+1$) and \eqref{093444}, each innermost sum of \eqref{1102} is still zero.
Therefore $B_{p^k}(n)=0$, for $k\geq k_p$.

Above, we assumed $M \in \PSL_2(\ZZ[\sqrt{-d}])$, and we found that the $k_p'$ of Lemma \ref{0939} agree with the $k_p$ of Lemma \ref{0948}.  If instead $M$ is fractional in the sense that the $M_{ij}$ have denominator $q_0$, then we need to multiply $F_1$ by $q_0^2$ to make it integral.  For any $p^{n_p}||q_0$, one can check that $\ZZ_p^3\subset \frac{1}{p^{2n_p}}J(\ZZ_p^8)$.  In this case, $B_{p^k}(n)=0$ for $k\geq k_p+2n_p$.
\end{proof}

The obstruction number $L_0$ in the statement of Corollary \ref{local} is thus given by
\begin{equation}
  \label{eqn:L0}
  L_0=\prod_p p^{k_p+2n_p}.
\end{equation}
The computation of upper bounds on $k_p$ is given by Theorem \ref{thm:explicit-strong}, and some examples are given in Section \ref{sec:examples}.

At this point, in order to prove Proposition \ref{1203}, as there are only finitely many bad primes, we have shown that it suffices to analyse the contribution of $1 + B_p(n)$ for good odd primes $p$.

\begin{lem}\label{Bpnlemma} Suppose $n$ is admissible.  Let $p$ be an odd prime not dividing $P_{bad}$.  Then
$B_{p}(n)=O(1/p^2)$ if $p\nmid n$ and $B_{p}(n)=O(1/p)$ if $p\mid n$, where the implied constants are independent of $n$.
\end{lem}

To prove this, we first prove the following.

\begin{lem}\label{taupnlemma}
  Let $p$ be an odd prime not dividing $P_{bad}$.
\al{\tau_{p}(n)=\left\{ \begin{array}{ll}
 \frac{1}{p}+O(\frac{1}{p^3})&\text{if } n\not\equiv 0\pmod{p}\\
 \frac{1}{p}+O(\frac{1}{p^2})&\text{if }n\equiv 0\pmod{p}
 \end{array} \right. ,}
 where the implied constants are independent of $n$.
\end{lem}

There are at least two proofs of this fact. One is the proof we give below, which works directly in the group $\textrm{SL}_2(\mathcal O_K)$.  Another approach is to consider the image under the spin homomorphism $\rho$ of $\mathcal A$ in $\textrm{O}_Q(\mathbb Z)$ where $Q(x_1,x_2,x_3,x_4)=x_2^2+dx_3^2+x_1x_4$, and note that the set of curvatures we are interested in is, up to a factor of $d$, exactly the set of fourth coordinates of points in the orbit $\rho(\mathcal A) v^T$, where $v=(-d,0,1,0)$.  By strong approximation, modulo $p$ the orbit $\rho(\mathcal A) v^T$ is simply the set of all solutions to $Q(x_1,x_2,x_3,x_4)\equiv d\pmod p$, and $\tau_p(n)$ is easily computed by counting representations modulo $p$ of $d$ by specific quadratic forms.  This passing between $\textrm{SL}_2(\mathbb C)$ and $\textrm{O}_{\mathbb{R}}(3,1)$ is a nod to the description of Apollonian circle packings in \cite{BK14}, \cite{GLMWY03} and \cite{Zh14}, where curvatures can be seen by looking at orbits of certain thin subgroups of $\textrm{O}_F(\mathbb Z)$ as described in the introduction.  Since we describe Apollonian packings somewhat more geometrically, we present the proof from that point of view.

\begin{proof}
  Let $p$ be an odd prime not dividing $P_{\textrm{bad}}$.  Let $\gamma \in \mathcal{A}$.  Write $\gamma_p$ for the reduction of $\gamma$ in $\mathcal{A}/\mathcal{A}(p)$.  By Lemma \ref{0948}, we have that $\gamma_p$ ranges over all of
  \[
    \textrm{SL}_2(\ZZ[\sqrt{-d}])/\textrm{SL}_2(\ZZ[\sqrt{-d}])(p) = \textrm{SL}_2(\ZZ[\sqrt{-d}]/(p) ).
  \]
  Therefore, we have
  \[
    \tau_p(n) = \frac{ \# \mathbb{V}[\ZZ/p\ZZ] }{ \# \textrm{SL}_2( \ZZ[\sqrt{-d}]/(p) ) }.
  \]
  For any commutative ring $R$ with identity, the allowable first columns of $\textrm{SL}_2(R)$ is a set
  \[
    U(R) = \{ \mbox{pairs }(a,b) \;\vert\; \mbox{as ideals, } (a,b) = R \}.
  \]
  We have that $\# U(R) = \# \mathbb{P}^1(R) \cdot \# R^*$.  Furthermore,
  \[
    \# \textrm{SL}_2(R) = \# U(R) \cdot \# \operatorname{Stab}_{*}( \textrm{SL}_2(R) )
    = \# U(R) \cdot \# R,
  \]
  where we write $\operatorname{Stab}_*$ for the stabilizer of any one element of $U$.
  In our case, $R = \ZZ[\sqrt{-d}]/(p)$, this implies
  \[
    \tau_p(n) = \frac{ \# \mathbb{V}[\ZZ/p\ZZ] }{ p^2 \# \mathbb{P}^1\left( \mathbb{Z}[\sqrt{-d}]/(p) \right) \#  \left( \mathbb{Z}[\sqrt{-d}]/(p) \right)^* }.
  \]
  We have
  \begin{equation}\label{sizep1}
    \# \mathbb{P}^1\left( \ZZ[\sqrt{-d}]/(p) \right) 
    =\left\{ \begin{array}{ll}
 p^2+1&\text{if }\left(\frac{-d}{p}\right)=-1\\
 p^2+2p+1&\text{if }\left(\frac{-d}{p}\right)=1
 \end{array} \right. .
\end{equation}
and
  \begin{equation}\label{sizep1b}
    \# \left( \ZZ[\sqrt{-d}]/(p) \right)^* 
    =\left\{ \begin{array}{ll}
 p^2-1&\text{if }\left(\frac{-d}{p}\right)=-1\\
 p^2-2p+1&\text{if }\left(\frac{-d}{p}\right)=1
 \end{array} \right. .
\end{equation}
  It remains to compute $\# \mathbb{V}[\ZZ/p\ZZ]$.  But we have
  \begin{align*}
    \# \mathbb{V}[\ZZ/p\ZZ] &= \# \{ \lambda \in \textrm{SL}_2\left( \ZZ[\sqrt{-d}]/(p) \right) : F_1(\lambda) = n \} \\
    &= \# \{ v \in U\left( \ZZ[\sqrt{-d}]/(p) \right) : F_1(v) = n \} \cdot \# \operatorname{Stab}_{*}\left( \textrm{SL}_2\left(\ZZ[\sqrt{-d}]/(p) \right) \right) \\
    &= p^2 \# \{ v \in U\left( \ZZ[\sqrt{-d}]/(p) \right) : F_1(v) = n \}.
  \end{align*}
  In the above, we use the notation $F_1(v) = F_1(\lambda)$ for any $\lambda$ having bottom row $v$ (upon which $F_1$ depends exclusively).

  Therefore, it remains to compute
  \[
\# \left\{ v \in U \left( \ZZ[\sqrt{-d}]/(p) \right) : F_1(v) = n \right\} .
\]
If we assume that $M = I$, then we can write the equation $F_1(\lambda)=F_1(v)=n$ explicitly in terms of
$$\lambda=\mat{a_1+a_2\sqrt{d}\bd{i}&b_1+b_2\sqrt{d}\bd{i}\\c_1+c_2\sqrt{d}\bd{i}&d_1+d_2\sqrt{d}\bd{i}}$$
as
\begin{equation}
  \label{eqn:f1v}
  c_1^2 + c_2^2 d + (c_1d_2 - c_2d_1) - n \equiv 0 \pmod p.
\end{equation}
We count the number of solutions by evaluating
the following exponential sum:
\begin{align*}
&\frac{1}{p}\sum_{s(p)}\sum_{c_1,c_2,d_1,d_2(p)}e_p\left(s(c_1^2+c_2^2d+c_1d_2-c_2d_1-n)  \right)\nonumber\\
=&\frac{1}{p}\sum_{s(p)}\sum_{c_1,c_2,d_1,d_2(p)}e_p\left(s(c_1+d_2/2)^2+sd(c_2-d_1/{2d})^2-sd_2^2/4-sd_1^2/{4d}-sn ) \right)\nonumber\\
=&\frac{1}{p}\sum_{s\equiv 0(p)}\cdots+\frac{1}{p}\sum_{s\not\equiv 0(p)}\cdots \nonumber\\
=&p^3+ \frac{1}{p}\cdot\sum_{s\neq 0(p)} p^2\left(\frac{s}{p}\right)\left(\frac{sd}{p}\right)\left(\frac{-s}{p}\right)\left(\frac{-s/d}{p}\right)\cdot e_p(-sn)\nonumber\\
=& \left\{ \begin{array}{ll} p^3+p(p-1) & \text{ if }n\equiv  0\pmod{p}\nonumber\\
  p^3-p\hspace{1.25cm} & \text{       if }n\not\equiv 0 \pmod{p} \\ \end{array} \right.
\end{align*}
where $\left(\frac{\cdot}{\cdot}\right)$ is the Legendre symbol and we obtained the second to last step by applying Gauss sums first to $c_1,c_2$, then to $d_1,d_2$.\par
To obtain $\#\mathbb{V}[\ZZ/p\ZZ]$ we need to subtract the contribution from solutions not in $U\left( \ZZ[\sqrt{-d}]/(p) \right)$.  It turns out if $n\equiv 0 \pmod p$ then all such are solutions to \eqref{eqn:f1v}; if $n\not\equiv 0 \pmod p$ then none such are solutions.  We thus arrive at the following result:
\begin{equation*}\label{110854b}
  \# V[\ZZ/p\ZZ]=p^2\cdot \left\{ \begin{array}{ll}
 p^3+p(p-1)-1&\text{if }\left(\frac{-d}{p}\right)=-1\text{ and }n\equiv0(p)\\
 p^3-p&\text{if }\left(\frac{-d}{p}\right)=-1\text{ and }n\not\equiv0(p)\\
p^3-p^2-p+1&\text{if }\left(\frac{-d}{p}\right)=1\text{ and }n\equiv0(p)\\
p^3-p&\text{if }\left(\frac{-d}{p}\right)=1\text{ and }n\not\equiv0(p)
 \end{array} \right. . 
\end{equation*}
Now, if $M \neq I$, the effect of $M$ on the equation \eqref{eqn:f1v} is to apply an invertible linear transformation to $\left(\ZZ[\sqrt{-d}]/(p)\right)^2$ (recall that we are dealing only with good primes $p$).  This takes $U\left( \ZZ[\sqrt{-d}]/(p) \right)$ to $U\left( \ZZ[\sqrt{-d}]/(p) \right)$.  Therefore, the number of solutions $\#\mathbb{V}[\ZZ/p\ZZ]$ is unaffected.

Therefore, we obtained the formula for $\tau_p(n)$:
\begin{equation*}\label{110854c}
\tau_p(n)=\left\{ \begin{array}{ll}
 \frac{p+1}{p^2+1}&\text{if }\left(\frac{-d}{p}\right)=-1\text{ and }n\equiv0(p)\\
 \frac{p}{p^2+1}&\text{if }\left(\frac{-d}{p}\right)=-1\text{ and }n\not\equiv0(p)\\
\frac{1}{p+1}&\text{if }\left(\frac{-d}{p}\right)=1\text{ and }n\equiv0(p)\\
\frac{p}{p^2-1}&\text{if }\left(\frac{-d}{p}\right)=1\text{ and }n\not\equiv0(p)
 \end{array} \right. .
\end{equation*}
and indeed we have that $\tau_p(n)=\frac{1}{p}+O(\frac{1}{p^3})$ if $n\not\equiv 0\pmod{p}$ and $\tau_p(n)=\frac{1}{p}+O(\frac{1}{p^2})$ if $n\equiv 0\pmod{p}$ as desired.
\end{proof}


\noindent\emph{Proof of Lemma ~\ref{Bpnlemma}:}

\noindent Recall from \eqref{1230} that 
\begin{align*}
 B_p(n)&=\sum_{r(p)}\tau_p(r)c_p(r-n)\\
&=\tau_p(n)(p-1)+\sum_{\substack{r(p)\\r\not\equiv n(p)}}-\tau_p(r)\\
&=\tau_p(n)(p-1)-(1-\tau_p(n))\\
&=p\tau_p(n)-1
\end{align*}
Now apply Lemma~\ref{taupnlemma}.
\qed

We now combine everything to obtain an estimate of $\fs(n)$.

\begin{lem} \label{12099}The term $\fs(n)\neq 0$ if and only if $n$ is admissible, and when $n$ is admissible, we have $\fs(n)\gg \frac{1}{\log n}$.
\end{lem}

\begin{proof}
  We have already observed that $n$ is admissible if and only if it is represented modulo all integers, which occurs if and only if $\fs(n) \neq 0$.
  If $n$ is admissible, Lemma 6.3 demonstrates that its growth is controlled by the product $\prod_{p\; \mbox{\tiny good}}(1 + B_p(n))$.  Lemma \ref{Bpnlemma} shows that $1 + B_p(n) = 1 + O(1/p)$ or $1 + O(1/p^2)$; the contribution from the latter converges, and the contribution from the former gives growth $1/\log n$.
\end{proof}

Finally, we show that the difference between $\fs_{Q_0}(n)$ and $\fs(n)$ is indeed small:  
 \begin{lem}\label{1202}
 We have 
$$\vert\fs_{Q_0}(n)-\fs(n)\vert\ll \frac{\log n}{Q_0}.$$
\end{lem}
Recall here that $Q_0$ is a small power of $N$.
\begin{proof}
  Let $L_1$ be as in Lemma \ref{0948}.  Write $q=q_1q_2q_3$, where $q_1=(q,L_1),q_2=(q/q_1,n)$, so that $(q_3,L_1 n)=1$.  Noting that $B_{q_1}(n)$ has a universal upper bound, and recalling that $B_q(n)$ is multiplicative with respect to $q$, we have
\aln{\vert\fs_{Q_0}(n)-\fs(n)\vert&\leq \sum_{q>Q_0}\vert B_q(n) \vert\\
&=\sum_{q_1|L_1}|B_{q_1}(n)|\sum_{\substack{(q_2,L_1)=1\\q_2|n}}|B_{q_2}(n)|\sum_{\substack{(q_3,L_1n)=1\\q_1q_2q_3\geq Q_0}}|B_{q_3}(n)|\\
&\ll\sum_{q_1|L_1}\sum_{\substack{(q_2,L_1)=1\\q_2|n}}\frac{1}{q_2}\sum_{\substack{(q_3,L_1n)=1\\q_3\geq \frac{Q_0}{q_1q_2}}}\frac{1}{q_3^2}\ll  \sum_{q_2|n}\frac{1}{q_2}\frac{q_2}{Q_0}\ll \frac{\log n}{Q_0}
}
as desired.
\end{proof}
Lemma \ref{12099} and Lemma \ref{1202} together imply Proposition \ref{1203}.  Therefore, by the discussion preceding Proposition~\ref{1203}, we have shown Theorem \ref{0429}.

\section{\label{sectionminor}minor arcs}

The aim of this section is to prove
\al{\label{0825}
\int_0^1(1-\mathfrak{T}(\theta))^2\vert\hrnu(\theta)\vert^2d\theta \ll T^{2\delta-2}N^{1-\eta}.
}
By Plancherel's theorem, \eqref{0825} leads to Theorem \ref{0430}.

We bound the integral \eqref{0825} above by $\mathcal{I}_1 + \mathcal{I}_2 + \mathcal{I}_3$, where $J$ is the depth of approximation (see \eqref{eqn:setJ}) and
 \al{
   &\mathcal{I}_1=\sum_{q<Q_0}\sideset{}'\sum_{r(q)}\int_{\frac{r}{q}-\frac{1}{qJ}}^{\frac{r}{q}+\frac{1}{qJ}}|(1-\Tf(\theta))\hrnu(\theta)|^2d\theta,\\
   &\mathcal{I}_2=\sum_{Q_0\leq q<X}\sideset{}'\sum_{r(q)}\int_{\frac{r}{q}-\frac{1}{qJ}}^{\frac{r}{q}+\frac{1}{qJ}}|(1-\Tf(\theta))\hrnu(\theta)|^2d\theta,\\
   &\label{i3}\mathcal{I}_3=\sum_{X\leq q\leq J}\sideset{}'\sum_{r(q)}\int_{\frac{r}{q}-\frac{1}{qJ}}^{\frac{r}{q}+\frac{1}{qJ}}|(1-\Tf(\theta))\hrnu(\theta)|^2d\theta.
} 
The integrand is periodic on $\RR$ modulo $1$, and by Dirichlet's Theorem on Diophantine approximation, the domains of these integrals cover the circle $\RR$ modulo $1$.

The first integral $\mathcal{I}_1$ concerns small $q$ in the range of the major arc analysis, the second integral $\mathcal{I}_2$ concerns $q$ in the intermediate range $Q_0\leq q<X$, and the last integral $\II_3$ concerns large $q$. 

In Section~\ref{sectionminor1} we show 

\al{\II_1\ll T^{4\delta-4}N^{1-\eta}.}

Then, in Sections~\ref{sectionminor2} and \ref{sectionminor3} we divide $[Q_0,X]$ dyadically and prove 
\al{\mathcal{I}_Q:=\sum_{Q<q\leq 2Q}\sideset{}'\sum_{r(q)}\int_{\frac{r}{q}-\frac{1}{qJ}}^{\frac{r}{q}+\frac{1}{qJ}}|\hrnu(\theta)|^2d\theta\ll T^{4\delta-4}N^{1-\eta},}
where $Q_0\leq Q<X$ and $X\leq Q\leq J$ respectively. In doing this, we deal with the ranges of $Q$ corresponding to $\mathcal I_1$ and $\mathcal I_2$ separately and this will give the desired upper bounds on those sums.

It is evident that whether or not $M$ is fractional has little effect in the minor arc analysis: the main player here is the congruence subgroup $\Gamma(L)$ which gives rise to shifted quadratic forms.  We can simply replace the shifted quadratic form by a constant multiple of the form, and the analysis will run in exactly the same way.

\subsection{Lemmata for minor arcs}

In this section, we include some lemmata which will be used in the minor arc analysis.  The reader can choose to continue to the next section and refer back here for statements.  These lemmata relate to the evaluation and bounds for exponential sums of the form 
\al{{S}(q,A,B,C,D,E)=\sum_{x,y(q)}e(Ax^2+Bxy+Cy^2+Dx+Ey).}
and certain of their averages.  For simplicity we assume $q$ is odd.  
For $z\in\QQ_p$, we define 
\begin{displaymath}
\textrm{deg}_{p^m}(z)=\max_{-\infty<k\leq m}\{k: p^{-k}z\in\ZZ_p\}.
\end{displaymath}

We need the following lemma, which is a direct corollary of Gauss sums (see Page 13 of \cite{Dav}).
\begin{lem}\label{0343} For $a,b \in \ZZ$, we have
\al{\nonumber
&\sum_{x\in \ZZ/p^m\ZZ}e_{p^m}(ax^2+bx)=\\
&\begin{cases} p^{m}\cdot \bd{1}\{p^m|b\}& \text{if }p^{m}|a \\ p^{m/2}(p^m, a)^{1/2}i^{\epsilon\left(\frac{p^m}{(p^m,a)}\right)}\left(\begin{array}{ccc}\frac{a}{(p^{m-1},a)}\\ p\end{array}\right)e_{p^{}}\left(-\frac{b^2}{4a}\right)\cdot\bd{1}\{\text{deg}_{p^m}(b)\geq \textrm{deg}_{p^m}(a)\} & \text{if }p^{m}\nmid a  \end{cases},
}
where $\epsilon(n)=0$ if $n\equiv 1 \pmod 4$ and $\epsilon(n)=1$ if $n\equiv 3 \pmod 4$, and $\left(\begin{array}{ccc}\cdot\\\cdot\end{array}\right)$ is the Legendre symbol.
\end{lem}

The Legendre symbol $\left(\begin{array}{ccc}a\\ p\end{array}\right)=1$ if $a$ is a quadratic residue, and $-1$ if it is a quadratic non-residue.  By convention we also let $\left(\begin{array}{ccc}a\\ p\end{array}\right)=1$ if $a\equiv 0 \pmod p$.  The Legendre symbol is multiplicative only on the set of nonzero congruence classes mod $p$.

Write $g(x,y)=Ax^2+Bxy+Cy^2$ and $\Delta_g=B^2-4AC$. 
Let $k_g=\textrm{deg}_{p^m}(A,B,C)$. If $k_g<m$ and $p^m| \Delta_g/p^{k_g}$, we say $g$ is \emph{degenerate} at $p^m$; in this case $g$ is essentially a quadratic form of only one variable.   

From Lemma \ref{0343}, we obtain:

\begin{lem}\label{0626}Let $p$ be an odd prime.  Let $k_g=\textrm{deg}_{p^m}(\text{gcd}(A,B,C))$. If $k_g=m$, then 
\al{\nonumber
S(p^m,A,B,C,D,E)=p^{2m}\bd{1}\{p^m| \{D, E\}\}.
}
If $k_g<m$, then
\al{\nonumber S(p^m,A,B,C,D,E)=&p^mp^{\frac{k_g}{2}}\left(p^{m},\frac{\Delta_g}{p^{k_g}}\right)^{\frac{1}{2}}i^{\epsilon(p^{m-k_g})}i^{\epsilon\left(\frac{p^m}{(p^m,\Delta_g/p^{k_g})}\right)}e_{p^m}\left(\frac{g(E,-D)}{\Delta_g}\right)\\
&\cdot (-1)^{\upsilon(g)} \chi(p^m,A,B,C,D,E)
}
where $\upsilon (g)=0$ if $g$ is non-degenerate at $p^m$ and $\left(\begin{array}{ccc}-\frac{\Delta_g}{(\Delta_g, p^{2m-2})}\\p\end{array}\right)=1$, or $g$ is degenerate and the quadratic form $g(x,y)/p^{k_g}$ can represent nonzero quadratic residue mod $p$; $\upsilon(g)=1$ otherwise. The function $\chi(p^m;A,B,C, D,E)=1$ if $S(p^m, A,B,C,D,E)\neq 0$, and  $\chi(p^m;A,B,C, D,E)=0$ if $S(p^m, A,B,C,D,E)\neq 0$.
\end{lem}

\begin{proof}

It is a case-by-case proof, and the statement of Lemma \ref{0626} is a synthesis of all cases. \par
If $k_g=m$, the proof is trivial.  We thus assume 
If $k_g<m$.  Then after a linear unimodular change of variables,  we can rewrite 
\al{{S}(p^m,A,B,C,D,E)={S}(p^m,A',0,C',D',E')}
where $\textrm{deg}_{p^m}(A')=k_g$ and $C'=\frac{-\Delta_g}{4A'}$.  For instance, if $\textrm{deg}_{p^m}(A)=\textrm{deg}_{p^m}(\textrm{gcd}(A,B,C))$, then we can let $x'=x+\frac{B}{2A}y,y'=y$, then $Ax^2+Bxy+Cy^2+Dx+Ey=Ax'^2+(C-\frac{B^2}{4A})y'^2+Dx'+(E-\frac{BD}{2A})y'$. 
If, instead, $\textrm{deg}_{p^m}(B) < \textrm{deg}_{p^m}(A), \textrm{deg}_{p^m}(C)$, then we can apply the change $x' = x+y$, $y' = x-y$ to reduce to the previous case.

Now we can evaluate ${S}(p^m,A,B,C,D,E)={S}(p^m,A',0,C',D',E')$ from Lemma \ref{0343},

We have
\al{\label{1214}\nonumber& {S}(p^m,A,B,C,D,E)={S}(p^m,A',0,C',D',E')\\
\nonumber=&p^m(p^m,A')^{\frac{1}{2}}(p^m,C')^{\frac{1}{2}}i^{\epsilon\left(\frac{p^m}{(p^m,A')}\right)}  i^{\epsilon\left(\frac{p^m}{(p^m,C')}\right)}\\
&\cdot \left(\begin{array}{ccc}\frac{A'}{(A', p^{m-1})}\\p\end{array}\right)\left(\begin{array}{ccc}\frac{C'}{(C', p^{m-1})}\\ p\end{array}\right)  \bd{1}\left\{\begin{array}{ccc}\textrm{deg}_{p^m}(D')\geq \textrm{deg}_{p^m}(A')\\ \textrm{deg}_{p^m}(E')\geq \textrm{deg}_{p^m}(C')\end{array}\right\} e_{p^m}\left(\frac{g(E,-D)}{\Delta_g}\right),}

We interpret \eqref{1214} in an intrinsic way.  First, while all other factors are nonzero, the factor 
\al{\label{0421} \bd{1}\left\{\begin{array}{ccc}\textrm{deg}_{p^m}(D')\geq \textrm{deg}_{p^m}(A')\\ \textrm{deg}_{p^m}(E')\geq \textrm{deg}_{p^m}(C')\end{array}\right\}}
is the same as the indicator function indicating whether $S$ is zero or not.  So we have $\eqref{0421}=\chi(p^m,A,B,C,D,E)$.

For the term $A'$, we know $\text{deg}_{p^m}(A')=\deg_{p^m}(\text{gcd}(A,B,C))=k_g < m$, and that $\text{deg}_{p^m}(C')=\text{deg}_{p^m}(\Delta_g/A')$. \par

If $p^{m}\nmid C'$, then 
$\left(\begin{array}{ccc}\frac{A'}{(A', p^{m-1})}\\p\end{array}\right)\left(\begin{array}{ccc}\frac{C'}{(C', p^{m-1})}\\ p\end{array}\right)=\left(\begin{array}{ccc} {-\frac{\Delta_g}{4(\Delta_g,p^{2m-2})}}\\{p} \end{array}\right)$. 

If $p^{m}\mid C'$, then $\left(\begin{array}{ccc}\frac{C'}{(C', p^{m-1})}\\p\end{array}\right)=1$, and  
\al{
\left(\begin{array}{ccc}\frac{A'}{(A', p^{m-1})}\\p\end{array}\right)=\begin{cases}1 & \text{if nonzero quadratic residue is represented by }\\ &g(x,y)/{p^{k_g}} \text{in } \ZZ/p\ZZ\\
-1 & \text{otherwise}.\end{cases}
}
\end{proof}

We note here that the function $\chi(p^m,A,B,C,D,E)$ concerns whether the $p^m-$degrees of the $x,y$ coefficients are bigger than or equal to that of the $x^2,y^2$ coefficients after diagonalizing the quadratic part of $Ax^2+Bxy+Cy^2+Dx+Ey$.  We list the following two noteworthy properties of $\chi$: 
\begin{enumerate}
\item $\chi(p^m,A,B,C,D,E)$ is invariant under scaling of the quadratic part or the linear part, i.e. for any $(r,p)=1$,   \al{\label{0451}\chi(p^m,A,B,C,D,E)=\chi(p^m,rA,rB,rC,D,E)=\chi(p^m,A,B,C,rD,rE).}
\item $\chi(p^m,A,B,C,D,E)$ is invariant under changing variables of $x,y$. If $x=x_1+a,y=y_1+b$, then 
\al{\nonumber\label{0442}Ax^2+Bxy&+Cy^2+Dx+Ey=Ax_1^2+Bx_1y_1+Cy_1^2+(2Aa+Bb+D)x_1\\&+(2Cb+Ba+E)y_1+ Aa^2+Bab+Cb^2+Da+Eb.}
Comparing the coefficients of the quadratic parts and linear parts of \eqref{0442}, we have
\al{\label{0452}\chi(p^m,A,B,C,D,E)=\chi(p^m,A,B,C,2Aa+Bb+D, 2Cb+Ba+E).
}
\end{enumerate}

In Section \ref{sectionminor1} we will encounter the exponential sum 
\al{\mathcal{S}_\gamma(q,u,r,\xi,\zeta)=\frac{1}{q^2}\sum_{x_0,y_0(q)}e_q\left(r\ff_{M\gamma}\left(\BB ux_0+uu^{*},\BB uy_0\right)+x_0\xi+y_0\zeta\right).
}

Write $$\ff_{M\gamma}(x,y)=\tilde{\ff}_{M\gamma}(x,y) + \fd_\gamma = A'' x^2+B'' xy+C'' y^2+\fd_\gamma.$$  The quadratic form has discriminant $\Delta\fd_\gamma^2$.  We assume that $M$ is integral, so that $\ff_{M\gamma}$ is primitive and integral by Lemma \ref{lem:primitive}. If $M$ is not integral, then one needs to multiply the curvature formula by a universal constant, to obtain integrality.  By Lemma \ref{lem:primitive}, the gcd of the coefficients of $\ff_{M\gamma}$ after this normalization is bounded for all $\gamma$, and consequently all the estimates from Lemma \ref{1121}, \ref{0856}, \ref{0857} stand, up to a constant factor.  \par
We first give a bound for $\mathcal{S}_\gamma(q,u,r,\xi,\zeta)$:
\begin{lem}\label{1121}
 Assume that $(r,q)=1$.   Then
  \aln{\vert S_{\gamma}(q,u,r,\xi,\zeta)\vert \leq  \frac{|\Delta|^{1/2}u^2L^2(q,\fd_\gamma^2)}{q}.
  }
\end{lem}
\begin{proof}[Proof of Lemma \ref{1121} for $q$ odd]
  For the proof when $q$ is even, see the discussion at the end of this section.

First we consider the case $q=p^m$.  If $p^m\mid u^2L^2$, then we trivially bound $|\mathcal S_\gamma|\leq 1$ and we automatically get the lemma.  We thus assume $p^m\nmid u^2L^2$, then we apply the second case of Lemma \ref{0626} to analyze $\mathcal S_\gamma$. 

We write
\al{\nonumber\label{0554}\mathcal{S}_\gamma(p^m,u,r,\xi,\zeta)=&\frac{1}{p^{2m}}\sum_{x_0,y_0(p^m)}e_{p^m}\left(r\ff_{M\gamma}\left(\BB ux_0+uu^{*},\BB uy_0\right)+x_0\xi+y_0\zeta\right)\\
 \nonumber=&\frac{1}{p^{2m}}\sum_{x_0,y_0(p^m)}e_{p^m}(rL^2u^2(A'' x_0^2+B''x_0y_0+C''y_0^2)\\
 &+(2rA''Lu^2u^{*}+\xi)x_0+(rB''Lu^2u^*+\zeta)y_0+ru^2A''{u^*}^2+r\fd_\gamma)
}

Therefore,  having the primitivity of $\ff_{M\gamma}$ in mind and applying Lemma \ref{0626} to  \eqref{0554}, we obtain (here $p^{k_g} = (p^m, u^2L^2)$):
\begin{align}
  \label{0839}
\nonumber&\mathcal{S}_\gamma(p^m,u,r,\xi,\zeta)\nonumber \\
= &\frac{1}{p^m}
e_{p^m}(r\fd_{\gamma} + ru^2A''u^*{}^2)
(p^m,u^2L^2)^{\frac{1}{2}}
(p^m,u^2L^2\fd_\gamma^2\Delta)^{\frac{1}{2}}\nonumber \\
&i^{\epsilon\left(\frac{p^m}{(p^m,u^2L^2)}\right)}
i^{\epsilon\left(\frac{p^m}{(p^m,u^2L^2\fd_\gamma^2\Delta)}\right)}
\cdot(-1)^{\upsilon(ru^2L^2\ff_{M\gamma})} \nonumber \\
&\cdot e_{p^m}\left(
\frac{
  A''(rB''Lu^2u^* + \zeta)^2
  -B''( 2rA''Lu^2u^* + \xi)(rB''Lu^2u^*+\zeta)
  +C''(2rA''Lu^2u^* + \xi)^2
}
{
  rL^2u^2( B''^2 - 4 A''C'')
}
\right)\nonumber \\
&\cdot \chi\left(p^m,ru^2L^2A'',ru^2L^2B'',ru^2L^2C'', 2rA''Lu^2u^*+\xi, rB''Lu^2u^* + \zeta \right) \nonumber \\
= &\frac{1}{p^m}
e_{p^m}\left(r\fd_{\gamma} - \frac{u^*\xi}{L}\right)
(p^m,u^2L^2)^{\frac{1}{2}}
(p^m,u^2L^2\fd_\gamma^2\Delta)^{\frac{1}{2}}\nonumber \\
&i^{\epsilon\left(\frac{p^m}{(p^m,u^2L^2)}\right)}
i^{\epsilon\left(\frac{p^m}{(p^m,u^2L^2\fd_\gamma^2\Delta)}\right)}
\cdot(-1)^{\upsilon(ru^2L^2\ff_{M\gamma})}
\cdot e_{p^m}\left(\frac{\tilde{\ff}_{M\gamma}(\zeta,-\xi)}{ru^2L^2\Delta\fd_\gamma^2}\right)\nonumber \\
&\cdot \chi\left(p^m,ru^2L^2A'',ru^2L^2B'',ru^2L^2C'', 2rA''Lu^2u^*+\xi, rB''Lu^2u^* + \zeta \right)
\end{align}

From \eqref{0839} we thus have
$$| S_\gamma(p^m,u,r,\xi,\zeta) |\leq \frac{1}{p^m}  (p^m,u^2L^2)^{\frac{1}{2}}  (p^m,u^2L^2\fd_\gamma^2\Delta)^{\frac{1}{2}}.$$
Using the multiplicativity of $\mathcal{S}_\gamma$, we obtain 
\al{\label{1024}\nonumber
|\mathcal{S}_\gamma(q,u,r,\xi,\zeta)|&\leq \prod_{p_i^{n_i}\Vert q}\frac{(p^{n_i},u^2L^2)(p^{n_i},u^2L^2\fd_\gamma^2\Delta)^{\frac{1}{2}}}{p^{n_i}}\\&\leq \frac{(q,u^2L^2)^{\frac{1}{2}}(q,u^2L^2\fd_\gamma^2\Delta)^{\frac{1}{2}}}{q}\leq \frac{|\Delta|^{1/2}u^2L^2(q,\fd_\gamma^2)}{q}.
}
\end{proof}

We will also encounter a certain average of such sums.  
Let 
\al{\label{1113}\mathcal{S}(q,u,\gamma,\xi,\zeta,\gamma',\xi',\zeta')=\sideset{}'\sum_{r(q)}\mathcal{S}_{\gamma}(q,u,r,\xi,\zeta)\overline{\mathcal{S}_{\gamma'}(q,u,r,\xi',\zeta')}}

Set $q=p^m$.  From \eqref{0839}, we can write 
\al{\eqref{1113}=S_1\cdot S_2,}
with $S_1$ and $S_2$ as follows.  The factor $S_1$ consists of factors not involving $r$:
\al{\label{0558}\nonumber S_1=&\frac{1}{p^{2m}}(p^m,u^2L^2)(p^m,u^2L^2\fd_\gamma^2\Delta)^{\frac{1}{2}}(p^m,u^2L^2\fd_{\gamma'}^2\Delta)^{\frac{1}{2}}i^{2\epsilon\left(\frac{p^m}{(p^m,u^2L^2)}\right)} \\& i^{\epsilon\left(\frac{p^m}{(p^m,u^2L^2\fd_\gamma^2\Delta)}\right)}i^{\epsilon\left(\frac{p^m}{(p^m,u^2L^2\fd_{\gamma'}^2\Delta)}\right)} 
e\left(\frac{u^*(\xi'-\xi)}{L}\right)}

For $S_2$, we have  
\al{\label{0112}\nonumber S_2=&\sideset{}'\sum_{r(p^m)}e_{p^m}(r(\fd_\gamma-\fd_{\gamma'}))
 (-1)^{\upsilon(ru^2L^2\ff_{M\gamma})}
 (-1)^{\upsilon(ru^2L^2\ff_{M\gamma'})}
\cdot e_{p^m}\left(\frac{\tilde{\ff}_{M\gamma}(\zeta,-\xi)}{ru^2L^2\Delta\fd_\gamma^2}\right) \\
&e_{p^m}\left(-\frac{\tilde{\ff}_{M\gamma'}(\zeta',-\xi')}{ru^2L^2\Delta\fd_{\gamma'}^2}\right)\cdot \chi(*)
}
We can bound $S_1$ directly from \eqref{0558}:
\al{\label{0513}\vert S_1\vert\leq \frac{(p^m,u^2L^2)(p^m,u^2L^2\fd_\gamma^2\Delta)^{\frac{1}{2}}(p^m,u^2L^2\fd_{\gamma'}^2\Delta)^{\frac{1}{2}}}{p^{2m}} .
}

We note that $(-1)^{\upsilon(ru^2L^2\ff_{M\gamma})},(-1)^{\upsilon(ru^2L^2\ff_{M\gamma'})}$ are multiplicative over $r$.  
Moreover, from \eqref{0451} and \eqref{0452}, we observe that with all other parameters fixed, $\chi$ is a periodic function over $r$ with period dividing $(L,p^m)$.  Therefore the function $\chi(*)$ can be viewed as a function on $(\ZZ/(p^m,L)\ZZ)^*$ bounded by $1$, so can be written as at most $(p^m,L)$ linearly combined multiplicative characters on $\ZZ/p^m\ZZ$ with coefficients bounded by 1. Therefore, the factor $S_2$ is a combination of at most $(p^m,L)$ Kloosterman-Sali\'e sums.

If $\fd_\gamma\neq\fd_{\gamma'}$, applying Kloosterman's elementary $3/4$ bound for this type of sum (Lemma 3.4.1, \cite{XZ14}), we obtain
\al{\label{0514}\vert S_2\vert \ll (p^m,L) p^{\frac{3}{4}m+\epsilon}(p^m, \fd_\gamma-\fd_{\gamma'})^{\frac{1}{4}}.}

If $\fd_\gamma=\fd_{\gamma'}$ but $\ff_{M\gamma}(\zeta,-\xi)\neq \ff_{M\gamma'}(\zeta',-\xi')$, then we can use the last two factors in the summand of \eqref{0112} to obtain a bound for $S_2$.  It can be checked that if $S_2\neq0$, then the condition that $\chi(r;*)=1$ in \eqref{0112} leads to 
$$\textrm{deg}_{p^m}\left(\frac{\ff_{M\gamma}(\zeta,-\xi)}{ru^2L^2\Delta\fd_\gamma^2}\right),\textrm{deg}_{p^m}\left(\frac{\ff_{M\gamma'}(\zeta',-\xi')}{ru^2L^2\Delta\fd_{\gamma'}^2}\right)\geq 0. $$

Therefore, the elementary Kloosterman $3/4$ bound in this case gives 
\al{\label{0512}
\vert S_2\vert \ll  (p^m,L)p^{\frac{3}{4}m+\epsilon}(p^m,\ff_{M\gamma}(\zeta,-\xi)-\ff_{M\gamma'}(\zeta',-\xi'))^{\frac{1}{4}}.
}

Collecting \eqref{0513}, \eqref{0514}, \eqref{0512}, using the multiplicativity of $\mathcal{S}(q,u,r,\gamma,\xi,\zeta,\gamma',\xi',\zeta')$, and absorbing $\Delta,L$ in the $\ll$ relation,  we obtain the following two lemmas in the case $q$ is odd.

\begin{lem}\label{0856}  
  If $\fd_\gamma\neq\fd_{\gamma'}$, then \aln{\vert\mathcal{S}(q,u,r,\gamma,\xi,\zeta,\gamma',\xi',\zeta')\vert \ll u^4q^{-\frac{5}{4}+\epsilon}(q,\fd_\gamma-\fd_{\gamma'})^{\frac{1}{4}}(q,\fd_\gamma^2)^{\frac{1}{2}}(q,\fd_{\gamma'}^2)^{\frac{1}{2}}.
}
\end{lem}

\begin{lem}\label{0857}
If $\fd_\gamma=\fd_{\gamma'}$ and $\ff_{M\gamma}(\zeta,-\xi)\neq \ff_{M\gamma'}(\zeta',-\xi')$, then

\aln{\vert\mathcal{S}(q,u,r,\gamma,\xi,\zeta,\gamma',\xi',\zeta')\vert
 \ll & u^4q^{-\frac{5}{4}+\epsilon}\vert\ff_{M\gamma}(\zeta,-\xi)-\ff_{M\gamma'}(\zeta',-\xi'  )\vert^{\frac{1}{4}}(q,\fd_\gamma^2).
}
\end{lem}

We briefly explain how to extend Lemmas \ref{0343} through \ref{0857} when $q$ is even. It is enough to consider $q=2^m$ by multiplicativity.  The extra complication arises in Lemma \ref{0343} when we complete squares for some exponential sums (e.g., $\sum_{x=0}^7e_{8}(x^2+x)$):  we encounter certain ``restricted" Gauss sums, meaning the sum index is restricted to certain congruence classes mod 2.  This slightly alters the statement of Lemma \ref{0343} for $q = 2^m$.  We can handle this by writing an indicator function of the allowed congruence classes.  In Lemmas \ref{1121}, \ref{0856} and \ref{0857}, we can handle the extra indicator function by writing it as a linear combination of two additive characters to the modulus 2.  We obtain a linear combination of more Kloosterman-Sali\'{e} sums in Lemmas \ref{0856} and \ref{0857}, and this eventually gives an extra constant factor to the bound on $|S_2|$.  The rest of the proof is the same.

\subsection{\label{sectionminor1} Minor arc analysis, part I}

We begin by estimating $\II_1$.  First  we take the Fourier transform of $\rnu$ (defined at \eqref{0922}):

\al{\label{0549}\hrnu(\theta)=\sum_{\substack{u<U \\ (u,L)=1}}\mu(u)\sum_{\gamma\in\FF_T}\mathcal{R}_{u,\gamma}(\theta),
}
where 
\al{
\mathcal{R}_{u,\gamma}(\theta)=\sum_{x,y\in\ZZ} \psi\left(\frac{\BB ux +uu^{*}}{X}\right)\psi\left(\frac{\BB uy}{X}\right)e(\ff_{M\gamma}(\BB ux+uu^{*},{ \BB uy})\theta ).
}

We will first give an $L^{\infty}$ bound for $\mathcal{R}_{u,\gamma}$ (see \eqref{1205}). \\

Write $\theta=\frac{r}{q}+\beta$ and rearrange the order of $x,y$ according to the congruence classes mod $q$: 
\al{\label{0844}\nonumber
\mathcal{R}_{u,\gamma}\left(\frac{r}{q}+\beta\right)=&\sum_{x_0,y_0(q)}e\left(\ff_{M\gamma}(\BB ux_0+uu^{*}, \BB uy_0)\frac{r}{q}\right)\\
&\cdot\left[\sum_{\substack{x\equiv x_0(q)\\y\equiv y_0(q)} } \psi\left(\frac{\BB ux +uu^{*}}{X}\right)\psi\left(\frac{\BB uy}{X}\right) e\left(\ff_{M\gamma}\left(\BB ux+uu^{*},\BB uy\right)\beta\right)\right]
}
Applying Poisson summation to the $x,y$ sum in the bracket $[\cdot]$, we obtain
\al{\label{0845}\nonumber
[\cdot]=&\sum_{\xi,\zeta\in\ZZ}\int_{-\infty}^{\infty} \int_{-\infty}^{\infty}\psi\left(\frac{\BB u(x_0+qx)+uu^{*}}{X}\right)\psi\left(\frac{\BB u(y_0+qy)}{X}\right)\\
\nonumber& \cdot e\left(\ff_{M\gamma}\left(\BB u(x_0+qx)+uu^{*},\BB u(y_0+qy)\right)\beta\right)e(-x\xi-y\zeta)dxdy     \\
\nonumber
=&\frac{X^2}{q^2\BB^2 u^2}\sum_{\xi,\zeta\in\ZZ}\int_{-\infty}^{\infty} \int_{-\infty}^{\infty}\psi(x)\psi(y)e\left(\ff_{M\gamma}\left(Xx,Xy\right)\beta-\frac{X\xi}{qu\BB}x-\frac{X\zeta}{qu\BB} y\right)e\left(\frac{u^*\xi}{\BB q}\right)dxdy\cdot \\
&\cdot e_q(x_0\xi+y_0\zeta)
}
Plugging \eqref{0845} back into \eqref{0844}, we have 

\al{\label{0543}
\mathcal{R}_{u,\gamma}\left(\frac{r}{q}+\beta\right)=\frac{X^2}{\BB^2 u^2}\sum_{\xi,\zeta\in\ZZ}\mathcal{S}_\gamma(q,u,r,\xi,\zeta)\mathcal{J}_\gamma(\beta; q,u,\xi,\zeta),
}
where 
\al{\mathcal{S}_\gamma(q,u,r,\xi,\zeta)=\frac{1}{q^2}\sum_{x_0,y_0(q)}e_q\left(r\ff_{M\gamma}\left(\BB ux+uu^{*}, \BB uy_0\right)+x_0\xi+y_0\zeta\right)
}
and 
\al{\mathcal{J}_\gamma(\beta; q,u,\xi,\zeta)=\int_{-\infty}^{\infty} \int_{-\infty}^{\infty}\psi(x)\psi(y)e\left(\ff_{M\gamma}\left(Xx,Xy\right)\beta-\frac{X\xi}{qu\BB}x-\frac{X\zeta}{qu\BB} y\right)e\left(\frac{u^*\xi}{\BB q}\right)dxdy
}

Note that the sum in (\ref{0543}) is principally supported on a few terms, since the $\mathcal {J}_\gamma$ term decays quickly.  We will use non-stationary and stationary phase methods to give bounds for the $\mathcal{J}_\gamma$ terms.  We review the statements here, for reference.

\begin{prop}[\cite{Zh14}, Page 24, Non-stationary phase]
  Let $\phi$ be a smooth compactly supported function on $(-\infty,\infty)$ and $f$ be a function which, in the support of $\phi$, satisfies 
  \begin{enumerate}
    \item $|f'(x)| > A > 0$,
    \item $A \ge |f^{(2)}(x)|, \ldots, |f^{(n)}(x)|$.
  \end{enumerate}
  Then
  \[
    \int_{-\infty}^{\infty} \phi(x) e(f(x))dx \ll_{\phi,N} A^{-N}.
  \]
\end{prop}

\begin{prop}[\cite{Zh14}, Page 25, Stationary phase]
Let $f$ be a quadratic polynomial of two variables $x$ and $y$ whose homogeneous part has discriminant $-D$ with $D > 0$.  Let $\phi(x,y)$ be a smooth compactly supported function on $\RR^2$, then
\[
  \int_{-\infty}^{\infty} 
  \int_{-\infty}^{\infty}
  \phi(x,y) e(f(x,y))dxdy
  \ll_\phi
  \frac{1}{\sqrt{D}}.
\]
\end{prop}

We apply the non-stationary phase to $\mathcal{J}_\gamma$.  We can obtain a bound $A$ as required in the statement by taking
\[
  \frac{X\xi}{qu} \text{ or } \frac{X\zeta}{qu} \gg T^2X^2|\beta|.
\]
(Note that the discriminant of $\mathfrak{f}_{M\gamma}$ is bounded above by $T^4$.)  Using the former for example, the value of $A$ is then $XU/quL$, which is $>1$ since $u < U$, $q < Q_0$ and by \eqref{eqn:growing-par} and \eqref{0737}.

Therefore,  the main contribution of the $\xi,\zeta$ sum in \eqref{0543} comes from the $\xi,\zeta$ terms such that $\frac{X\xi}{qu}\ll T^2X^2|\beta|$ and $\frac{X\zeta}{qu}\ll T^2X^2|\beta|$, or in other words the terms $\xi,\zeta$ such that
$$\xi,\zeta\ll quT^2X|\beta|\ll uT^2X/J =  u<U,$$ where we used $|\beta|\leq \frac{1}{qJ}$ and $J=T^2X$ (by \eqref{eqn:setJ}). \\

For the terms $\xi,\zeta\ll u$, we have an upper bound for $\mathcal{J}_\gamma$ using the stationary phase:
\al{\label{0600}|\mathcal{J}_\gamma(\beta; q,u,\xi,\zeta)|\ll \min\left\{1,\frac{1}{T^2X^2|\beta|}\right\}
}
Lemma \ref{1121} and \eqref{0600} together lead to a bound for $\mathcal{R}_{u,\gamma}\left(\frac{r}{q}+\beta\right)$ and hence for $\hrnu\left(\frac{r}{q}+\beta\right)$:
\al{\label{1205}
\left\vert\hrnu\left(\frac{r}{q}+\beta\right)\right\vert\ll \frac{T^{2\delta-2}U}{|\beta|}.
}
Now we are ready to give an estimate for $\mathcal{I}_1$.  We rewrite $\mathcal{I}_1$ as
\al{\mathcal{I}_1=\sum_{q<Q_0}\sideset{}'\sum_{r(q)}\int_{-\frac{1}{qJ}}^{\frac{1}{qJ}}\left\vert\left(1-\Tf\left(\frac{r}{q}+\beta\right)\right)\hrnu\left(\frac{r}{q}+\beta\right)\right\vert^2d\beta.
}

We now split the integral $\int_{-\frac{1}{qJ}}^{\frac{1}{qJ}}$ above into three parts $\int_{-\frac{K_0}{N}}^{\frac{K_0}{N}}$, $\int_{\frac{K_0}{N}}^{\frac{1}{qJ}}$ and $\int_{-\frac{1}{qJ}}^{-\frac{K_0}{N}}$.  For each integral we use \eqref{1205} to bound the $\hrnu$ term.  In the first integral, we use $$\left\vert\left(1-\Tf\left(\frac{r}{q}+\beta\right)\right)\right\vert^2=\frac{N^2\beta^2}{K_0^2},$$ and in the second and third integral, we trivially bound $\left\vert\left(1-\Tf\left(\frac{r}{q}+\beta\right)\right)\right\vert^2$ above by 1. 
 Altogether, we have 
 \begin{lem}
 \al{
 \mathcal{I}_1\ll \frac{U^2T^{4\delta-4}NQ_0^2}{K_0}
  }
\end{lem}
Since $K_0\gg Q_0^3\gg Q_0^2 U^2N^{\epsilon}$ (see \eqref{1145}), we have $\mathcal{I}_1\ll T^{4\delta-4}N^{1-\epsilon}$.

\subsection{\label{sectionminor2} Minor arc analysis, part II}
In this section we give an upper bound for 
\al{\label{0916}\mathcal{I}_Q=\sum_{Q<q\leq 2Q}\int_{-\frac{1}{qJ}}^{\frac{1}{qJ}}\sideset{}'\sum_{r(q)}\left|\hrnu\left(\frac{r}{q}+\beta\right)\right|^2d\beta}
for $Q_0 < Q < X$ and show the following.
\begin{lem}\label{lemmaI2}\aln{\mathcal{I}_2\ll T^{4\delta-4}N^{1-\eta}}
\end{lem}
\begin{proof}
Going back to \eqref{0549}, we apply Cauchy-Schwartz to the $u$ sum to get an upper bound for $\hrnu(\frac{r}{q}+\beta)$:
\al{\label{cauchy-schwartz}
  \left\vert\hrnu\left(\frac{r}{q}+\beta\right)\right\vert^2\ll U\sum_{u<U}\sum_{\gamma\in\FF_T}\sum_{\gamma'\in\FF_T}\mathcal{R}_{u,\gamma}\left(\frac{r}{q}+\beta\right)\overline{\mathcal{R}_{u,\gamma'}\left(\frac{r}{q}+\beta\right)}
}
Using \eqref{0543}, we obtain 
\al{\label{0731}\nonumber\sideset{}'\sum_{r(q)}^{}{}\left\vert\hrnu\left(\frac{r}{q}+\beta\right)\right\vert^2\ll &U\sum_{u<U}\frac{X^4}{u^4}\sum_{\xi,\zeta\in\ZZ}\sum_{\xi',\zeta'\in\ZZ}\sum_{\gamma\in\FF_T}\sum_{\gamma'\in\FF_T}\mathcal{S}(q,u,\gamma,\xi,\zeta,\gamma',\xi',\zeta')\\&\cdot \mathcal{J}_\gamma(\beta;q,u,\xi,\zeta)\overline{\mathcal{J}_{\gamma'}(\beta;q,u,\xi',\zeta')}
}
By the non-stationary phase, the main contribution to \eqref{0731} comes from the terms $\xi,\zeta,\xi',\zeta'\ll U$, and for these terms, we have 
\al{\label{0858}\mathcal{J}_\gamma(\beta;q,u,\xi,\zeta)\overline{\mathcal{J}_{\gamma'}(\beta;q,u,\xi',\zeta')}\ll\min\left\{1,\frac{1}{T^4X^4\beta^2}\right\}
}
Using Lemma \ref{0856} together with \eqref{0858}, we obtain
\al{\label{0915}\sideset{}'\sum_{r(q)}\left\vert\hrnu\left(\frac{r}{q}+\beta\right)\right\vert^2\ll \sum_{\gamma\in\FF_T}\sum_{\gamma'\in\FF_T}U^6X^4q^{-\frac{5}{4}+\epsilon}(q,\fd_\gamma-\fd_{\gamma'})^{\frac{1}{4}}(q,\fd_\gamma^2)^{\frac12}(q,\fd_{\gamma'}^2)^\frac12\min\left\{1,\frac{1}{T^4X^4\beta^2}\right\}
}
Observe that $qJ \le T^2X^2$, so that
\al{\label{0859}\int_{-\frac{1}{qJ}}^{\frac{1}{qJ}}\min\left\{1,\frac{1}{T^4X^4\beta^2}\right\}d\beta\ll \frac{1}{T^{2}X^2}  }
Plug  \eqref{0915} and \eqref{0859} into \eqref{0916}, and we obtain
\al{\label{0946}\mathcal{I}_Q\ll\frac{N^{\epsilon}U^6X^2}{T^2Q^{\frac{5}{4}}}\sum_{Q<q\leq 2Q}\sum_{\gamma\in\FF_T}\sum_{\gamma'\in\FF_T}(q,\fd_\gamma-\fd_{\gamma'})^{\frac{1}{4}}(q,\fd_\gamma^2)^{1/2}(q,\fd_{\gamma'}^2)^{1/2}
}
 We split \eqref{0946} into two parts $\mathcal{I}_Q^{(=)}$ and $\mathcal{I}_Q^{(\neq)}$ according to whether $\fd_\gamma=\fd_{\gamma'}$ or not.
 We first estimate $\mathcal{I}_Q^{(=)}$:
\al{\label{0756}\nonumber\mci_Q^{(=)}\ll&\frac{N^{\epsilon}U^6X^2}{QT^2}\sum_{Q<q\leq 2Q}\sum_{\gamma\in\FF_T}\sum_{\substack{\gamma'\in\FF_T\\   \fd_{\gamma'}=\fd_{\gamma} }}(q,\fd_\gamma^2)\\
\nonumber\ll& \frac{N^{\epsilon}U^6X^2}{QT^2}\sum_{\gamma\in\FF_T}\sum_{\substack{\gamma'\in\FF_T\\ \fd_{\gamma'}=\fd_{\gamma} }}\sum_{d|\fd_\gamma^2}d\sum_{Q<q\leq 2Q}\bd{1}\{d|q\}\\
\nonumber\ll& \frac{N^{\epsilon}U^6X^2}{QT^2}\sum_{\gamma\in\FF_T}\sum_{\substack{\gamma'\in\FF_T\\ \fd_{\gamma'}=\fd_{\gamma} }}\sum_{d|\fd_\gamma^2}Q\\
\nonumber\ll& \frac{N^{\epsilon}U^6X^2T^{2\delta}}{T^2}\sum_{\substack{\gamma'\in\FF_T\\ \fd_{\gamma'}=\fd_{\gamma} }}1\\
\ll& \frac{N^{1+\epsilon}U^6T^{4\delta-4}}{T^{\eta_0}},
 }
 where we have bounded the number of divisors of $\fd_\gamma^2$ by $N^\epsilon$ and in the last step we used Lemma \ref{bk1} to estimate the sum (using modulus $T$ in the statement of the Lemma), with reference to \eqref{eqn:growing-par}.

Now we estimate $\mathcal{I}_Q^{(\neq)}$.  We introduce a new parameter $H$ and we further split $\mathcal{I}_Q^{(\neq)}$ into two $\mathcal{I}_Q^{(\neq,\geq)}$ and $\mathcal{I}_Q^{(\neq,<)}$ according to $(\fd_\gamma,\fd_{\gamma'})\geq H$ or not.

We first estimate $\mathcal{I}_Q^{(\neq,\geq)}$.  Recall \eqref{0915}.  Then we have

\al{\nonumber
\mathcal{I}_Q^{(\neq,\geq)}\ll &\frac{N^{\epsilon}U^6X^2}{QT^2}\sum_{Q<q\leq 2Q}\sum_{\gamma\in\FF_T}\sum_{\substack{\gamma'\in\FF_T\\(\fd_{\gamma'},\fd_{\gamma})\geq H }}(q,\fd_\gamma^2)^{\frac{1}{2}}(q,\fd_{\gamma'}^2)^{\frac{1}{2}}\\
 \label{1147}\ll &\frac{N^{\epsilon}U^6X^2}{QT^2}\sum_{\gamma\in\FF_T}\sum_{\substack{h|\fd_\gamma\\h\geq H}}\sum_{\substack{\gamma'\in\FF_T\\\fd_{\gamma'}\equiv\fd_{\gamma}(h) }}\sum_{q_1|\fd_\gamma^2}\sum_{q_2|\fd_{\gamma'}^2}(q_1q_2)^{\frac{1}{2}}\sum_{Q<q\leq2Q}\bd{1}\{[q_1,q_2]|q\}.
}

Notice that $[q_1,q_2]\geq (q_1q_2)^{\frac{1}{2}}$.  Therefore,
\al{\label{1151}
\eqref{1147}\ll \frac{N^{\epsilon}U^6X^2}{T^2}\sum_{\gamma\in\FF_T}\sum_{\substack{h|\fd_\gamma\\h\geq H}}\sum_{\substack{\gamma'\in\FF_T\\\fd_{\gamma'}\equiv\fd_{\gamma}(h) }}1}
Again using Lemma \ref{bk1}, we have 
\al{\label{1155}\eqref{1151}\ll \frac{N^{\epsilon}U^6X^2}{T^2H^{\eta_0}}\sum_{\gamma\in\FF_T}\sum_{\substack{h|\fd_\gamma\\h\geq H}}1\ll \frac{N^{1+\epsilon}U^6T^{4\delta-4}}{H^{\eta_0}}
}
Now we estimate $\mathcal{I}_Q^{(\neq,<)}$.  
Using \eqref{0915} and replacing $(q,\fd_\gamma^2)^{\frac{1}{2}},(q,\fd_{\gamma'}^2)^{\frac{1}{2}}$ by $(q,\fd_\gamma),(q,\fd_{\gamma'})$, we have
\al{\label{0233}\nonumber\mathcal{I}_Q^{(\neq,<)}\ll& \frac{N^{\epsilon}U^6X^2}{T^2Q^{\frac{5}{4}}}\sum_{Q<q\leq 2Q}\sum_{\gamma\in\FF_T}\sum_{\substack{\gamma'\in\FF_T\\(\fd_{\gamma'},\fd_{\gamma})< H }}(q,\fd_\gamma-\fd_{\gamma'})^{\frac{1}{4}}(q,\fd_\gamma)(q,\fd_{\gamma'})\\
\ll& \frac{N^{\epsilon}U^6X^2}{T^2Q^{\frac{5}{4}}}\sum_{\gamma\in\FF_T} \sum_{d_1|\fd_\gamma}\sum_{\substack{d_3\leq 2Q}}\sum_{\substack{\gamma'\in\FF_T\\\fd_\gamma'\equiv\fd_\gamma(d_3)}} \sum_{d_2|\fd_{\gamma'}} d_1d_2d_3^{\frac{1}{4}}\sum_{Q<q\leq 2Q}\bd{1}\{[d_1,d_2,d_3]|q\}
}
Writing $h=(\fd_\gamma,\fd_{\gamma'})$, then $\frac{d_1}{h},\frac{d_2}{h},\frac{d_3}{h}$ are mutually relatively prime.  Since $h<H$, we have the estimate 
\al{\sum_{Q<q\leq 2Q}\bd{1}\{[d_1,d_2,d_3]|q\}\leq \frac{QH^2}{d_1d_2d_3}  }
Therefore,  
\al{\label{0755}\nonumber
\eqref{0233}\ll & \frac{N^{\epsilon}U^6X^2H^2}{T^2Q^{\frac{1}{4}}}\sum_{\gamma\in\FF_T} \sum_{d_1|\fd_\gamma}\sum_{\substack{d_3\leq 2Q}}\sum_{\substack{\gamma'\in\FF_T\\\fd_\gamma'\equiv\fd_\gamma(d_3)}} \sum_{d_2|\fd_{\gamma'}}d_3^{-\frac{3}{4}}\\
\nonumber\ll &\frac{N^{\epsilon}U^6X^2H^2T^{4\delta}}{T^2Q^{\frac{1}{4}}}\sum_{d_3\leq 2Q}d_3^{-\frac{3}{4}-\eta_0}\\
 \ll &\frac{N^{1+\epsilon}U^6H^2T^{4\delta-4}}{Q^{\eta_0}}
}
To make the terms at \eqref{0756}, \eqref{1155}, \eqref{0755}, and later on at \eqref{0757} $\ll T^{4\delta-4}N^{1-\eta}$ for an appropriate positive $\eta$, we can set
\al{\boxed{\label{1145}H=Q_0^{\frac{\eta_0}{4}},U=H^{\frac{\eta_0}{20}}}
}
Thus we have proven Lemma~\ref{lemmaI2}.
\end{proof}

\subsection{\label{sectionminor3} Minor arc analysis, part III}

In this section we give an upper bound for $\mathcal{I}_Q$ when $Q$ is large, i.e. $X<Q\leq J$.  We keep all notation from the previous sections.  Namely, we show the following.
\begin{lem}\label{lemmaI3}\aln{\mathcal{I}_3\ll T^{4\delta-4}N^{1-\eta}.}
\end{lem}
\begin{proof}
Recall
\al{\label{1051}
\hrnu\left(\frac{r}{q}+\beta\right)=\sum_{\substack{u<U \\ (u,L)=1}}\mu(u)\sum_{\gamma\in\FF_T}\mathcal{R}_{u,\gamma}\left(\frac{r}{q}+\beta\right),
}
where 
\al{\label{1033}
\mathcal{R}_{u,\gamma}\left(\frac{r}{q}+\beta\right)=\sum_{x,y,\in\ZZ}\psi\left(\frac{\BB ux +uu^{*}}{X}\right)\psi\left(\frac{\BB uy}{X }\right)e\left(\ff_{M\gamma}\left(\BB ux+uu^{*},uLy\right)\left(\frac{r}{q}+\beta\right)\right).
}
We insert extra harmonics by writing  $e_q\left(r\ff_{M\gamma}\left(\BB ux+uu^{*},uLy\right)\right)$ into its Fourier expansion:

\al{\label{1032}\nonumber
&e_q\left(r\ff_{M\gamma}\left(\BB ux+uu^{*},uLy\right)\right)\\
\nonumber=&\frac{1}{q^2}\sum_{m(q)}\sum_{n(q)}\sum_{s(q)}\sum_{t(q)}e_q\left(r\ff_{M\gamma}\left(\BB us+uu^{*},\BB ut\right)+sm+tn\right) e\left(-\frac{mx}{q}-\frac{ny}{q}\right)\\=& \sum_{m(q)} \sum_{n(q)} \mathcal{S}_\gamma(q,u,r,m,n)e\left(-\frac{mx}{q}-\frac{ny}{q}\right)
}

Inserting \eqref{1032} into \eqref{1033}, we obtain
\al{
\mathcal{R}_{u,\gamma}    \left(\frac{r}{q}+\beta\right)= \sum_{m(q)} \sum_{n(q)} \mathcal{S}_\gamma(q,u,r,m,n)\lambda_\gamma\left(\beta,X,u,\frac{m}{q},\frac{n}{q}\right)
}
where 
\al{\label{0143}\lambda_\gamma(\beta,X,u,s,t)=\sum_{x,y\in\ZZ}\psi\left(\frac{\BB ux +uu^{*}}{X}\right)\psi\left(\frac{\BB uy}{X}\right)e\left(\ff_{M\gamma}\left(\BB ux+uu^{*},{\BB uy}\right)\beta-sx-ty\right)
}

Now we apply Cauchy-Schwartz to the $u$ sum for $\hrnu$ (see \eqref{cauchy-schwartz}, \eqref{1051}), and insert it back into $\mathcal{I}_Q$ at \eqref{0916}.  We have 
\al{\label{0136}\nonumber\II_Q=&\sum_{Q< q \leq 2Q}\sideset{}'\sum_{r(q)}\int_{-\frac{1}{qJ}}^{\frac{1}{qJ}}\left|\hrnu\left(\frac{r}{q}+\beta\right)\right|^2d\beta\\
\nonumber\ll& U\sum_{u<U}\sum_{Q< q\leq 2Q}\sideset{}'\sum_{r(q)}\int_{-\frac{1}{qJ}}^{\frac{1}{qJ}}\left|\sum_{\gamma\in\FF_T}\mathcal{R}_{u,\gamma}\left(\frac{r}{q}+\beta\right)\right|^2d\beta\\
\nonumber\ll& U\sum_{u<U}\sum_{\gamma\in\FF_T}\sum_{\gamma'\in\FF_T}\sum_{Q< q\leq2Q}\sum_{m,n,m',n'(q)}\left(\sideset{}'\sum_{r(q)}\mathcal{S}_{\gamma}(q,u,r,m,n)\overline{\mathcal{S}_{\gamma'}(q,u,r,m',n')}\right)\\
  &\nonumber \cdot \int_{-\frac{1}{qJ}}^{\frac{1}{qJ}}\lambda_{\gamma}\left(\beta,X,u,\frac{m}{q},\frac{n}{q}\right)\overline{\lambda_{\gamma'}\left(\beta,X,u,\frac{m'}{q},\frac{n'}{q}\right)}d\beta\\
=& \nonumber U\sum_{u<U}\sum_{\gamma\in\FF_T}\sum_{\gamma'\in\FF_T}\sum_{Q< q\leq2Q}\sum_{m,n,m',n'(q)}\mathcal{S}(q,u,\gamma, m,n,\gamma',m',n')\\
& \cdot \int_{-\frac{1}{qJ}}^{\frac{1}{qJ}}\lambda_{\gamma}\left(\beta,X,u,\frac{m}{q},\frac{n}{q}\right)\overline{\lambda_{\gamma'}\left(\beta,X,u,\frac{m'}{q},\frac{n'}{q}\right)}d\beta
}
Applying Poisson summation and non-stationary phase to $\lambda_\gamma$ and $\lambda_{\gamma'}$, we see that the main contribution to \eqref{0136} comes from the terms  $m$,$n$,$m'$,$n'$ $\ll\frac{qu}{X}$.   For these terms, we use the trivial bound:
\al{\label{0148}\left|\lambda_\gamma\left(\beta,X,u,\frac{m}{q},\frac{n}{q}\right)\right|,\left|\lambda_{\gamma'}\left(\beta,X,u,\frac{m'}{q},\frac{n'}{q}\right)\right|\ll \frac{X^2}{u^2}} 

From \eqref{0136} and \eqref{0148} we have
\al{\label{0211}
\II_Q\ll \frac{UX^4}{QJ}\sum_{u<U}\frac{1}{u^4}\sum_{\gamma\in\FF_T}\sum_{\gamma'\in\FF_T}\sum_{Q<q\leq 2Q}\sum_{m,n,m',n'\ll \frac{uq}{X}}\vert\mathcal{S}(q,u,\gamma,m,n,\gamma',m',n')  \vert
}
We split \eqref{0211}  into ${\mathcal{I}_Q^{(=)}}$ and ${\mathcal{I}_Q^{(\neq)}}$ according to whether $\fd_\gamma=\fd_{\gamma'}$ or not:
\al{\label{05121}
{\mathcal{I}_Q^{(=)}}= \frac{UX^4}{QJ}\sum_{u<U}\frac{1}{u^4}\sum_{\gamma\in\FF_T}\sum_{\substack{\gamma'\in\FF_T\\ \fd_{\gamma'}=\fd_\gamma }}\sum_{Q<q\leq 2Q}\sum_{m,n,m',n'\ll \frac{uq}{X}}\vert\mathcal{S}(q,u,\gamma,m,n,\gamma',m',n')  \vert
}
and
\al{
{\mathcal{I}_Q^{(\neq)}}= \frac{UX^4}{QJ}\sum_{u<U}\frac{1}{u^4}\sum_{\gamma\in\FF_T}\sum_{\substack{\gamma'\in\FF_T\\ \fd_{\gamma'}\neq\fd_\gamma }}\sum_{Q<q\leq 2Q}\sum_{m,n,m',n'\ll \frac{uq}{X}}\vert\mathcal{S}(q,u,\gamma,m,n,\gamma',m',n')  \vert
}
We first deal with $\mathcal{I}_Q^{(\neq)}$.  Using Lemma \ref{0856} to bound $|\mathcal{S}|$, we obtain:
\al{\label{1134}
\mathcal{I}_Q^{(\neq)}\ll \frac{UX^4}{QJ}\sum_{u<U}\frac{1}{u^4}\sum_{\gamma\in\FF_T}\sum_{\substack{\gamma'\in\FF_T\\ \fd_{\gamma'}\neq\fd_\gamma }}\sum_{Q<q\leq 2Q}\sum_{m,n,m',n'\ll \frac{uq}{X}}u^4q^{-\frac{5}{4}+\epsilon}(q,\fd_\gamma-\fd_{\gamma'})^{\frac{1}{4}}(q,\fd_\gamma^2)^{\frac{1}{2}}(q,\fd_{\gamma'}^2)^{\frac{1}{2}}
}
Bounding $(q,\fd_\gamma-\fd_{\gamma'})$, $(q,\fd_\gamma^2)$ $(q,\fd_{\gamma'}^2)$ by $T,T^2,T^2$ respectively, we obtain 
\al{
\II_Q^{(\neq)}\ll U^6T^{4\delta+\frac{1}{4}}X^{-1}Q^{\frac{11}{4}+\epsilon}\ll T^{4\delta-4}N^{1+\epsilon}\left(U^6T^{\frac{9}{4}}X^{-3}Q^{\frac{11}{4}}\right).
}
Since $Q\leq J=T^2X$, the term in the parentheses above is $\ll U^2T^{\frac{17}{4}}X^{-\frac{1}{4}}$, and thus we have obtained a significant power saving for $\II_Q^{(\neq)}$. \par
Now we deal with $\II_Q^{(=)}$.  We split $\II_Q^{(=)}$ into $\II_Q^{(=,=)}$ and $\II_Q^{(=,\neq)}$ according to whether $\ff_{M\gamma}(n,-m)=\ff_{M\gamma'}(n',-m')$ or not. We first give an upper bound for $\II_Q^{(=,\neq)}$.  From Lemma \ref{0857},
\al{\label{1135}\nonumber
\II_Q^{(=,\neq)}\ll &\frac{UX^4}{QJ}\sum_{u<U}\frac{1}{u^4}\sum_{m,n,m',n'\ll \frac{uq}{X}
}\sum_{\gamma\in\FF_T}\sum_{\substack{\gamma'\in\FF_T\\ \fd_{\gamma'}=\fd_\gamma \\\ff_{M\gamma}(\xi,-\zeta)\neq\ff_{M\gamma'}(\xi',-\zeta')}}\sum_{Q<q\leq 2Q}u^4q^{-\frac{5}{4}}(q,\fd_\gamma^2)^{\frac{1}{2}}(q,\fd_{\gamma'}^2)^{\frac{1}{2}}\\
&\cdot \vert\ff_{M\gamma}(n,-m)-\ff_{M\gamma'}(n',-m'  )\vert^{\frac{1}{4}}.
}
We bound $(q,\fd_{\gamma}^2), (q,\fd_{\gamma'}^2)$ by $T^2$ and $ \vert\ff_{M\gamma}(n,-m)-\ff_{M\gamma'}(n',-m')\vert$ by $\frac{T^2u^2q^2}{X^2}$, so that we have
\al{\label{0614}
\II_Q^{(=,\neq)}\ll U^{\frac{13}{2}}Q^{\frac{13}{4}+\epsilon}T^{4\delta+\frac{1}{2}}X^{-\frac{3}{2}}\ll T^{4\delta-4}N^{1+\epsilon}(U^{\frac{13}{2}}T^{9}X^{-\frac{1}{4}}),
}
where we have used $Q\leq T^2X$.  Again we obtain a significant power saving from \eqref{0614}.\\

Finally we estimate $\II_Q^{(=,=)}$.  From Lemma \ref{0856} and \eqref{05121}, we have
\al{\label{0803}
\II_Q^{(=,=)}\ll\frac{N^{\epsilon}UX^4}{QJ}\sum_{u<U}\frac{1}{u^4}\sum_{\gamma\in\FF_T}\sum_{Q< q\leq2Q}\sum_{m,n\ll\frac{uq}{X}}\frac{(\fd_\gamma^2,q)}{q}u^4\sum_{\sk{{\gamma'}\in\FF_T\\\fd_{\gamma'}=\fd_{\gamma}}}\sum_{\sk{m',n'\ll \frac{uq}{X}\\ \ff_{M\gamma'}(n',-m')=\ff_{M\gamma}(n,-m)}}1
}

To analyze \eqref{0803}, we introduce the following two lemmata, the proofs of which are minor adaptions of the proofs of Lemma 3.15 and Lemma 3.16 from \cite{Zh14}.
\begin{lem}\label{0832}
 Fix $\gamma \in \FF_T$.  Then we have
 \aln{\sum_{\sk{{\gamma'}\in\FF_T\\\fd_{\gamma'}=\fd_{\gamma}}}\sum_{\sk{m',n'\ll \frac{uq}{X}\\ \ff_{M\gamma'}(n',-m')=\ff_{M\gamma}(n,-m)}}1 \ll N^{\epsilon}\left(\tilde\ff_{M\gamma}(n,-m),\fd_\gamma^2\right)^{\frac{1}{2}}}
\end{lem}

\begin{lem}\label{0833lem} For any $\gamma\in\FF_T$, $d|\fd_\gamma^2$ and any integer $W>0$, we have 
  \aln{\sum_{\substack{m,n\leq W\\ \tilde\ff_{M\gamma}(n,-m)\equiv 0(d)}}1\ll W^2d^{-\frac{1}{2}}+W.}
\end{lem}

We now return to \eqref{0803}. From Lemma \ref{0832} and Lemma \ref{0833lem}, we have
\al{\label{0757}
\nonumber\mathcal{I}_Q^{(=,=)}\ll& \frac{N^{\epsilon}UX^4}{QJ}\sum_{u<U}\sum_{\gamma\in\FF_T}\sum_{Q< q\leq2Q}\sum_{m,n\ll\frac{uq}{X}}\frac{(\fd_\gamma^2,q)}{q}\left(\tilde\ff_{M\gamma}(m,-n),\fd_\gamma^2\right)^{\frac{1}{2}}\\
\nonumber\ll &\frac{N^{\epsilon}UX^4}{QJ}\sum_{u<U}\sum_{\gamma\in\FF_T}\sum_{Q< q\leq2Q}\frac{(\fd_\gamma^2,q)}{q}\sum_{d|\fd_\gamma^2}d^{\frac{1}{2}}\sum_{m,n\ll\frac{uq}{X}}\bd{1}\{\tilde\ff_{M\gamma}(m,-n)\equiv 0(d)\}\\
\nonumber\ll &\frac{N^{\epsilon}UX^4}{QJ}\sum_{u<U}\sum_{\gamma\in\FF_T}\sum_{Q< q\leq2Q}\frac{(\fd_\gamma^2,q)}{q}\sum_{d|\fd_\gamma^2}d^{\frac{1}{2}}\left(\frac{u^2q^2}{X^2d^{\frac{1}{2}}}+\frac{uq}{X}\right)\\
\nonumber \ll &\frac{N^{\epsilon}U^4X^4}{QJ}\sum_{\gamma\in\FF_T}\sum_{Q< q\leq2Q}\frac{(\fd_\gamma^2,q)}{q}\cdot\frac{T^2q}{X}\\
\nonumber\ll & \frac{N^{\epsilon}U^4X^3T^2}{QJ}\sum_{\gamma\in\FF_T}\sum_{d|\fd_\gamma^2}d \sum_{\substack{Q<q\leq 2Q\\ q\equiv0(d)}}1\\
\nonumber\ll&  {N^{\epsilon}U^4X^2T^{2\delta}}\ll N^{1+\epsilon}T^{4\delta-4}(U^4T^{2-2\delta})\\\ll& N^{1-\eta}T^{4\delta-4}
}
Thus we have a power saving here, too.  \\

Put together, \eqref{1134}, \eqref{1135}, \eqref{0757} lead to the desired bound in Lemma~\ref{lemmaI3}.
\end{proof}

With the bounds on $\mathcal I_1$, $\mathcal I_2,$ and $\mathcal I_3$ that we have obtained here, Theorem \ref{0430} and the main Theorem \ref{mainthm} follow.  

\section{Spectral Gap for a Class of Kleinian Groups}\label{spectralgapsec}

In this section, we prove Theorem \ref{mainspectralthm}, which will in particular imply that a familial group $\mathcal{A}$ has a geometric spectral gap. 
Theorem \ref{mainspectralthm} concerns more generally an infinite-covolume, geometrically finite, Zariski dense Kleinian group $\mathcal{A} < \textrm{PSL}_2(\mathcal{O}_K)$ containing a Zariski dense subgroup $\Gamma < \textrm{PSL}_2(\ZZ)$.  

We first simplify the situation by moving to $\textrm{SL}_2$ instead of $\textrm{PSL}_2$.  In particular, if let $\mathcal{A}'$ be the preimage of $\mathcal{A}$ in $\textrm{SL}_2$, then the quotients $\mathcal{A}' \backslash \mathbb{H}$ and $\mathcal{A} \backslash \mathbb{H}$ are the same.  Therefore, their geometric spectral theories agree.  The properties of being geometrically finite, infinite-covolume, Zariski dense and having a Zariski dense surface subgroup are preserved. 

Assume also that $\mathcal{A}$ is not itself contained in $\textrm{SL}_2(\ZZ)$ (in which case it has a spectral gap in the senses described below by \cite{BV12}).

Assume also that $\Gamma$ has a multiplicative structure, in the sense that
for any $q=\prod_{i}p_i^{n_i}$,
$${\Gamma}/{\Gamma(q)}\cong \prod_i{\Gamma}/{\Gamma}(p_i^{n_i}).$$
For, if $\Gamma$ does not have this multiplicative structure, we replace $\Gamma$ by $\widehat{\Gamma} := \Gamma \cap \Lambda$, where $\Lambda$ is a principal congruence subgroup of $\textrm{SL}_2(\ZZ)$, so that $\widehat{\Gamma}$ has a multiplicative structure.  The existence of this subgroup is guaranteed by the strong approximation property.  Then $\widehat{\Gamma}$ still has Zariski closure $\textrm{SL}_2(\RR)$ as it is finite index.

As a byproduct, we prove a version of strong approximation for $\mathcal{A}$, as follows.

\begin{thm}\label{thm:explicit-strong}
 Let $\mathcal A$ and $\Gamma$ be as above.  There exists an integer $P_{\textrm{bad}}$ depending on $\mathcal A$, such that if $q\in\ZZ$, with $q=q_{\textrm{bad}}\cdot q_1$ where $(q_1,P_{\textrm{bad}})=1$, we have

\begin{enumerate}
\item
$$\mathcal A/\mathcal{A}(q)\cong \mathcal{A}/\mathcal A(q_{\textrm{bad}})\times \mathcal{A}/\mathcal A({q_{1}})$$
 \item 
   \aln{
     \mathcal{A}/\mathcal A(q_1)&= {\textrm{SL}_2( \mathcal O_K)}/\textrm{SL}_2(\mathcal O_K)(q_1)\\
     &\cong \prod_{p_i^{n_i}||q_1} \textrm{SL}_2(\mathcal{O}_K)/\textrm{SL}_2(\mathcal O_K)(p_i^{n_i})
   }
\item For each $p|P_{\textrm{bad}}$, there exists $m_p\geq 1$ such that for all $k_p \ge m_p$,
   \[
     \left. \mathcal{A}\left(p^{m_p}\right)\middle/\mathcal{A}\left(p^{k_p}\right) \right.
     =
     \left. \textrm{SL}\left(2,\mathcal{O}_K\right)\left(p^{m_p}\right)\middle/\textrm{SL}\left(2,\mathcal{O}_K\right)\left(p^{k_p}\right) \right..
   \]
     Moreover, we can choose $m_p$ so that $m_p\leq m_p'+\iota_p$, where  $m_p'$ is the smallest positive power $m$ of $p$ such that for all $k_p' \ge m_p'$,
    \[
      \Gamma(p^{m_p'})/\Gamma(p^{k_p'}) = \textrm{SL}_2(\ZZ)(p^{m_p'})/\textrm{SL}_2(\ZZ)(p^{k_p'}).
  \]
and $\iota_p$ is the smallest non-negative integer such that \[
    p^{\iota_p} \mathfrak{sl}(2, \ZZ_p \otimes_\ZZ \mathcal{O}_K) \subset \operatorname{Span}_{\ZZ_p}( \mathcal{A} \cdot \mathfrak{sl}(2,\ZZ_p)  ).
  \]
  In this notation, the action of $\mathcal{A}$ is the restriction of the adjoint action of the Lie group $\textrm{SL}$ on its Lie algebra $\mathfrak{sl}$, i.e. conjugation.
\item If $p|P_{\textrm{bad}}$, then $p$ can only be possibly one of the following:
  \begin{enumerate}
    \item $p=2,3$, or
    \item $p$ is such that $\Gamma/\Gamma(p) \neq \textrm{SL}_2(\ZZ/p\ZZ)$, or
    \item $p$ is a common factor of all curvatures in the associated orbit $\mathcal{A} \cdot \mathbb{P}^1(\RR)$ (after scaling all raw curvatures by $\frac{1}{\sqrt{-\Delta}}$).
  \end{enumerate}
\end{enumerate}
\end{thm}

Note that in the case of a familial group $\mathcal A$ (which is the object of this paper), $\Gamma$ can be taken to be the principal congruence subgroup of $\textrm{SL}_2(\mathbb Z)$ contained in $\mathcal A$, in which case the bad primes of the second kind in the theorem above are simply those dividing the level of this congruence subgroup.

We begin with the definitions of geometric and combinatorial spectral gaps for any Kleinian group $H$. 
Let $\Delta$ be the hyperbolic Laplacian operator associated to the metric $ds^2=\frac{dx^2+dy^2+dz^2}{z^2}$ on $\mathbb{H}^3$:
$$\Delta=-z^2\left(\frac{\partial^2}{\partial x^2}+\frac{\partial^2}{\partial y^2}+\frac{\partial^2}{\partial z^2}\right)+z\frac{\partial}{\partial z}.$$
For any integer $q > 1$, let $H(q)$ denote the kernel of reduction modulo $q$.
The operator $\Delta$ is symmetric and positive definite on $L^2(H(q)\backslash \mathbb{H}^3)$ with the standard inner product.  From Lax-Phillips \cite{LP82}, the discrete spectrum consists of finitely many eigenvalues \[\lambda_0(q)=\delta(2-\delta)<\lambda_1(q)\leq\lambda_2(q)\cdots. \]
We assume $\lambda_0(q)\not=\lambda_1(q)$ in order to be able to define the spectral gap; this is guaranteed in the case that $H=\mathcal A$ by geometric finiteness of $\mathcal A$.  If there exists $\epsilon>0$ independent of $q$ such that $\lambda_1(q)-\lambda_0(q)\geq\epsilon$ for all $q$ then $\epsilon$ is called a geometric spectral gap and $H$ is said to have a \emph{geometric spectral gap}.

We now recall the definition of a combinatorial spectral gap for $H$.  Suppose $H$ has a finite symmetric generating set $S$.
Let \[ \lambda_n'(H,S) \le \cdots \le \lambda_1'(H,S) \le \lambda_0'(H,S) = 1\] denote the eigenvalues of the averaging operator $T_{H,S} := 1 - \Delta_{H,S}/|S|$ where $\Delta_{H,S}$ is the discrete Laplacian operator
\[
        \left( \Delta_{H,S} f \right) (g) = \sum_{h \in S}  \left( f(g) - f( h g ) \right).
\]

We say that $H$ has a \emph{combinatorial spectral gap} if there is a finite symmetric collection of generators $S$ and a positive $\epsilon$ such that 
        \[
	  \lambda_1'(H/H(q), \overline{S} ) < 1 - \epsilon
        \]
	for all positive integers $q$, where $\epsilon$ is independent of $q$ (here $\overline{S}$ denotes the image of $S$ modulo $q$).  Informally, a spectral gap for $H/H(q)$ gives a measure of how quickly a random walk on the Cayley graph of $H/H(q)$ reaches the whole graph.  A spectral gap for $H$ indicates a uniform rate for all $q$.

	We now give an overview of the proof of Theorem \ref{mainspectralthm}.
Let $T$ be an element of $\mathcal{A}$ which does not normalize $\textrm{SL}_2(\RR)$, i.e., $T\not\in \textrm{SL}_2(\RR)\cup i\textrm{SL}_2(\RR)$.  
Write $\Gamma^{'}=T\Gamma T^{-1}$, and let $\mathcal{A}^{'}=\langle \Gamma,\Gamma^{'}\rangle < \mathcal{A}$.  We first prove a combinatorial spectral gap for $\mathcal{A}^{'}$, using ideas similar to those of Varj\'u in the appendix of \cite{BK14}, some of which have also been used by Sarnak in \cite{Sa90}, Shalom in \cite{Sh99}, and Kassabov-Lubotzky-Nikolov in \cite{KLN06}.  We then convert this to a combinatorial spectral gap for $\mathcal{A}$.
Finally, we use the fact that a {combinatorial spectral gap} for $\mathcal{A}$ implies a geometric spectral gap if the Hausdorff dimension of the limit set of $\mathcal A$ is greater than $1$ via a variant of \cite[Theorem 1.2]{BGS11}, which states that geometric and combinatorial spectral gaps co-occur as long as the Hausdorff dimension of the limit set of the group is greater than $1$.  

	\begin{prop}
	  \label{prop:spectral-aprime}
	  $\mathcal{A}'$ has a combinatorial spectral gap.
	\end{prop}

As a Zariski-dense subgroup of $\textrm{SL}_2(\ZZ)$, $\Gamma$ is known to have a spectral gap (see  \cite{BV12}), and therefore so does $\Gamma'$.  We will show that $\mathcal{A}'/\mathcal{A}^{'}(q)$ is made up of a bounded number of copies of $\Gamma/\Gamma(q)$ and $\Gamma'/\Gamma'(q)$, which will imply a spectral gap for $\mathcal{A}^{'}/\mathcal{A}^{'}(q)$.  To be precise, we quote a Lemma of Varj\'u:

	\begin{lem}[{\cite[Lemma A.4]{BK14}}]
	  \label{lem:varju}
        Let $G$ be a finite group with a finite symmetric generating set $S$.  Suppose $G_1, \ldots, G_k < G$, and that for each $g \in G$, there exist $g_i \in G_i$ such that $g = g_1g_2\ldots g_k$.  Then,
        \[
                1 - \lambda_1'(G,S) \ge \min_{1 \le i \le k} 
                \left\{
                        \frac{ |S \cap G_i| }{ |S| }
                        \cdot
                        \frac{ 1 - \lambda_1'(G_i, S \cap G_i) }{2k^2}
                \right\} .
        \]
\end{lem}

As a consequence, we have immediately:
\begin{lem}
  \label{lem:varju2}
Suppose $G$ is a group with a finite symmetric generating set $S$ and a tuple $(G_1, G_2, \ldots, G_k)$ of subgroups of $G$ such that:
\begin{enumerate}
        \item each $G_i$ has a spectral gap;
        \item $S \cap G_i \neq \emptyset$ for each $G_i$;
        \item for each integer $q$, for each $g \in \rho_q(G)$, it is possible to write $g = g_1 g_2 \cdots g_k$ where $g_i \in \rho_q(G_i)$, i.e.
                \[
                        \rho_q(G) = \prod_{i=1}^k \rho_q(G_i).
                \]
\end{enumerate}
		Then $G$ has a combinatorial spectral gap.
	      \end{lem}

	      To verify the hypotheses of Lemma \ref{lem:varju2} for $G = \mathcal{A}'$, we will use $k=2k_0$, and the tuple $(G_1,G_2, \ldots, G_{2k_0}) = (\Gamma, \Gamma', \Gamma, \Gamma', \ldots, \Gamma')$.  Write
$$A_k(q)=\{g_1h_1\cdots g_kh_k: g_1,\cdots g_k \in \Gamma/\Gamma(q), h_1\cdots h_k\in \Gamma^{'}/\Gamma^{'}(q)\}.$$  
Then, for the third hypothesis of Lemma \ref{lem:varju2}, we need to show:

\begin{lem}\label{lem:spectral-aprime-third}
There exists some $k_0$ such that $A_{k_0}(q)=\mathcal{A}^{'}/\mathcal{A}^{'}(q)$ for every $q$.
\end{lem}

Our approach is to break $q$ into prime powers, and prove a universal bound for prime powers for all but finitely many `bad primes'.  We therefore break the proof into two lemmata dealing with the good primes and bad primes, respectively.  The first lemma uses some geometric arguments to construct elements of $\mathcal{A}'/\mathcal{A}'(p^m)$ in terms of $\Gamma$ and $\Gamma'$.  The second lemma works prime-by-prime, and uses the Lie algebra $\mathfrak{sl}_2$ to lift to higher powers of $p$ uniformly.

\begin{lem}
  \label{lem:goodprimes}
  There exists a finite set of primes $\mathcal{S}$ such that, for $p \notin \mathcal{S}$, and for all $m \ge 1$, we have
  \[
    A_{14}(p^m) = \mathcal{A}'/\mathcal{A}'(p^m).
  \]
\end{lem}

\begin{proof}
  Throughout the proof we assume $p \not\in \mathcal{S}$, and we augment $\mathcal{S}$ as necessary while preserving its finiteness.  

  Consider $C_T = T^{-1} \cdot \PP^1(\RR)$.  If $C_T$ is a line, let $\gamma$ be the identity matrix.  Otherwise it is a circle, and we write $r\sqrt{d}$ for its radius, $x_0 + \sqrt{-d}y_0$ for its center, and let
  \[
    \gamma = \begin{pmatrix} 1 & -x_0 \\ 0 & 1 \end{pmatrix}.
  \]
  Note that $x_0, y_0, r$ are rational numbers which may be written with denominator $b$, the curvature of $C_T$ (formulae for these integers in terms of the entries of $T$ are given in \cite[Proposition 3.7]{StangeVis2}).
  Then the intersection points of $\gamma T^{-1} \cdot \PP^1(\RR)$ with the imaginary axis are of the form $\sqrt{-d}s$ where, in the case that $C_T$ is a line, $\sqrt{-d}s$ is the height of the line, and if $C_T$ is a circle, $s = y_0 \pm r$, and $r\sqrt{d}$ is the radius of $T^{-1} \cdot \PP^1(\RR)$.  In any case, choose such an $s$, and remark that $\gamma$ and $s$ are defined over $\QQ$.

Consider reduction modulo $p^m \mathcal{O}_K$ on the projective line: 
\[
  \rho_{p^m}: \PP^1(\mathcal{O}_K) \rightarrow \PP^1(\mathcal{O}_K/(p^m)).
\]
Then the reduction map
\[
  \rho_{p^m}: \textrm{SL}_2(\mathcal{O}_K) \rightarrow \textrm{SL}_2(\mathcal{O}_K/(p^m))
\]
is equivariant with respect to reduction on the projective line.
Let $\mathcal{S}$ contain any primes where
\[
  \rho_{p^m}: \Gamma \rightarrow \textrm{SL}_2(\ZZ/(p^m))
\]
is not surjective for some $m \ge 1$ (there are finitely many such, by strong approximation for $\Gamma$). 
We allow for $p$ to be inert, split, or ramified.

Let $\mathcal{S}$ also contain any primes appearing in denominators of $\gamma$, so that $\rho_{p^m}(\gamma)$ is defined, and has a lift in $\Gamma$.
By expanding $\mathcal{S}$, we may assume $p$ and $s$ are coprime.

Therefore $s$ is invertible modulo $p^m$ and there is a representation $\phi_s: \mathcal{O}_K/(p^m) \rightarrow M(2,\ZZ/(p^m))$ given by
\[
  x + y\sqrt{-d} \mapsto
  \begin{pmatrix}
    x & -y d s \\
    ys^{-1} & x 
  \end{pmatrix}.
\] 
In particular, the eigenvalues of the matrix are $x \pm y \sqrt{-d}$, and the determinant is the norm $N(x+y\sqrt{-d})$.
It has exactly two fixed points modulo $p^m$, namely $\pm s\sqrt{-d}$.  

Let $x$ and $y$ be a solution to $x^2 + dy^2 \equiv 1 \pmod{p^m}$ having $\gcd(xy,p)=1$.  The existence of such is a consequence of an argument with Gauss sums \cite[Exercise 13(v), p. 32]{Cassels}, if $p \ge 5$.  Therefore let $2,3 \in \mathcal{S}$.  
Therefore, $x+y\sqrt{-d}$ is of norm $1$ modulo $p^m$, so that $\phi_s(x+y\sqrt{-d})$ is in $\textrm{SL}_2(\ZZ/(p^m))$, and therefore has a lift, call it $T_0$, in $\Gamma$.  
We guarantee that neither of $(x\pm y\sqrt{-d})^2$ are equivalent to integers modulo $p^m$ (i.e. in the subring $\ZZ/(p^m) \subset \mathcal{O}_K/(p^m)$), since $p \nmid 2xy$ by construction.

Therefore $T\gamma^{-1} T_0 \gamma T^{-1}$, considered modulo $p^m$, has a fixed point in $\ZZ/(p^m)$.  Since $\textrm{SL}_2(\ZZ/(p^m))$ is transitive on $\PP^1(\ZZ/(p^m))$, we can conjugate this fixed point to $\infty$ modulo $p^m$.
Therefore, we find an element $T_1$ in $\Gamma T \Gamma T^{-1} \Gamma$ which fixes $\infty$ modulo $p^m$.

So $T_1$ has the form
\[
  T_1 \equiv \begin{pmatrix}
    a_0 & b \\
    0 & a_1
  \end{pmatrix} \pmod{p^m},
\]
where $a_0, a_1 \in \mathcal{O}_K$, $a_0 a_1 \equiv 1 \pmod{p^m}$.  As $a_0$ and $a_1$ are the eigenvalues of $T_1$ and hence $T_0$, they are $x \pm y \sqrt{-d}$.  In particular, we have arranged that $a_0^2 \not\in \ZZ/(p^m)$. 

Now take
\[
  T_{2,n} = T_1 \begin{pmatrix} 1 & n\\ 0 & 1 \end{pmatrix} T_1^{-1}
  \equiv \begin{pmatrix}
    1 & na_0^2 \\
    0 & 1
  \end{pmatrix}.
\]
We know $a_0^2 \notin \ZZ/(p^m)$ and $a_0^2$ is invertible.  Now, this implies that $a_0^2\ZZ/(p^m) + \ZZ/(p^m) = \mathcal{O}_K/(p^m)$.  This implies that all upper triangular matrices are in $\Gamma T\Gamma T^{-1} \Gamma T \Gamma T^{-1} \Gamma$ modulo $p^m$.

The rest of the proof follows Varj\'u.  Specifically, an exactly analogous argument shows that the lower triangular matrices with $1$'s on the diagonal are also in $\Gamma T\Gamma T^{-1} \Gamma T \Gamma T^{-1} \Gamma$ modulo $p^m$.  Therefore, in $A_{7}(p^m)$ we obtain all elements of the form
\[
  \mat{1&a\\0&1}\cdot\mat{1&0\\b&1}\cdot\mat{1&c\\0&1}=\mat{1+ab&a+c+abc\\b&1+bc}.
\]
This includes any matrix $\gamma$ with lower-left entry not congruent to $0$ modulo $p$, since it is possible to solve for $a,b,c$ modulo $p^m$ in that circumstance.  As this is more than half of the group $\rho_{p^m}(\mathcal{A}')$, the Lemma is proved.
\end{proof}

\begin{lem}
  \label{lem:badprimes}
  Let $p$ be any prime.  Then there exists some positive integers $k_p$ and $m_p$ such that
  \[
    A_{k_p}(p^m)=\mathcal{A}^{'}/\mathcal{A}^{'}(p^m)
  \]
  for all $m \ge m_p$.
\end{lem}

\begin{proof}
  Let $SL_2$ act on $\mathfrak{sl}_2$ via the standard adjoint action of a Lie group on its Lie algebra by conjugation, i.e.
  \[
    SL_2 \times \mathfrak{sl}_2 \rightarrow \mathfrak{sl}_2, \quad g \times v \mapsto g \cdot v := g v g^{-1}.
  \]
  We will first find a $\QQ_p$-basis of $\mathfrak{sl}(2,\QQ_p \otimes_{\QQ} K_d)$ formed of elements from $\mathfrak{sl}(2,\QQ)$ and $\Gamma T~\cdot~\mathfrak{sl}(2,\QQ)$.  Using this basis, we will apply an inductive argument to show that, for all $m \ge m_p$ (where $m_p$ will be defined below), a finite-index subgroup of $\textrm{SL}_2(\mathcal{O}_K)/\textrm{SL}_2(\mathcal{O}_K)(p^m)$, whose index is independent of $m$, is contained in $A_4(p^m)$.

  To find the aforementioned basis, we begin with the standard basis for the real Lie algebra $\mathfrak{sl}(2,\RR)$:
\[
  H=\mat{1&0\\0&-1}, \quad R=\mat{0&1\\0&0}, \quad L=\mat{0&0\\1&0}.
\]
The above is also a $\QQ_p$-basis for $\mathfrak{sl}(2,\QQ_p)$ for any $p$, and a $\QQ$-basis for $\mathfrak{sl}(2,\QQ)$.

First, we remark that $\Gamma(v)$ spans $\textrm{SL}_2(\RR)(v)$ over $\RR$ for any non-zero $v \in \mathfrak{sl}(2,\CC)$.  For, since $\Gamma$ is Zariski dense in $SL_2$, and the adjoint action is Zariski continuous, the Zariski closure of the orbit $\Gamma(v)$ in $\mathfrak{sl}(2,\RR)$ is $\textrm{SL}_2(\RR)(v)$.  

Next, we claim that the orbit $\textrm{SL}_2(\RR)(v)$ must be of real dimension $3$.  This follows from irreducibility of the adjoint action of $\textrm{SL}_2(\RR)$ on $\mathfrak{sl}(2,\RR)$ in the case that $v \in \mathfrak{sl}(2,\RR)$.  In fact, the same elementary irreducibility argument shows that the orbit $\textrm{SL}_2(\RR)(v)$ for any $v \in \mathfrak{sl}(2,\CC)$ is at least $3$-dimensional (any $v$ can be conjugated to be diagonal, hence $\lambda H$ with $\lambda \in \CC$; then conjugations and linear combinations yield $\lambda R$ and $\lambda L$).

Furthermore, for $v \notin \mathfrak{sl}(2,\RR)$, we have $\textrm{SL}_2(\RR)(v) \cap \mathfrak{sl}(2,\RR) = \{0\}$.
By dimensional considerations, then, in this case
\[
        \operatorname{Span}_\RR( \Gamma(v), \mathfrak{sl}(2,\RR) ) = \mathfrak{sl}(2,\CC).
\]

Next, we show that the stabilizer of $\mathfrak{sl}(2,\RR)$ under the adjoint action of $\textrm{SL}_2(\CC)$ is exactly $\textrm{SL}_2(\RR) \cup i\textrm{SL}_2(\RR)$.  For, suppose $M$ is in the stabilizer.  Then, taking $\mathfrak{m} \in \mathfrak{sl}(2,\RR) \cap \textrm{SL}_2(\RR)$ (for example, an elliptic element of order $2$ with fixed points on $\PP^1(\RR)$), we find that it must stabilize the circle $M(\PP^1(\RR))$, which is only possible if $M(\PP^1(\RR)) = \PP^1(\RR)$.  Hence the stabilizer of $\mathfrak{sl}(2,\RR)$ is contained in the stabilizer of $\PP^1(\RR)$ under the $\textrm{SL}_2(\CC)$ action on $\PP^1(\CC)$.

We have assumed $T \not\in \textrm{SL}_2(\RR) \cup i\textrm{SL}_2(\RR)$.  Hence, by simplicity, $T(\mathfrak{sl}(2,\RR)) \cap \mathfrak{sl}(2,\RR) = \{0\}$.  In particular, we may take any ${w}\in \mathfrak{sl}(2,\QQ)$, and obtain $T(w)\not\in \mathfrak{sl}(2,\RR)$.  We may now conclude that for some appropriate choice of $\gamma_2, \gamma_3, \gamma_4 \in \Gamma$, we have:
\[
  Span_\RR\{H,R,L, w_2=\gamma_2(T(w)), w_3=\gamma_3(T(w)),w_4=\gamma_4(T(w)) \} =\mathfrak{sl}(2,\CC).
\]
Let $W$ denote this basis, where we have chosen $w \in \mathfrak{sl}(2,\ZZ)$.  

We may actually conclude that $W$ is a $\QQ$-basis of $\mathfrak{sl}(2,K_d)$, which is $3$ $K_d$-dimensional, hence $6$ $\QQ$-dimensional.  
We may extend scalars and find that $W$ is also a $\QQ_p$-basis of $\mathfrak{sl}(2,\QQ_p \otimes_{\QQ} K_d)$.

We have therefore found the desired $\QQ_p$-basis of $\textrm{SL}_2(\QQ_p \otimes_{\QQ} K_d)$, namely $W$.

Next we define $m_p$.
Since $\Gamma$ is Zariski dense, for each $p$ we can find a positive $m_p'$ such that for all $m\geq m_p'$, $\Gamma(p^m)$ is dense in $\textrm{SL}_2(\ZZ_p)(p^m)$.  For technical reasons, we take $m_p = m_p' + \iota_p$ where $\iota_p$ is the smallest non-negative integer so that
\begin{equation}
  \label{eqn:technical}
  p^{\iota_p} \mathfrak{sl}(2,\ZZ_p \otimes_{\ZZ} \mathcal{O}_K) \subset \operatorname{Span}_{\ZZ_p}(W).
\end{equation}
This $\iota_p$ is necessarily finite.
In the case that $W$ is a $\ZZ_p$-integral basis of $\mathfrak{sl}(2,\ZZ_p \otimes_{\ZZ} \mathcal{O}_K)$, then $\iota_p=0$, and the technical condition may be dropped in the sense that $m_p = m_p'$.  

Next, we prove the following \textbf{claim:}  For any $g\in \textrm{SL}_2(\mathcal{O}_K)(p^{m_p})/\textrm{SL}_2(\mathcal{O}_K)(p^{m})$ and any $m\geq m_{p}$, we 
may express $g$ as 
\aln{
  g\equiv L_1 (\gamma_2T L_2 T^{-1} \gamma_2^{-1}) (\gamma_3T L_3 T^{-1}\gamma_3^{-1}) (\gamma_4T L_4 T^{-1}\gamma_4^{-1}) \pmod{p^m}
}
for some $L_1,L_2,L_3,L_4\in\Gamma$.

This would imply $g\in A_4(p^m)$.  

We prove this by induction.  The base case $m=m_p$ is trivial.  Suppose for $m=k$ we can find $L_{1,k},L_{2,k},L_{3,k},L_{4,k} \in \Gamma$ such that 
\aln{
  g\equiv L_{1,k} (\gamma_2 T L_{2,k} T^{-1}\gamma_2^{-1}) (\gamma_3T L_{3,k} T^{-1}\gamma_3^{-1}) (\gamma_4T L_{4,k} T^{-1}\gamma_4^{-1}) \pmod{p^k}
}
Then 
$$
g = L_{1,k} (\gamma_2 T L_{2,k} T^{-1}\gamma_2^{-1}) (\gamma_3T L_{3,k} T^{-1}\gamma_3^{-1}) (\gamma_4T L_{4,k} T^{-1}\gamma_4^{-1}) + p^k u 
$$ 
for some $u\in\mathfrak{sl}(2,\ZZ_p \otimes_{\ZZ} \mathcal{O}_K)$.  Therefore, using the basis $W$, and the fact that $H$, $R$, $L$ give a $\ZZ_p$-integral basis for $\mathfrak{sl}(2,\ZZ_p)$, we can find $u_1 \in \mathfrak{sl}(2,\ZZ_p)$, and $t_2, t_3, t_4 \in \QQ_p$ so that
\[
  u = u_1 + t_2w_2 + t_3w_3 + t_4w_4.
\]
If $W$ forms a $\ZZ_p$-integral basis for $\mathfrak{sl}(2,\ZZ_p \otimes_{\ZZ} \mathcal{O}_K)$, then $t_i \in \ZZ_p$ for $i=2,3,4$.  Otherwise,
\[
  t_i \in p^{-\iota_p}\ZZ_p,
\]
(by \eqref{eqn:technical}).  This implies
\[
  t_ip^{k}w \in p^{k - \iota_p}\mathfrak{sl}(2,\ZZ_p).
\]

Since $\Gamma(p^{k-\iota_p})$ is dense in $\textrm{SL}_2(\ZZ_p)(p^{k-\iota_p})$, we can find $\beta_1,\beta_2,\beta_3,\beta_4\in\Gamma$ such that  $\beta_i \equiv I \pmod {p^{k-\iota_p}}$ and
\aln{&\beta_1\equiv I+p^{k}u_1 \pmod{p^{k+1}}\\
&\beta_2\equiv I+t_2p^{k}w \pmod{p^{k+1}}\\
&\beta_3\equiv I+t_3p^{k}w \pmod{p^{k+1}}\\
&\beta_4\equiv I+t_4p^{k}w \pmod{p^{k+1}}
 }
Then we set $L_{i,k+1}=L_{i,k}\beta_i$ for $i=1,2,3,4$.  
This is enough to prove the statement for $m = k+1$ (here, we rely on the fact that $k, k-\iota_p \ge 1$):
\begin{align*}
  L_{1,k+1}& (\gamma_2 T L_{2,k+1} T^{-1}\gamma_2^{-1}) (\gamma_3T L_{3,k+1} T^{-1}\gamma_3^{-1}) (\gamma_4T L_{4,k+1} T^{-1}\gamma_4^{-1}) \pmod{p^{k+1}} \\
  \equiv& L_{1,k} (\gamma_2 T L_{2,k} T^{-1}\gamma_2^{-1}) (\gamma_3T L_{3,k} T^{-1}\gamma_3^{-1}) (\gamma_4T L_{4,k} T^{-1}\gamma_4^{-1}) + p^ku \equiv g \pmod{p^{k+1}}.
\end{align*}
This completes the induction.
Therefore, we have $g\in A_4(p^m)$ for any $g\in \textrm{SL}_2(\mathcal{O}_K)(p^{m_p})/\textrm{SL}_2(\mathcal{O}_K)(p^{m})$ and any $m\geq m_{p}$.  

Now, $[\textrm{SL}_2(\mathcal{O}_K):\textrm{SL}_2(\mathcal{O}_K)(p^{m_p})]\leq p^{6m_p}$.  It must be that $A_{1}(p^m)$ contains something outside $\textrm{SL}_2(\mathcal{O}_K)(p^{m_p})/\textrm{SL}_2(\mathcal{O}_K)(p^m)$.  Therefore, $A_{4+1}(p^m)$ contains at least two cosets; $A_{4+2}(p^m)$ contains at least $3$ cosets and so forth.  So if we set $k_p=4+p^{6m_p}$, we have $A_{k_p}(p^{m})=\mathcal{A}^{'}/\mathcal{A}^{'}(p^m)$ for all $m \ge m_p$.
\end{proof}

\begin{proof}[Proof of Lemma \ref{lem:spectral-aprime-third}]

For each $p$ and $m$, there is a $k_p$ such that
\[
  \mathcal{A}'/\mathcal{A}'(p^m) = A_{k_{p,m}}(p^m).
\]
For $p \in \mathcal{S}$, this $k_{p,m}$ is uniform with respect to $p$ (Lemma \ref{lem:goodprimes}), while for any fixed $p \notin \mathcal{S}$, this $k_{p,m}$ is uniform for $m \ge m_p$ (Lemma \ref{lem:badprimes}).  As $\mathcal{S}$ is finite, the supremum of the $k_{p,m}$ is finite, say $k_0$.  Therefore,
\[
  \mathcal{A}'/\mathcal{A}'(p^m)
  = A_{k_0}(p^m)
\]
for any $p$, $m$.

We have assumed $\Gamma$ and therefore $\Gamma^{'}$ have a multiplicative structure.  In other words, for any $q=\prod_{i}p_i^{n_i}$, we have 
\begin{align*}
  {\Gamma}/{\Gamma(q)}&\cong \prod_i{\Gamma}/{\Gamma}(p_i^{n_i}),\\
{\Gamma}^{'}/{\Gamma}^{'}(q)&\cong \prod_i{\Gamma}^{'}/{\Gamma}^{'}(p_i^{n_i}).
\end{align*}

A direct corollary is that $\mathcal{A}^{'}$ also has a multiplicative structure
\begin{equation}\label{multiplicative}
  \mathcal{A}^{'}/\mathcal{A}^{'}(q)\cong \prod_i\mathcal{A}^{'}/\mathcal{A}^{'}(p_i^{n_i}),
\end{equation}
since $\mathcal{A}^{'}$ is generated by $\Gamma$ and $\Gamma^{'}$, and that $A_{k_0}$ has a multiplicative structure:
\[
  A_{k_0}(q)\cong \prod_i A_{k_0}(p_i^{n_i}).
\]

These isomorphisms are compatible so that the composition of isomorphisms
\[
  A_{k_0}(q)\cong \prod_i A_{k_0}(p_i^{n_i})
  = \prod_i \mathcal{A}'/\mathcal{A}'(p_i^{n_i})
  \cong 
  \mathcal{A}^{'}/\mathcal{A}^{'}(q)
\]
is the identity map.  Therefore,
\[
  A_{k_0}(q)=\mathcal{A}^{'}/\mathcal{A}^{'}(q)
\]
as desired.
\end{proof}

\begin{proof}[Proof of Proposition \ref{prop:spectral-aprime}]
  We verify the hypotheses of Lemma \ref{lem:varju2} for $G = \mathcal{A}'$, $S = S' \cup T S' T^{-1}$, where $S$ is a finite set of generators for $\Gamma$, $k=2k_0$, and $(G_1,G_2, \ldots, G_{2k_0}) = (\Gamma, \Gamma', \Gamma, \Gamma', \ldots, \Gamma')$.  
  The group $\Gamma$ has a spectral gap as a Zariski dense subgroup of $\textrm{SL}_2(\ZZ)$, by  \cite{BV12}; hence $\Gamma'$ does also.  The second hypothesis is immediate, and the third is verified by Lemma \ref{lem:spectral-aprime-third}.  Therefore $\mathcal{A}'$ has a combinatorial spectral gap.
\end{proof}

Next, we wish to pass from $\mathcal{A}'$ to $\mathcal{A}$.

\begin{prop}
  \label{prop:spectral-a}
  $\mathcal{A}$ has a combinatorial spectral gap.
\end{prop}

Before proving this, we note that our main spectral theorem follows immediately.

\begin{proof}[Proof of Theorem~\ref{mainspectralthm}]
  The theorem follows from the fact that $\mathcal{A}$ has a combinatorial spectral gap (Proposition \ref{prop:spectral-a}) and a version of \cite[Theorem 1.2]{BV12} for $\textrm{SL}_2(\mathcal O_K)$ giving equivalence of geometric and spectral gaps when the Hausdorff dimension of the limit set of the group is greater than $1$, which would follow from the arguments in \cite{BV12} modified as described in the paragraph preceding Theorem 2.1 in \cite{BV12}.
\end{proof}

To prove Proposition~\ref{prop:spectral-a}, we recall an equivalent condition for a combinatorial spectral gap to the one given at the beginning of this section.  Given a graph $G$ and subset $V$, write $\partial V$ for the set of edges joining $V$ to its complement in $G$.  Then define the \emph{expansion ratio} of $G$ to be
\[
  h_G := \min_{V \subset G, |V| \leq \frac{1}{2}|G| } \frac{ |\partial V| }{ |V| }.
\]
Let $\epsilon_G$ be the gap between the two biggest eigenvalues of the discrete Laplacian operator on $G$. It is known the expansion ratio of $G$ is related to $\epsilon_G$ by the inequalities \cite[Propositions 3.2.31, 3.2.33]{Ko13}:
\[
 \frac{h_G^2}{2M_G^2}\leq \epsilon_G\leq 2M_Gh_G,
\]
where $M_G$ is the maximum valence of vertices in $G$.
In particular, $h_{\rho_q G}$ is bounded away from $0$ uniformly with respect to $q$ if and only if $G$ satisfies a combinatorial spectral gap.

\begin{proof}[Proof of Proposition \ref{prop:spectral-a}]
We will demonstrate the existence of a positive constant $h$ such that for any positive integer $q$, and any $V\subset \mathcal{A}/\mathcal{A}(q)$ with $\vert V\vert\leq \frac{1}{2} \vert\mathcal{A}/\mathcal{A}(q)\vert$, we have 
\al{\label{0833}\vert\partial{V}\vert \geq h \vert V \vert.}

We use the corresponding property for $\mathcal{A}'$ (which has a combinatorial spectral gap by Proposition \ref{prop:spectral-aprime}).
Let $h_0$ be such that for any positive integer $q$ and any $V \subset \mathcal{A}'/\mathcal{A}'(q)$ with $| V| \leq \frac{1}{2} | \mathcal{A}'/\mathcal{A}'(q) |$, we have
\al{\label{0833b}\vert\partial{V}\vert \geq h_0 \vert V \vert.}

By the strong approximation property for $\mathcal{A}$ and $\mathcal{A}'$,  there is a universal $M$ such that the index $[\mathcal{A}/\mathcal{A}(q):\mathcal{A}^{'}/\mathcal{A}^{'}(q)]\leq M$.  
Let $S$ be a finite generating set for $\mathcal{A}$, which is symmetric under inverses (this exists since we assume $\mathcal{A}$ is geometrically finite, hence finitely generated).  
We say two cosets $a\mathcal{A}^{'}/\mathcal{A}^{'}(q) $ and $a'\mathcal{A}^{'}/\mathcal{A}^{'}(q)$ are connected if there exists some $s\in S$ such that $sa \mathcal{A}^{'}/\mathcal{A}^{'}(q)=a'\mathcal{A}^{'}/\mathcal{A}^{'}(q)$.  By the symmetry of ${S}$, this connectedness is an equivalence relation. 

Fix $q$.  Let $a_1\mathcal{A}^{'}/\mathcal{A}^{'}(q),\cdots, a_l \mathcal{A}^{'}/\mathcal{A}^{'}(q) $ be the cosets of $\mathcal{A}^{'}/\mathcal{A}^{'}(q)$ in $\mathcal{A}/\mathcal{A}(q)$, with $l\leq M$.  If $l = 1$, then \eqref{0833} follows trivially from \eqref{0833b}, with $h = h_0$, for this value of $q$.  Therefore, assume $l \ge 2$.

Let $V_i=V\cap a_i \mathcal{A}^{'}/\mathcal{A}^{'}(q)$ and define
$$\kappa=\max\left\{ \big\vert |V_i|-|V_j| \big\vert: a_i  \mathcal{A}^{'}/\mathcal{A}^{'}(q) \text{ and } a_j\mathcal{A}^{'}/\mathcal{A}^{'}(q) \text{ are connected}  \right \}. $$

Case 1:  $\kappa\leq \frac{|V|}{10l^2}$.  Then we have $$\max\{\vert V_i\vert\}-\min \{\vert V_i \vert\}\leq l\kappa\leq \frac{\vert V \vert}{10l}.$$  From this, one finds that for each $i$,  
\al{\label{0901}\frac{9}{10}\frac{|V|}{l}\leq |V_i| \leq \frac{11}{10}\frac{|V|}{l}.}

Case 1a:  If $|V|\leq \frac{10}{22}\vert \mathcal{A}/\mathcal{A}(q) \vert$, then each $|V_i|\leq \frac{1}{2} \vert \mathcal{A}^{'}/\mathcal{A}^{'}(q) \vert $.  Applying \eqref{0833b}, we have $\vert Eg(V_i, a_i \mathcal{A}^{'}/\mathcal{A}^{'}(q)-V_i)\vert \geq h_0 |V_i|$. Therefore, 
\al{\label{1054}|\partial V| \geq \sum_i \vert Eg(V_i, a_i \mathcal{A}^{'}/\mathcal{A}^{'}(q)-V_i)\vert \geq h_0 \vert V\vert . }

Case 1b:  If $|V|\geq \frac{10}{22}\vert \mathcal{A}/\mathcal{A}(q) \vert$, then from \eqref{0901} we have 
$$\frac{9}{22} \frac{|\mathcal{A}^{'}/\mathcal{A}^{'}(q)|}{l}\leq |V_i|\leq \frac{11}{22}\frac{|\mathcal{A}^{'}/\mathcal{A}^{'}(q)|}{l}.$$
And then it can be worked out that 
\al{\label{1055} |\partial V|\geq\frac{9h_0}{22}|\mathcal{A}/\mathcal{A}(q)|\geq \frac{9h_0}{11}|V| }

Case 2:  $\kappa\geq \frac{|V|}{10l^2}$.  There exists $s\in{S}$ such that $sa_i \mathcal{A}^{'}/\mathcal{A}^{'}(q)=a_j \mathcal{A}^{'}/\mathcal{A}^{'}(q)$ and  $\big\vert |V_i|-|V_j| \big\vert =\kappa$.  Since multiplication by $s$ is a bijection between $a_i \mathcal{A}^{'}/\mathcal{A}^{'}(q)$ and $a_j \mathcal{A}^{'}/\mathcal{A}^{'}(q)$, by the pigeon hole principle, multiplication by $s$ must map at least $\kappa$ elements from the bigger set, say $V_i$, to $a_j\mathcal{A}^{'}/\mathcal{A}^{'}(q)-V_j$, so we have at least 
\al{\label{1056}|\partial V| \geq \kappa = \frac{|V|}{10l^2}\geq \frac{|V|}{10M^2}}
Combining \eqref{1054},\eqref{1055} and \eqref{1056}, we find we can set $h=\min\{\frac{9h_0}{11}, \frac{1}{10M^2}\}$. 

\end{proof}

Lastly, we prove the statement of explicit strong approximation, with reference to the proof of the spectral gap just completed.

\begin{proof}[Proof of Theorem \ref{thm:explicit-strong}]
  First, we isolate the primes $p$ for which $\mathcal{A}/\mathcal{A}(p) \neq \textrm{SL}_2(\mathcal{O}_K)/\textrm{SL}_2(\mathcal{O}_K)(p)$.  Lemma \ref{lem:goodprimes} shows that $\mathcal{A}/\mathcal{A}(p) = \textrm{SL}_2(\mathcal{O}_K)/\textrm{SL}_2(\mathcal{O}_K)(p)$ for `good' primes, but in the course of the proof, we throw a variety of primes into $\mathcal{S}$, for which we do not prove this; they are to be dealt with as bad primes.  The first class of primes placed in $\mathcal{S}$ are those arising from the denominator of $\gamma$.  The denominator of $\gamma$ is always a divisor of the curvature of $T^{-1} \cdot \mathbb{P}^1(\RR)$ (\cite[Proposition 3.7]{StangeVis2}).  Therefore, by choice of $T$ (applying an element of $\Gamma$ to $T^{-1}$), we can avoid any prime not dividing all curvatures in $\mathcal{A} \cdot \mathbb{P}^1(\RR)$.  The second class of primes removed are those not coprime to $s$.  However, by choice of $s$, we can again avoid any odd prime not dividing the curvature $b$ (since $r = 1/b$, so that $s = (y \pm 1)/b$ for some integer $y$ \cite[Proposition 3.7]{StangeVis2}).  Other primes moved to $\mathcal{S}$ during the proof are those $p$ for which $\Gamma/\Gamma(p^m) \neq \textrm{SL}_2(\ZZ/p^m\ZZ)$ for some $m \ge 1$, and the special primes $p=2,3$.  Note that if $\Gamma/\Gamma(p) = \textrm{SL}_2(\ZZ/p\ZZ)$, then by Lemma~3 on page IV-23 of J-P.Serre in \cite{Serre} one automatically has that $\Gamma/\Gamma(p^m) = \textrm{SL}_2(\ZZ/p^m\ZZ)$.  Hence the statement in part 3(b) of Theorem~\ref{thm:explicit-strong} is equivalent to $\Gamma/\Gamma(p^m) \neq \textrm{SL}_2(\ZZ/p^m\ZZ)$ for some $m\geq 1$.  For all primes not contained in $\mathcal{S}$, the proof demonstrates that $\mathcal{A}'/\mathcal{A}'(p) = \textrm{SL}_2(\mathcal{O}_K)/\textrm{SL}_2(\mathcal{O}_K)(p)$, which implies $\mathcal{A}/\mathcal{A}(p) = \textrm{SL}_2(\mathcal{O}_K)/\textrm{SL}_2(\mathcal{O}_K)(p)$.  

Now let $P_{bad}$ be the product of the primes of $\mathcal{S}$ as above.  We obtain parts (1) and (2) immediately from the fact that $\mathcal{A}/\mathcal{A}(p) = \textrm{SL}_2(\mathcal{O}_K)/\textrm{SL}_2(\mathcal{O}_K)(p)$ for all other primes.  Part (4) is by definition.  

It remains to prove part (3).
Let $p|P_{\textrm{bad}}$.
  In the course of the proof of Lemma \ref{lem:badprimes}, we find that $\mathcal{A}'(p^{m_p})/\mathcal{A}'(p^k) = \textrm{SL}_2(\mathcal{O}_K)(p^{m_p})/\textrm{SL}_2(\mathcal{O}_K)(p^k)$, where by judicious choice of the basis $W$ in the proof, $m_p = m_p' + \iota_p$ where $\iota_p$ is as defined as the smallest non-negative integer so that
  \[
    p^{\iota_p} \mathfrak{sl}(2, \ZZ_p \otimes_\ZZ \mathcal{O}_K) \subset \operatorname{Span}_{\ZZ_p}( \mathcal{A}' \cdot \mathfrak{sl}(2,\ZZ_p)  ).
  \]
  
  However, if the goal is only that $\mathcal{A}(p^{m_p})/\mathcal{A}(p^k) = \textrm{SL}_2(\mathcal{O}_K)(p^{m_p})/\textrm{SL}_2(\mathcal{O}_K)(p^k)$, and not a spectral gap for $\mathcal{A}$, the proof of Lemma \ref{lem:badprimes} can be modified for $\mathcal{A}$ instead of $\mathcal{A}'$, as follows.  Using the same justification, we find that $\operatorname{Span}_{\QQ_p}( \mathcal{A} ( \mathfrak{sl}(2,\QQ_p) ) )$ is of dimension $6$, hence we can find a $\QQ_p$-basis of $\mathfrak{sl}(2, \QQ_p \otimes_\QQ K_d)$ of the form
  \[
  H,R,L, w_2=a_2(w), w_3=a_3(w),w_4=a_4(w),
  \]
  where $w \in \mathfrak{sl}(2,\ZZ_p )$.  We may choose $w$ and $a_i$ such that we have a $\ZZ_p$-basis for $\operatorname{Span}_{\ZZ_p}(\mathcal{A}( \mathfrak{sl}(2,\ZZ_p)))$.  Running the rest of the proof with $a_i$ in place of $\gamma_iT$, we no longer obtain a spectral gap but we obtain surjectivity with the stated $\iota_p$.
\end{proof}

\section{Example packings}\label{sec:examples}

As discussed in the introduction, Kontorovich and Nakamura present a collection of examples which satisfy the hypotheses of Theorem \ref{mainthm}.  Here we first present one explicit example appearing in Kontorovich and Nakamura satisfying the hypotheses of Theorem \ref{mainthm}.  Second, we verify that the hypotheses hold for the entire family of $K$-Apollonian packings.

\subsection{A cuboctohedral packing}\label{sec:cuboct}

The packing presented here is neither the Apollonian packing, nor any $K$-Apollonian packing, but it appears as an example of a super-integral polyhedral packing of Kontorovich and Nakamura \cite{KN}.  The packing is shown in Figure \ref{fig:cuboct}, where cuboctahedral symmetry is evident.

Define
\begin{align*}
    G_1 &= \left\langle
  c_1(z) = \overline{z} + \sqrt{-6}, \quad
  c_2(z) = \frac{\overline{z}}{ \frac{-\sqrt{-6}}{6}\overline{z} + 1}, \quad
  c_3(z) = \frac{(1+\sqrt{-6})\overline{z} - 3\sqrt{-6}}{ \frac{\sqrt{-6}}{3}\overline{z} + 1-\sqrt{-6}}  
  \right\rangle, \\
  G_2 &= \left\langle
  a_1(z) = - \overline{z}, \quad
  a_2(z) = - \overline{z}+6, \quad
  a_3(z) = \frac{\overline{z}}{\overline{z}-1}, \quad
  a_4(z) = \frac{5\overline{z}-12}{2\overline{z}-5}
  \right\rangle.
\end{align*}
Define $\mathcal{A}''$
as a group generated by the fourteen reflections:
\begin{gather*}
  \mathcal{A}'' = \left\langle
  a_1, \quad a_2, \quad a_3, \quad a_4, \quad
  c_1 a_3 c_1, \quad
  c_1 a_4 c_1, \quad
  c_2 a_4 c_2, \quad
  c_3 a_3 c_3, \quad
  c_1 c_3 a_3 c_3 c_1,  \quad
  c_3 a_1 c_3, \right. \\ \left.
  c_3 c_2 a_4 c_2 c_3, \quad
  c_2 c_3 a_3 c_3 c_2, \quad
  c_2 c_3 a_1 c_3 c_2, \quad
  c_1 c_2 c_3 a_1 c_3 c_2 c_1
  \right\rangle.
\end{gather*}
Note that
\[
  G_2 < \mathcal{A}'' < G_1 G_2 G_1^{-1} <
  M\left( \PGL_2(\ZZ[\sqrt{-6}])\rtimes \mathfrak{c} \right)M^{-1}, \quad
   M = \begin{pmatrix}
     \sqrt{-6} & 0 \\
     0 & 1
   \end{pmatrix},
\]

These 14 reflections correspond to the 14 faces of a cuboctahedron.  The fundamental domain therefore consists of hyperbolic upper half 3-space minus $14$ tangent geodesic hemispheres.  This shows that $\mathcal{A}''$ is of infinite covolume but geometrically finite.

Let $\mathcal{A} = \mathcal{A}'' \cap \PSL_2(\mathcal{O}_K)$.  The limit set of $\mathcal{A}''$ is shown in Figure \ref{fig:cuboct}.  Since $[\mathcal{A}'':\mathcal{A}]$ is finite, this limit set is the closure of a union of finitely many $K$-rational M\"obius images of a single circle orbit; in this case, of $\mathcal{A}C$ where $C = \widehat{\RR} + \sqrt{-6}$.  Therefore we aim to demonstrate that $\mathcal{A}$ is an infinite-covolume, geometrically finite, Zariski dense, familial Kleinian group.

The geometric finiteness and infinite covolume are inherited by $\mathcal{A}$ from $\mathcal{A}''$, as it is finite index.  By arguments exactly analogous to those in \cite[Theorems 9.3-9.4]{StangeKApp}, the limit sets of $\mathcal{A}''$ and $\mathcal{A}$ have Hausdorff dimension greater than $1$ and are Zariski dense.

It simply remains to prove the following lemma.

\begin{lemma}
  \label{lem:cong}
  The group $G_2$ is a congruence subgroup of $\PGL_2(\ZZ)$.
\end{lemma}

This implies $G_2 \cap \mathcal{A}$ is a congruence subgroup of $\PSL_2(\ZZ)$.

\begin{proof}[Proof of Lemma \ref{lem:cong}]
  We will show that $G_2$ contains the principal congruence subgroup $\Gamma(6)$.  Let
  \[
    L = \begin{pmatrix} 1 & 1 \\ 0 & 1 \end{pmatrix},
    \quad
    R = \begin{pmatrix} 1 & 0 \\ 1 & 1 \end{pmatrix}.
  \]
  These matrices generate $\PSL_2(\ZZ)$. 
  We will use the fact that $\Gamma(6)$ is the subgroup generated by the following elements \cite[p. 1357]{Hsu}:
    $L^6, R^6, L^2R^3L^{-2}R^{-3}, L^3R^2L^{-3}R^{-2}$.
  It suffices now to verify that
  \begin{align*}
    (a_1 a_2)^{-1} &= L^6, \\
    (a_1 a_3)^{-6} &= R^6,  \\
    (a_1 a_4)^{-1}(a_1a_3)^{4} &= L^2R^3L^{-2}R^{-3}, \\
    (a_1a_2)^{-1}a_1a_4(a_1a_3)^2 &= L^3R^2L^{-3}R^{-2}.
  \end{align*}
\end{proof}

Finally, we apply Theorem \ref{thm:explicit-strong}.  The potential bad primes are exactly $p=2,3$, since the curvatures of the packing are coprime and the congruence subgroup is of level $6$.  Letting $T_0 = a_3a_1 \in \mathcal{A}$ and $V = c_1a_3c_1a_1 \in \mathcal{A}$, and using the notation $H,L,R$ for the basis of $\mathfrak{sl}(2,\mathbb{Z})$ as in the proof of Lemma \ref{lem:badprimes}, one can compute the following elements of $\mathcal{A} \cdot \mathfrak{sl}(2,\mathbb{Z})$:
\[
  VHV^{-1}, \quad
  VLV^{-1}, \quad
  VRV^{-1}, \quad
  T_0VRV^{-1}T_0^{-1}, \quad
  T_0^{-1}VRV^{-1}T_0.
\]
These are enough to verify that $\iota_2 \leq 1$ and $\iota_3 = 0$.  Therefore the modulus of the congruence obstruction divides $12$.  As experimental confirmation, computing curvatures $\le 159$ appearing in the limit set packing (Figure \ref{fig:cuboct}), we find that the curvatures missing are exactly those $\equiv 7,9,11 \pmod{12}$ plus the exceptional absentees $13$ and $16$.

\subsection{$K$-Apollonian packings}\label{sec:kapp}

In this section we show that all $K$-Apollonian circle packings satisfy the hypotheses of Theorem \ref{mainthm}.  For an example of a $K$-Apollonian packing, see Figure \ref{fig:kapp}. 

The (strong) $K$-Apollonian groups defined in \cite{StangeKApp} are shown there to be finitely generated Zariski dense subgroups of $\PSL_2(\OK)$ containing congruence subgroups (either $\Pi(2)$ or $\Gamma^3$ in the notation of \cite[Section 10]{StangeKApp}).  They are of infinite covolume since they are of infinite index, and each packing contains the horizontal line $\widehat{\RR} + \sqrt{\Delta}/{2}$.  Therefore all the hypotheses of Theorem \ref{mainthm} are satisfied save geometric finiteness.  For that, it suffices to consider the remark following Theorem \ref{mainthm}.  

However, it may be useful to give an explicit description of a group associated to the packing.  
For each imaginary quadratic field $K$, we may use an adaptation of the weak $K$-Apollonian group given in \cite[Theorem 9.2]{StangeKApp}:
\[
  \mathcal{A}' =
  \left\langle
   S = \begin{pmatrix}
	0 & 1 \\
	-1 & 0
      \end{pmatrix}, \quad
      T = \begin{pmatrix}
	1 & 1 \\
	0 & 1
      \end{pmatrix}, \quad
   V = \begin{pmatrix}
    -1 & \tau \\ 0 & 1 
  \end{pmatrix}
  \right\rangle  < \PGL_2(\mathcal{O}_K)
\]
This group has the $K$-Apollonian packing as a limit set, and this limit set is of the form $\mathcal{A}' \widehat{\RR} = \mathcal{A}'( \widehat{\RR} + \sqrt{\Delta}/{2})$.  It has the following fundamental domain, given here as a list of the boundaries in $\widehat{\CC}$ of its geodesic walls:
\begin{align*}
  A&: \Re(z) = 0, \; \Im(z) \le \Im(\tau)/2 \quad \\
  B&: \Im(z) = \Im(\tau)/2, \; 0 \le \Re(z) \le 1\quad \\
  C&: \Re(z) = 1, \; \Im(z) \le \Im(\tau)/2 \quad \\
  D&: |z-1/2| = 1/2.
\end{align*}
It is straightforward to verify that this region satisfies the Poincar\'e Polyhedron Theorem and is therefore a fundamental domain for $\mathcal{A}$; it is therefore geometrically finite and of infinite covolume.  It has $\PSL_2(\ZZ) < \mathcal{A}$ (in the form of the first two generators above).  It is Zariski dense by the same arguments as in \cite[Section 10]{StangeKApp}.

In order to apply Theorem \ref{mainthm}, we need only pass to the finite-index subgroup $\mathcal{A} = \mathcal{A}' \cap \PSL_2(\mathcal{O}_K)$, by replacing $V$ with \[
  V_0 = VST^{-1}SV = \begin{pmatrix} \tau-1 & -\tau^2 \\ 1 & -\tau -1 \end{pmatrix}.
\]

The curvatures of the $K$-Apollonian circle packings are primitive integral (after scaling by $\sqrt{-\Delta}$).  Therefore, with this choice of group, Theorem \ref{thm:explicit-strong} tells us immediately that the only primes of bad reduction for strong approximation are $2$ and $3$.  In fact, it tells us more.  Write $L,R,H$ for the usual generators of $\mathfrak{sl}(2,\ZZ)$ as in the proof of Lemma \ref{lem:badprimes}.  Then following matrices are among $\mathcal{A} \cdot \mathfrak{sl}(2,\ZZ)$:
\[
 V_0RV_0^{-1}, \quad
 SV_0RV_0^{-1}S, \quad
 TV_0RV_0^{-1}T^{-1}, \quad
 STV_0RV_0^{-1}T^{-1}S, \quad
 TSV_0RV_0^{-1}ST^{-1}.
\]
Using these suffices to verify that for $\Delta \equiv 0 \pmod 4$, $\iota_2 \le 2$ and $\iota_3 = 0$; while for $\Delta \not\equiv 0 \pmod 4$, $\iota_2 \le 1$ and $\iota_3 = 0$.  Then Theorem \ref{thm:explicit-strong} tells us that the modulus of the congruence obstruction for $K$-Apollonian packings is a divisor of $24$ in all cases, and in fact a divisor of $12$ if $\Delta \not\equiv 0 \pmod 4$.  This is in accordance with \cite[Conjecture 1.4]{StangeKApp}, which gives an explicit prediction for the modulus for the congruence obstruction.

\newpage
\section{Notations}\label{notations}
\begin{longtable}{lll}
\caption[Table of Notation used in Sections~\ref{integralitysec} through \ref{sectionminor}]{Table of Notation used in Sections~\ref{integralitysec} through \ref{sectionminor}}\label{notationtable}\\
$\mathcal A$ & a familial Kleinian group in $\PSL_2(K)$, &\\
& \quad\quad assumed from Section 3 onwards to be in $\PSL_2(\mathbb{Z}[\sqrt{-d}])$ & \\
$\mathcal A(q)$& elements of $\mathcal A$ congruent to identity modulo $q$& \\
$\beta$ & $\theta-\frac{r}{q}$; $|\beta|\leq\frac{K_0}{N}$&\\ 
$B_q(n)$& $\frac{1}{[\Aa:\Aa(q)]}\sum_{\gamma_0\in \Aa/\Aa(q)}c_q\left(\ff_{M\gamma_0}(Lx+1,Ly)-n\right)$&\\
$C$& a circle tangent to the real line& \\
$\widehat{\CC}$ & the extended complex plane&\\
$c_q(n)$&$\sideset{}'\sum_{r(q)}e\left(\frac{rn}{q}\right)$&\\
$\delta$& hausdorff dimension of limit set of $\mathcal A$&\\
$\Delta$& discriminant of $K$&\\
 $\epsilon(n)$& $0$ if $n\equiv 1 \pmod{4}$ and $1$ if $n\equiv 3 \pmod{4}$&\\
$\mathfrak{d}_\gamma$& $2\frac{\Im (\overline{C_{M\cdot \gamma}} D_{M\cdot \gamma})}{\sqrt{-\Delta}}$ (i.e., the shift of the shifted form)&\\ 
$e(x)$& $e^{2\pi i x}$&\\
$e_q(x)$& $e^{\frac{2\pi i x}{q}}$& \\
$\epsilon$& small positive number&\\
$\eta$& small positive number depending on $M$, $\mathcal A$, and $C$&\\
$\enn$& minor arcs (error term) defined in \eqref{errorterm}&\\
$\mathcal{E}_N^U(n)$ & modification of error term defined in \eqref{errortermu}&\\
$f\ll g$ & $f=O(g)$& \\
$f\asymp g$ & $f\ll g$ and $g\ll f$&\\
$\FF, \FF_T$&growing region in $\mathcal A$ defined in \eqref{FF}&\\
$\mathfrak{\widehat{f}}_{M\gamma}(a,c)$ & shifted binary form $\sqrt{-\Delta}\left\vert C_{M \cdot\gamma} a+D_{M\gamma} c\right\vert^2+2\Im (\overline{C_{M\gamma} }D_{M\gamma})$&\\
$F_1,F_2,F_3$& defined in \eqref{F1} and \eqref{F2F3}&\\
$\gamma$&element of $\mathcal A$&\\
$h$&$(\fd_\gamma,\fd_{\gamma'})$&\\
$\mathcal{I}_1$&$\sum_{q<Q_0}\sideset{}'\sum_{r(q)}\int_{r/q-1/qJ}^{r/q+1/qJ}|(1-\Tf(\theta))\hrnu(\theta)|^2d\theta$&\\
$\mathcal{I}_2$&$\sum_{Q_0\leq q<X}\sideset{}'\sum_{r(q)}\int_{r/q-1/qJ}^{r/q+1/qJ}|(1-\Tf(\theta))\hrnu(\theta)|^2d\theta$&\\
$\mathcal{I}_3$&$\sum_{X\leq q\leq J}\sideset{}'\sum_{r(q)}\int_{r/q-1/qJ}^{r/q+1/qJ}|(1-\Tf(\theta))\hrnu(\theta)|^2d\theta$&\\
$\mathcal{I}_Q$&$\sum_{Q<q\leq 2Q}\int_{-1/qJ}^{1/qJ}\sideset{}'\sum_{r(q)}\left|\hrnu\left(\frac{r}{q}+\beta\right)\right|^2d\beta$&\\
$\mathcal{J}_\gamma(\beta; q,u,\xi,\zeta)$&$\int_{-\infty}^{\infty} \int_{-\infty}^{\infty}\psi(x)\psi(y)e\left(\ff_{M\gamma}\left(Xx,Xy\right)\beta-\frac{X\xi}{qu\BB}x-\frac{X\zeta}{qu\BB} y\right)e\left(\frac{u^*\xi}{\BB q}\right)dxdy$&\\
$\Im$ &imaginary part &\\
$J$ & $T^2X$, depth of approximation; see \eqref{eqn:setJ}\\
$K$ & $\QQ(\sqrt{-d})$&\\
$\kappa(\cdot)$& curvature of circle $\cdot$ &\\
$\mathcal{K}$ & the set of curvatures in integral packing&\\
$\mathcal{K}_a$& $\{n\in\ZZ\;\;\vert\;\; \forall q\in\ZZ, \exists k\in\mathcal{K}, \text{such that }n\equiv k(\textrm{mod }q)\}$&\\
$\mathcal{K}_a(N)$& $\mathcal{K}_a\cap [0,N]$&\\
$K_0$& small power of $N$ given in \eqref{0737}, depending on spectral gap &\\
$L$& the level of the congruence subgroup of $\textrm{PSL}_2(\mathbb Z)$ contained in $\mathcal A$&\\
$L_0$ & positive integer such that $\mathcal{K}_a$ is union of some congruence classes mod $L_0$& \\
$\lambda_\gamma(\beta,X,u,s,t)$&$\sum_{x,y\in\ZZ}\psi\left(\frac{\BB ux +uu^{*}}{X}\right)\psi\left(\frac{\BB uy}{X}\right)e\left(\ff_{M\gamma}\left(\BB ux+uu^{*},{\BB uy}\right)\beta-sx-ty\right)$&\\
$M$ & Moebius transformation in $\PSL_2(K)$&\\
$\mathcal{M}_N(n)$& major arcs (main term) defined in \eqref{mainterm}&\\
$\mathcal{M}_N^U(n)$& modification of main term defined in \eqref{maintermu}&\\
$\mathfrak{M}(n)$&$\frac{K_0}{N}\sum_{\gamma\in\FF}\hat{\mathfrak{t}}\left(\frac{K_0}{N}(\ff_{M\gamma}(Lx+1,Ly)-n)\right)$&\\
$N$& a growing parameter; see Section \ref{sec:setup}  &\\
$\mathcal O_K$& ring of integers in $K$&\\
$p$, $p_i$ & prime numbers & \\
$p^j|| n$& $p^j| n$ and $p^{j+1}\nmid n$&\\
$\psi$ & smooth function supported on  $[1,2]$, with $\psi\geq 0$ and $\int_\RR \psi (x)dx=1$&\\
$P_{\textrm{bad}}$ &product of bad primes&\\
$q$ & positive integer&\\
$Q_0$& small power of $N$ given in \eqref{0737}, depending on spectral gap &\\
$\widehat{\RR}$ & the extended real line, manifest as the horizontal axis in $\widehat{\CC}$&\\
$\Re$ &real part &\\
$\mathcal{R}_N(n)$&representation number of $n$ in packing defined in (\ref{RNdef})&\\
$\rnu(n)$& modification of $\mathcal{R}_N(n)$ defined in (\ref{0922})&\\
$\hrnu(\frac{r}{q}+\beta)$&$\sum_{u<U}\mu(u)\sum_{\gamma\in\FF_T}\mathcal{R}_{u,\gamma}(\frac{r}{q}+\beta)$&\\
$\frac{r}{q}$& rational number of small denominator&\\
$\sum_{r(q)}^{'}$ & sum over all $0 \le r < q$ where $(r,q)=1$& \\ 
$\fs_{Q_0}(n)$ &$\sum_{q<Q_0} \frac{1}{[\Aa:\Aa(q)]}\sum_{\gamma_0\in \Aa/\Aa(q)}c_q\left(\ff_{M\cdot \gamma_0}(Lx+1,Ly)-n\right)$&\\
$\fs(n)$&$\sum_{q=1}^{\infty}B_q(n)$&\\
${S}(q,A,B,C,D,E)$& $\sum_{x,y(q)}e(Ax^2+Bxy+Cy^2+Dx+Ey)$&\\
$\mathcal{S}_\gamma(q,u,r,\xi,\zeta)$&$\frac{1}{q^2}\sum_{x_0,y_0(q)}e_q\left(r\ff_{M\gamma}\left(\BB ux_0+uu^{*},\BB uy_0\right)+x_0\xi+y_0\zeta\right)$&\\
$\mathcal{S}(q,u,\gamma,\xi,\zeta,\gamma',\xi',\zeta')$&$\sideset{}'\sum_{r(q)}\mathcal{S}_{\gamma}(q,u,r,\xi,\zeta)\overline{\mathcal{S}_{\gamma'}(q,u,r,\xi',\zeta')}$&\\
$T$& $N^{1/200}$; see Section \ref{sec:setup}& \\
$T_1,T_2$&growing parameters used to define $\FF_T$ in (\ref{FF})&\\
$\tf(x)$&$\text{max}\{0, 1 - |x|\}$, a hat function used in definition of major arcs&\\
$\mathfrak{T}$ & spike function in (\ref{spike}) used to define major arcs&\\
$\tau_q(r)$&$\frac{1}{[\Aa:\Aa(q)]}\sum_{\gamma_0\in\Aa/\Aa(q)}\bd{1}\{\ff_{M\gamma_0}(Lx+1,Ly)=r\}$&\\
$\theta$& number in $[0,1]$&\\
$\Theta$& max of $\Theta_1,\Theta_2$ in Lemma \ref{bk2} and Lemma \ref{bk3} in context of $\mathcal A$&\\
$U$& small power of $N$; see Section \ref{sec:setup}&\\
$u$& positive number less than $U$&\\
$u^{*}$ &integer such that $uu^{*}\equiv 1(\BB)$&\\
$X$& $N^{99/200}$; see Section \ref{sec:setup}&\\
$\# \cdot $& cardinality of finite set $\cdot$&\\
$\mathbf{1}\left\{ \cdot \right\}$ &characteristic function \\
$|| \cdot ||$ &Frobenius norm \\
$( \cdot , \cdot )$ &$\gcd( \cdot, \cdot )$ \\
\end{longtable}

\bibliographystyle{plain}
\bibliography{LGC}

\end{document}